\documentclass[a4paper,reqno,11pt]{amsart}
\usepackage[margin=1.3in]{geometry}
\usepackage[utf8]{inputenc}	
\usepackage{indentfirst} 
\usepackage{amsmath,amsfonts,amssymb,,mathrsfs}
\usepackage{bbm}
\usepackage[colorlinks=true]{hyperref}
\usepackage{enumerate}
\usepackage{graphicx}
\usepackage{xcolor}
\usepackage{enumitem}
\usepackage{stackengine}

\newcommand{\restr}{%
  \,\raisebox{-.127ex}{\reflectbox{\rotatebox[origin=br]{-90}{$\lnot$}}}\,%
}

\allowdisplaybreaks

\theoremstyle{definition}
\newtheorem{definition}{Definition}[section]
\theoremstyle{remark}
\newtheorem{remark}[definition]{Remark}
\theoremstyle{theorem}
\newtheorem{theorem}[definition]{Theorem}
\theoremstyle{theorem}
\newtheorem{lemma}[definition]{Lemma}
\theoremstyle{remark}

\theoremstyle{theorem}
\newtheorem{corollary}[definition]{Corollary}
\theoremstyle{theorem}
\newtheorem{proposition}[definition]{Proposition}
\theoremstyle{theorem}

\numberwithin{equation}{section}
\allowdisplaybreaks

\def\Xint#1{\mathchoice
    {\XXint\displaystyle\textstyle{#1}}%
    {\XXint\textstyle\scriptstyle{#1}}%
    {\XXint\scriptstyle\scriptscriptstyle{#1}}%
    {\XXint\scriptscriptstyle\scriptscriptstyle{#1}}%
    \!\int}
\def\XXint#1#2#3{{\setbox0=\hbox{$#1{#2#3}{\int}$}
      \vcenter{\hbox{$#2#3$}}\kern-.5\wd0}}

\def\mint{\Xint-}

\definecolor{ao}{rgb}{0.0, 0.5, 0.0}

\DeclareMathOperator*{\aplim}{ap-\lim}
\DeclareMathOperator*{\aplims}{ap-\limsup}
\DeclareMathOperator*{\aplimi}{ap-\liminf}

\newcommand{\Om}{\Omega}
\newcommand{\R}{\mathbb{R}}

\newcommand{\HH}{\mathcal{H}}
\newcommand{\di}{\mathrm{d}}

\newcommand{\UUU}{\color{black}}
\newcommand{\EEE}{\color{black}}

\def\namedlabel#1#2{\begingroup
    #2%
    \def\@currentlabel{#2}%
    \phantomsection\label{#1}\endgroup
}

\makeatletter
\@namedef{subjclassname@2020}{%
  \textup{2020} Mathematics Subject Classification}
\makeatother

\title[Jump set slicing and applications]{A general criterion for jump set slicing \\ and applications}
\author[S. Almi]{Stefano Almi}
\address[Stefano Almi]{Institute of Analysis and Scientific Computing, TU Wien, Wiedner-Hauptstrasse 8-10, 1040 Vienna, Austria \& Universit\'a di Napoli Federico II, Dipartimento di Matematica e Applicazioni R. Caccioppoli, via Cintia, Monte S. Angelo, 80126 Naples, Italy. }
\email{stefano.almi@unina.it}

\author[E. Tasso]{Emanuele Tasso}
\address[Emanuele Tasso]{Institute of Analysis and Scientific Computing, TU Wien, Wiedner-Hauptstrasse 8-10, 1040 Vienna, Austria}
\email{emanuele.tasso@tuwien.ac.at}

 \subjclass[2020]{49Q20, 
 			   26B30. 
			   	   }
\keywords{One-dimensional slicing, Jump set, Integralgeometric measures, Bounded deformation, Riemannian manifolds}

\begin{document}

\maketitle

\begin{abstract}
In this paper a novel criterion for the slicing of the jump set of a function is provided, which bypasses the codimension-one and the parallelogram law techniques developed in $BD$-spaces. The approach builds upon a recent rectifiability result of integralgeometric measures and is further applied to the study of the structure of the jump set of functions with generalized bounded deformation in a Riemannian setting.
\end{abstract}

\section{Introduction}
\label{s:intro}
In the setting of free discontinuity problem and free boundary problems, a well established technique in investigating the structure of the gradient of functions $u$ is the so called \emph{slicing}. Such a technique allows to reconstruct the properties of the gradient by studying the slices of $u$ which are lower-dimensional and simpler to handle. One of the most challenging steps consists in establishing a precise relation between the one-dimensional slices of the jump set $J_u$ of $u$ and the jump of its one-dimensional slices. More precisely, given a direction $\xi \in \mathbb{S}^{n-1}$ and $y \in \xi^\bot$, one is interested in showing that the restriction of $J_u$ along the line $\ell:= \{y+t\xi : t \in \mathbb{R}\}$ coincides with the jump set of $u |_\ell$. Such a result can be obtained in a $BV$-setting via Coarea Formula (see, e.g.,~\cite[Theorem~3.108]{afp}). The $BD$-case, as well as the more general $BV^{\mathcal{A}}$-case, required a new approach mainly relying 
 on the following two facts: the rectifiability of $\Theta_u$, namely, the set of strictly positive $(n-1)$-dimensional upper density of $|Eu|$ or $|\mathcal{A}u|$, and a combination of codimension-one slicing analyisis together with the so called parallelogramm law (cf.~\cite[Formula (5.4)]{MR1480240} for the $BD$-space and \cite[Formula (30)]{arr} for the $BV^{\mathcal{A}}$-space). The proof of the rectifiability of $\Theta_u$ is originally due to Kohn \cite{Kohn} and relies on the celebrated Federer's Structure Theorem, while the relation between the jump points of the codimension-one slices and the set $\Theta_u$ has been first pointed out by Ambrosio-Coscia-Dal Maso in~\cite{MR1480240}.  

The aim of this paper is to provide a unifying slicing criterion for the jump set $J_u$ of a measurable function $u$, which also presents applications in a non-Euclidean framework. With this motivation we will not only deal with slices along straight lines, as done in \cite{MR1480240, afp, arr}, but we will rather work with one-dimensional slices along solutions~$\gamma$ of a second order ODE driven by a sufficiently smooth field $F \colon \R^{n} \times \R^{n} \to \R^{n}$ which is $2$-homogeneous in the second variable (cf.~\eqref{e:quadratic}). \EEE In order to better explain the difficulties due to our general approach let us introduce some notation. Fix $\Omega$ an open set in $\mathbb{R}^n$ and $u \colon \Omega \to \mathbb{R}^m$ a measurable function. The first novelty of our work is the notion of \emph{families of curvilinear projections} $(P_\xi)_{\xi \in \mathbb{S}^{n-1}}$ where $P_\xi \colon \Omega \to \xi^\bot$ (cf.~Definition~\ref{d:CP}). The main feature of
 such a family is that it satisfies a transversality property in the sense of Definition~\ref{d:transversal} and it admits a parametrization $\varphi \colon \{(y+t\xi,\xi) \in \mathbb{R}^n \times \mathbb{S}^{n-1} \ : \ (y,t) \in [\xi^\bot \cap \mathrm{B}_{\rho}(0)] \times (-\tau,\tau) \} \to \mathbb{R}^n$ (cf.~Definition~\ref{d:param}). In particular, the level sets~$P_\xi^{-1}(y)$ are the images of $t \mapsto \varphi_\xi(y + t\xi):= \varphi((y+t\xi,\xi))$ \UUU solutions \EEE of
\begin{equation}
\label{e:ODE}
\ddot{\gamma} = F(\gamma,\dot{\gamma}).
\end{equation}
For $E \in \Omega$, $\xi \in \mathbb{S}^{n-1}$, and $y \in \xi^\bot$, we define the one-dimensional slices 
\begin{align}
\label{e:int1}
 E^{\xi}_{y}  &:= \{ t \in \R: \gamma(t) \in E\}\,,\\
\label{e:int2}
\hat{u}^{\xi}_{y} (t)  &:= u(\gamma(t)) \cdot g(\gamma(t),\dot{\gamma}(t)) \ \ \text{for $t \in \Om^{\xi}_{y}$}\,,
\end{align}
where $\gamma(t)=\varphi_\xi(y+t\xi)$ and $g \colon \Omega \times \mathbb{R}^n \to \mathbb{R}^m$ is a given continuous map. For later use we further denote $\dot{\varphi}_\xi(y+t\xi)$ the derivative in time of $t \mapsto \varphi_\xi(y + t\xi)$ and the velocity field $\xi_\varphi \colon \Omega \to \mathbb{R}^n$ as 
\[
\xi_\varphi(x):= \dot{\varphi}_\xi(P_\xi(x)+t_x \xi),
\]
where $t_x$ is the unique time satisfying $x= \varphi_\xi(P_\xi(x)+t_x \xi)$. We notice that with the choice
\begin{align}
    \label{e:int1100}
 F(x,v) & := - \bigg( \sum_{i,j=1}^n\Gamma^1_{ij}(x)v_iv_j,\dotsc,\sum_{i,j=1}^n\Gamma^n_{ij}(x)v_iv_j \bigg)  \qquad (x,v) \in \Omega \times \mathbb{R}^n\,,
\end{align}
the solutions $\gamma$ of \eqref{e:ODE} are the geodesics in coordinates of a Riemannian manifold with Christoffel symbols equal to $\Gamma^\ell_{ij}$.

The non-linear nature of our setting leads to a crucial point, which is the lack of symmetry required to exploit the parallelogram law technique and the possibility to perform codimension-one slicing on which the $BD$ and $BV^{\mathcal{A}}$ theories hinge. In order to overcome this difficulty, we develop a new approach which circumvents the rectifiability of $\Theta_u$ as well as Ambrosio-Coscia-Dal Maso's argument. At the best of our knowledge, this is the first time in which one dimensional slices involving projections on non-constant vector fields are considered.

The two fundamental conditions,  \UUU upon \EEE which our slicing criterion is based, read as follows:
\begin{enumerate}[label=(\roman*)]
\item the finiteness of a suitable measure $\mathscr{I}_{u,p}$ related to the jump points of $\hat{u}^\xi_y$;
\item a control on the size of the set of points $x \in \Omega$ around which certain \emph{oscillations} of $u$ do not vanish.
\end{enumerate}
For $1 \leq p \leq \infty$ the measure $\mathscr{I}_{u,p}$ is defined by a Carath\'eodory's contruction \UUU \cite[Section~2.10.1]{fed}. \EEE Namely, we set
\begin{equation*}
   \eta_\xi(B) := \int_{\xi^\bot} \sum_{ t \in B^\xi_y } \big( |[\hat{u}^\xi_y(t)]| \wedge 1 \big) \, \di \mathcal{H}^{n-1}(y) \qquad  B \in \mathcal{B}(\Omega),
   \end{equation*}
   where $[\hat{u}^\xi_y(t)]$ denotes the jump of $\hat{u}^\xi_y$ at $t$, $\mathcal{B}(\Omega)$ denotes the Borel $\sigma$-algebra of $\Omega$, and we use as a Gauge function $\zeta_p \colon \mathcal{B}(\Omega) \to [0,\infty]$ the $L^p$-norm of the map $\xi \mapsto \eta_\xi(B)$ with respect to the $(n-1)$-dimensional Hausdorff measure $\mathcal{H}^{n-1}$. As for (ii), for $\rho>0$ we define the set $\text{Osc}_u (\rho)$ (see also Definitions~\ref{d:oscillation} and~\ref{d:oscillation2}) as those points $x \in \Omega$ such that
\[
   \limsup_{r \searrow 0} \int_{\mathbb{S}^{n-1}} \inf_{\{\theta: {\rm Lip}(\theta) \leq 1\}} \bigg(\int_{-\rho/4}^{\rho/4} |\mathring{u}^\xi_x(rt) -\theta (t) |t^{n-1} \,  \di t \bigg) \di  \mathcal{H}^{n-1}(\xi) >0\,,
\]
where $\mathring{u}^\xi_x$ is given by \eqref{e:int2} with $\gamma$ being the unique solution of 
\begin{equation}
\label{e:int1000}
\begin{cases}
\ddot{\gamma} =F(\gamma,\dot{\gamma}), &\\
\gamma(0)=x,&\\ 
\dot{\gamma}(0)=\xi\,.&
\end{cases}
\end{equation}
We will then ask for the set $\text{Osc}_u (\rho)$ to be $\sigma$-finite with respect to the curvilinear version of Farvard's integralgeometric measure $\tilde{\mathcal{I}}^{n-1}$ (see Definition~\ref{d:Favard} for the precise definition).

With the above notation at hands we can state the main contribution of this paper.
\begin{theorem}
\label{t:int1}
Let $u \colon \Omega \to \mathbb{R}^m$ be measurable, let $F\in C^{\infty}(\mathbb{R}^n \times \mathbb{R}^n;\mathbb{R}^n)$ be $2$-homogeneous in the second variable, let $g \in C(\Omega \times \mathbb{R}^n;\mathbb{R}^m)$, and let $(P_\xi)_{\xi \in \mathbb{S}^{n-1}}$ be a family of curvilinear projections on $\Omega$. Suppose that
\begin{enumerate}
    \item  There exists $p \in (1, +\infty]$ such that $\mathscr{I}_{u,p}(\Omega) < +\infty$;   
    \item There exists $\rho>0$ such that $\emph{Osc}_u (\rho)$ is $\sigma$-finite w.r.t. to $\tilde{\mathcal{I}}^{n-1}$.
    \end{enumerate}
    Then, there exists a countably $(n-1)$-rectifiable set $R  \subseteq  \Omega$ such that
    \begin{equation}
        \label{e:int6.1}
         J_{\hat{u}^\xi_y}  \subseteq  R^\xi_y \qquad  \text{ for $\mathcal{H}^{n-1}$-a.e.~$\xi \in \mathbb{S}^{n-1}$, $\mathcal{H}^{n-1}$-a.e.~$y \in \xi^\bot$}.
    \end{equation}
    In particular, for every $x \in R$
    \begin{equation}
\label{e:mainslicepro9}
  \HH^{n-1} \big( \big\{ \xi \in \mathbb{S}^{n-1}: \, t_x^{\xi}  \in J_{\hat u^\xi_{P_\xi(x)}} \big\} \big) >0   
 \end{equation}
     \end{theorem}
The proof of Theorem \ref{t:int1} is obtained by exploiting a fundamental property of the measure $\mathscr{I}_{u,p}$. Namely, condition (1) implies that $\mathscr{I}_{u,p}$ is concentrated on points $x \in \Omega$ such that \eqref{e:mainslicepro9} holds true (cf.~Proposition \ref{p:keyprop}). The combination of such property with condition (2) allows us to apply the rectifiability criterion given in Theorem \ref{t:rectheorem} (see also \cite[Theorem 1.5]{Tas22}) and infer the $(n-1)$-rectifiability of the measure $\mathscr{I}_{u,1}$. In particular, this provides a countably $(n-1)$-rectifiable set $R$ such that \eqref{e:int6.1} holds true. 
 
 The inclusion in \eqref{e:int6.1} can be improved to an equality when $R$ is replaced by the jump set $J_\mathfrak{u}$ of $\mathfrak{u} \colon \Omega \to \mathbb{R}^m$, where $\mathfrak{u}(x):=\pi(x,u(x))$ and $\pi(x,\cdot)$ is the projection of~$\mathbb{R}^m$ onto the vector space generated by the image of $g(x,\cdot)$. The precise statement reads as follows.
 \begin{corollary}
     \label{c:int2}
     Let $u \colon \Omega \to \mathbb{R}^m$ be measurable, let $F\in C^{\infty}(\mathbb{R}^n \times \mathbb{R}^n;\mathbb{R}^n)$ be $2$-homogeneous in the second variable, let $g \in C(\Omega \times \mathbb{R}^n;\mathbb{R}^m)$ satisfy \eqref{G2}, and let $(P_\xi)_{\xi \in \mathbb{S}^{n-1}}$ be a family of curvilinear projections on $\Omega$. Suppose that 
     \begin{enumerate}
    \item There exists $p \in (1, +\infty]$ such that $\mathscr{I}_{u,p}(\Omega) < +\infty$\EEE;   
    \item There exists $\rho>0$ such that $\emph{Osc}_u(\rho)$ is $\sigma$-finite w.r.t. to $\tilde{\mathcal{I}}^{n-1}$;
    \item There exists a diffeomorphism $\tau \colon \mathbb{R} \to (-1,1)$ such that, setting $u_\xi(x):= u(x) \cdot g(x,\xi_\varphi(x))$ for $(x,\xi) \in \Om \times \mathbb{S}^{n-1}$, we have 
    \[
    D_\xi (\tau (u_\xi \circ \varphi_\xi )) \in \mathcal{M}^+_b(\varphi_\xi^{-1}(\Omega)) \qquad \text{for $\mathcal{H}^{n-1}$-a.e. $\xi \in \mathbb{S}^{n-1}$}.
    \]
    \end{enumerate}
 Then, it holds 
\begin{equation}
    \label{e:int6}
    J_{\hat{u}^\xi_y}=(J_{\mathfrak{u}})^\xi_y \qquad \text{ for $\mathcal{H}^{n-1}$-a.e.~$\xi \in \mathbb{S}^{n-1}$, $\mathcal{H}^{n-1}$-a.e.~$y \in \xi^\bot$}.
\end{equation}
\end{corollary}
 Corollary \ref{c:int2} is obtained as a consequence of \eqref{e:int6} together with the following two facts: \emph{(i)} by an adaptation of the argument in the proof of \cite[Theorem 5.2]{dal}, hypothesis~(3) of Corollary \ref{c:int2} leads to a precise relation between the traces of $u_\xi$ and the traces of its one dimensional slices $\hat{u}^\xi_y$ for $y \in \xi^\bot$ on any $(n-1)$-rectifiable set $R  \subseteq  \Omega$ (cf.~\eqref{e:corollary-relje}), and~\emph{(ii)} the jump set of any measurable function $u \colon \mathbb{R}^n \to \mathbb{R}^m$ is countably $(n-1)$-rectifiable, as it was remarkably pointed out in~\cite{del}. Combining all this ingredients in the proper way we end up with Corollary~\ref{c:int2}. We refer to Propositions~\ref{c:relje} and~\ref{p:prodmeas}. 
 
 Let us comment on the applicability of Theorem~\ref{t:int1} and Corollary~\ref{c:int2}. Conditions~(1) and~(3) \EEE  of Corollary~\ref{c:int2} are usually satisfied in spaces admitting a slicing representation. For instance, by choosing~$g$ equal to the orthogonal projection on the second component, the $BD$-theory can be recovered and property~(1) is valid with $p=\infty$. This is guaranteed from inequality $\eta_\xi(B) \leq |Eu|(B)$ valid for every $\xi \in \mathbb{S}^{n-1}$ and for every Borel set~$B$. Moreover,~(3) follows  by a control on $|D_\xi u_\xi|$. A similar argument leads to the validity of~(1) and~(3) in the $BV^{\mathcal{A}}$-framework with a different choice of $g$ dictated by the notion of spectral pairs (see, e.g.,~\cite{arr}). 
 
 The most delicate assumption is item (2) in Corollary \ref{c:int2}. In this regard we show in Section~\ref{s:applications}  that the $\sigma$-finiteness of $\text{Osc}_u (\rho) $ follows from the $\sigma$-finiteness of the set of points $x \in \Omega$ satisfying 
 \begin{align}
\label{e:voscillation20000}
       \limsup_{r \searrow 0 } \bigg(\inf_{a \in \Xi (\mathrm{B}_1(0))} \int_{\mathrm{B}_{1}(0)} |u_{r,x}-a| \wedge 1 \, \di z \bigg) >0,
\end{align} 
where $u_{r,x} \colon \mathrm{B}_1(0) \to \mathbb{R}^n$ is defined as $u_{r,x}(z):=u(x+rz)$ and, loosely speaking, the space $\Xi (\mathrm{B}_1(0))$ consists of all functions $a \in C^{\infty}(\mathrm{B}_1(0); \mathbb{R}^m)$ whose one-dimensional restrictions $a(\gamma(t)) \cdot g(\gamma(t),\dot{\gamma}(t))$ are $1$-Lipschitz continuous whenever $\gamma$ solves $\ddot{\gamma}= F(\gamma,\dot{\gamma})$. Given a first-order constant coefficients operator $\mathcal{A}$, when investigating the $\sigma$-finiteness of \eqref{e:voscillation20000} with respect to $\mathcal{H}^{n-1}$ for $u \in BV^{\mathcal{A}}$ one can make use of Poincar\'e type of inequalities (cf.~\cite[Theorem~3.1]{MR1480240} and \cite{Gme19}) and appeal to the standard density estimates for Radon measures (cf.~\cite{afp}). 
This implies in particular the (weaker) $\sigma$-finiteness required in (2). 

In Section \ref{s:GBD} we apply the above strategy to a novel space of functions with generalised bounded deformation $GDB_F(\Om)$ containing those maps $u$ whose one-dimensional slices~$t \mapsto u (\gamma(t)) \cdot \dot{\gamma} (t)$ have bounded variation when computed on solutions~$\gamma$ of the ODE \eqref{e:int1000}. Inspired by \cite{dal}, we say that $u \in GDB_F(\Om)$ if there exists $\lambda \in \mathcal{M}^+_b(\Om)$ such that for every Borel subset $B\subseteq \Om$ and every curvilinear projection $P_\xi \colon \Om \to \xi^\bot$
\begin{align*}
\int_{\xi^{\bot}} \Big( |{\rm D} \hat{u}^{\xi}_{y} | (B^{\xi}_{y} \setminus J^{1}_{\hat{u}^{\xi}_{y}}) & + \HH^{0} (B^{\xi}_{y} \cap J^{1}_{\hat{u}^{\xi}_{y} } ) \Big)\, \di \HH^{n-1}(y) 
\leq \|\dot{\varphi}_\xi\|^2_{L^\infty}{\rm Lip}(P_{\xi};\Om)^{n-1} \lambda(B)\,,
\end{align*}
where $J^1_{\hat{u}^{\xi}_{y}}:= \{t \in \Om^\xi_y : \, |[\hat{u}^{\xi}_y (t)]| > 1 \}$ (see Definition \ref{d:GBD}). In order to conclude the $\sigma$-finiteness of \eqref{e:voscillation20000}, we show in Theorem \ref{t:poincare} a weak Poincar\'e inequality involving functions in $\Xi ({\rm B}_{1}(0))$ and the measure $\lambda$. To this purpose, we have to make a stronger assumption on the field~$F$. Namely, we suppose $F$ to satisfy a condition which we call \emph{Rigid Interpolation} \eqref{hp:F2}. Such a condition requires a (local) control on the $L^\infty$-norm of the curvilinear symmetric gradient, seen as an operator acting on smooth vector fields, in terms of a discrete semi-norm which takes into account the values of vector fields on the vertices of $n$-dimensional simplexes of $\mathbb{R}^n$ (see \eqref{e:rip5}). In the Riemmannian setting, where~$F$ takes the form~\eqref{e:int1100}, in order to ensure the validity of~\eqref{hp:F2} we have to perform a careful analysis based on a blow-up argument regarding Christoffel symbols around a point, allowing for a (local) control of the kernel of the curvilinear symmetric gradient (see Section \ref{sub:Riemann}). This analysis will also serve as a starting point for investigating further structural properties of the space $GBD_F(\Om)$ as well as compactness and lower semicontinuity issues.  This will be the subject of investigation in a forthcoming paper~\cite{AT_23}.

\subsection*{Plan of the paper.} In Section~\ref{s:preliminaries} we present the basic notation and preliminaries, as well as recall the definition of integralgeometric measure and a related rectifiability criterion from~\cite{Tas22}. In Section~\ref{s:curvilinear} we discuss the notion of (family of) curvilinear projections and present its main properties. Section~\ref{s:directional} is devoted to the proofs of Theorem~\ref{t:int1} and Corollary~\ref{c:int2}. In Section~\ref{s:applications} we investigate the relation between~\eqref{e:voscillation20000} and condition~(2) of Theorem~\ref{t:int1}. Finally, in Section \ref{s:GBD} we define the space $GBD_F(\Om)$ and prove a Poincar\'e inequality in such setting (see Theorem \ref{t:poincare}), which, in turn, guarantees the applicability of Corollary \ref{c:int2}.

\section{Preliminaries and notation}
\label{s:preliminaries}

\subsection{Basic notation}
For $n , k \in \mathbb{N}$, we denote by~$\mathcal{L}^{n}$ and by~$\mathcal{H}^{k}$ the Lebesgue and the $k$-dimensional Hausdorff measure in~$\R^{n}$, respectively. The symbol $\mathbb{M}^n$ stands for the space of square matrices of order~$n$ with real coefficients, while~$\mathbb{M}^n_{sym}$ denotes its subspace of symmetric matrices. The set $\{e_{i}\}_{i=1}^{n}$ denotes the canonical basis of~$\R^{n}$ and $| \cdot|$ is the Euclidean norm on~$\R^{n}$. For every $\xi \in \R^{n}$, the map~$\pi_{\xi} \colon \R^{n} \to \R^{n}$ is the orthogonal projection over the hyperplane orthogonal to~$\xi$, which will be indicated with~$\xi^{\bot}$. For $x \in \R^{n}$ and~$\rho>0$, ${\rm B}_{\rho}(x)$ stands for the open ball of radius~$\rho$ and center~$x$ in~$\R^{n}$.

Given~$U_{j}$ a sequence of open subsets of~$\R^{n}$, $\Omega$ open subset of~$\R^{n}$, and~$f_{j} \in C^{\infty}(U_{j}; \R^{k})$, we say that $f_{j} \to f$ in~$C^{\infty}_{loc} (\Om; \R^{k})$ if~$f \in C^{\infty}(\Om; \R^{k})$, $U_{j} \nearrow \Om$, and $f_{j} \to f$ in~$C^{\infty}(W; \R^{k})$ for every $W \Subset \Om$. For a Lipschitz function  \UUU $f\colon \Om \to \R^{k}$, \EEE we denote by ${\rm Lip} (f;\Om)$ the least Lipschitz constant of~$f$ on~$\Om$. We will drop the dependence on the set whenever it is clear from the context.

Given an open subset~$U$ of~$\R^{n}$,~$\mathcal{M}_{b}(U)$ (resp.~$\mathcal{M}_{b}^{+}(U)$) is the space of bounded Radon measures on~$U$ (resp. bounded and positive Radon measures on~$U$). Given a Borel map~$\psi\colon U \to V \subseteq \R^{k}$ and a measure~$\mu \in \mathcal{M}_{b}(U)$, the push-forward measure of~$\mu$ through~$\psi$ is denoted by~$\psi_{\sharp}(\mu) \in \mathcal{M}_{b}(V)$. The set of all Borel \UUU subsets \EEE of~$U$ is indicated by~$\mathcal{B}(U)$. For every $A \subseteq \R^{n} \times \mathbb{S}^{n-1}$, every $\xi \in \mathbb{S}^{n-1}$, and every $x \in \R^{n}$ we will denote 
$$ A_{\xi} := \{ (x \in \R^{n}: (x, \xi) \in A\} \qquad A_{x} := \{ \xi \in \mathbb{S}^{n-1}: (x, \xi) \in A\}\,.$$

For every $p \in [1, +\infty]$, $L^{p}(U; \R^{k})$ stands for the space of $p$-summable functions from~$u$ with values in~$\R^{k}$. The usual $L^{p}$-norm is denoted by~$\| \cdot\|_{L^{p}(U)}$. We will drop the set~$U$ in the notation of the norm when there is no chance of misunderstanding.

\subsection{A rectifiability criterion for a class of integralgeometric measures}

This section is based on the techniques developed in \cite{Tas22}.

\begin{definition}[Countably rectifiable set]
 We say that a set $R \subseteq \Omega$ is countably $(n-1)$-rectifiable if and only if $R$ equals a countable union of images of Lipschitz maps~$(f_i)_i$ from some bounded sets $E_i \UUU \subseteq \EEE \mathbb{R}^{n-1}$ to $\Omega$. 
\end{definition}

 \begin{definition}[Rectifiable measure]
 Let $\mu$ be a measure on $\Omega$. We say that~$\mu$ is $(n-1)$-rectifiable if there exist an $(n-1)$-rectifiable set $R$ and a real-valued measurable function~$\theta$ such that
\begin{equation*}
\mu = \theta \, \mathcal{H}^{n-1} \restr R\,.
\end{equation*}
 \end{definition}


The notion of \emph{transversal family of maps} will play a fundamental role along this section. The following definition is an adaptation of \cite[Definition 2.4]{hov}. 

\begin{definition}[Transversality]
\label{d:transversal}
Let $\Omega \subseteq \mathbb{R}^n$ be open and let $S_i:=\{\xi \in \mathbb{S}^{n-1}  :  |\xi\cdot e_i| \geq 1/\sqrt{n}  \}$ for $i=1,\dotsc,n$. We say that a family of Lipschitz maps $P_\xi \colon \Omega \to \xi^\bot$ for $\xi \in \mathbb{S}^{n-1}$ is a transversal family of maps on~$\Omega$ if for every $i=1,\dotsc,n$ the maps
\begin{align*}
P^i_\xi(x) & := \pi_{e_i}\circ P_\xi(x) \qquad \text{for } \xi \in S_i, \  x \in  \Omega\,,
\\
 T^{i}_{xx'}(\xi) & := \frac{P^i_\xi(x) - P^i_\xi(x')}{|x-x'|} \qquad \text{for } \xi \in S_{i}, \ x, x' \in \Om \text{ with $x \neq x'$}
\end{align*}
satisfy the following properties:
\begin{enumerate}[label=(H.\arabic*),ref=H.\arabic*]
    \item For every $x \in \Omega$ the map $\xi \mapsto P^i_\xi(x)$ belongs to $C^2(S_i;\mathbb{R}^{n-1})$ and
    \begin{equation}
    \label{e:h1}
    \sup_{(\xi,x) \in S_i \times \Omega} |D^j_\xi P^i_\xi(x)| < \infty  \qquad \text{for }j=1,2\, ;
    \end{equation}
    \item \label{hp:H2} There exists a constant $C' >0$ such that for every $\xi \in S_i$ and $x,x' \in \Omega$ with $x \neq x'$ 
    \begin{equation}
    \label{e:h2}
        |T^{i}_{xx'}(\xi)| \leq C' \ \ \ \text{ implies } \ \ \
        |\text{J}_\xi T^{i}_{xx'}(\xi)| \geq C';
    \end{equation}
   \item \label{hp:H3} There exists a constant $C'' >0$ such that 
   \begin{equation}
   \label{e:h3}
       | D^j_\xi T^{i}_{xx'}(\xi) | \leq C''\qquad  \text{for }j=1,2\,
   \end{equation}
   for $\xi \in S_i$ and $x,x' \in \Omega$ with $x \neq x'$.
\end{enumerate}
\end{definition}

\begin{remark}\label{r:'}
Hypothesis \eqref{hp:H2} is equivalent to the following:
\begin{enumerate}[label=(H.2'), ref=H.2']
\item \label{hp:'}
There exists two constants $C'_1,C'_2 >0$ such that for every $\xi \in S_i$ and $x,x' \in \Omega$ with $x \neq x'$ 
    \begin{equation}
    \label{e:h2'}
        |T^{i}_{xx'}(\xi)| \leq C'_1 \ \ \ \text{ implies } \ \ \
        |\text{J}_\xi T^{i}_{xx'}(\xi)| \geq C'_2;
    \end{equation}
    \end{enumerate}
    Indeed ,\eqref{hp:H2} clearly implies \eqref{hp:'}. Viceversa if \eqref{hp:'} holds true with $C'_1 \leq C'_2$ we can simply replace $C'_2$ with $C'_1$ in \eqref{e:h2'} and the implication remains true. Similarly if $C'_1 > C'_2$ any triple $(\xi,x,x')$ satisfying $|T^{i}_{xx'}(\xi)| \leq C'_2$ satisfies $|T^{i}_{xx'}(\xi)| \leq C'_1$ as well, therefore implication \eqref{e:h2'} remains true by replacing $C'_1$ with $C'_2$. 
\end{remark}

\begin{remark}
\label{r:observation}
The definition of transversality above slightly differs from that in \cite{Tas22}. Indeed, it turns out that for every $i=1,\dotsc,n$ $(P^i_\xi)_{\xi \in S_i }$ is transversal in the sense of \cite[Definition 3.3]{Tas22}, while the entire family $(P_\xi)_{\xi \in \mathbb{S}^{n-1}}$ is not. However, the relevant property of transversality contained in \cite{Tas22} can be easily transferred to $(P_\xi)_{\xi \in \mathbb{S}^{n-1}}$ simply by writing $(P_\xi)_{\xi \in \mathbb{S}^{n-1}} = \bigcup_{i=1}^n (P^i_\xi)_{\xi \in S_i }$.
\end{remark}

We are now in position to introduce a curvilinear version of codimension one Farvard's integralgeometric measure.

\begin{definition}
\label{d:Favard}\UUU Let $(P_\xi)_{\xi \in \mathbb{S}^{n-1}}$ be a family of transversal maps on~$\Om$. We \EEE define the Borel regular measure $\tilde{\mathcal{I}}^{n-1}$ on $\Om$ as
\begin{align}
    \label{e:farvard1}
    \tilde{\mathcal{I}}^{n-1}(B)&:= \int_{\mathbb{S}^{n-1}} \bigg( \int_{\xi^\bot} \mathcal{H}^0(B \cap P^{-1}_\xi(y)) \, \di \mathcal{H}^{n-1}(y) \bigg) \di \mathcal{H}^{n-1}(\xi) \qquad B \in \mathcal{B}(\Omega)\,, \\
    \label{e:farvard2}
     \tilde{\mathcal{I}}^{n-1}(E) &:= \inf \{ \tilde{\mathcal{I}}^{n-1}(B) : E \subset B, \ B \in \mathcal{B}(\Omega) \}  \qquad E \UUU \subseteq \EEE \Omega\,.
\end{align}
Notice that the required measurability for \eqref{e:farvard1} can be obtained by means of the measurable projection theorem \cite[Section 2.2.13]{fed} with a similar argument as in \cite[Section 5]{Tas22}.
\end{definition}

Furthermore we need to define the following additional class of measures.

\begin{definition}
\UUU Let $(P_\xi)_{\xi \in \mathbb{S}^{n-1}}$ be a family of transversal maps on~$\Om$ and let $(\eta_\xi)_{\xi \in \mathbb{S}^{n-1}}$ be a family of Borel regular measures of~$\mathbb{R}^n$ satisfying
\begin{equation}
    \label{e:condigm2000}
    \xi \mapsto \eta_\xi(A_\xi) \text{ is $\mathcal{H}^{n-1}$-measurable  for every $A \in \mathcal{B} (  \Omega \times \mathbb{S}^{n-1} )$.}
\end{equation}
Let moreover $f_B(\xi):= \eta_\xi(B)$ for every $\xi \in \mathbb{S}^{n-1}$ and $B \in \mathcal{B}( \Omega)$. For $p \in [1, +\infty]$, \EEE we define the set function 
\begin{equation}
    \label{e:condigm1000}
    \zeta_p(B) := \|f_B  \|_{L^p(\mathbb{S}^{n-1})}.
\end{equation}
Via the classical Caratheodory's construction we define the measure
\begin{equation}
     \label{e:caratheodoryc2}
    \mathscr{I}_p^{n-1}(E) := \sup_{\delta>0} \, \inf_{G_\delta} \sum_{B \in G_\delta} \zeta_p(B),
 \end{equation}
whenever $E \subseteq \Om$ and where $G_\delta$ is the family of all countable Borel coverings of $E$ made of sets having diameter less than or equal to~$\delta$. 
\end{definition}

\begin{definition}
\UUU Let  $(P_\xi)_{\xi \in \mathbb{S}^{n-1}}$ be a family of transversal maps on~$\Om$ and let $(\eta_\xi)_{\xi \in\mathbb{S}^{n-1}}$ be as in \eqref{e:condigm2000}. We \EEE  define the set functions
 \begin{align}
     \label{e:condigm1.0.0}
     \hat{\zeta}(A) & := \int_{\mathbb{S}^{n-1}} \eta_\xi(A_\xi) \, \di \mathcal{H}^{n-1}(\xi) \qquad \text{for $A \in \mathcal{B}( \Omega \times \mathbb{S}^{n-1})$}\,,
     \\
    \label{e:caratheodoryc2.1.0}
    \hat{\mathscr{I}}_{n-1}(F) & :=  \sup_{\delta>0} \, \inf_{G'_\delta} \sum_{B \in G'_\delta} \hat{\zeta}(B) \qquad \text{for $F \subseteq \Omega \times \mathbb{S}^{n-1}$}\,,
 \end{align}
 where~$G'_\delta$ is the family of all countable Borel coverings of~$F$ made of sets having diameter less than or equal to $\delta$.
\end{definition}

We have the following general representation result for $\mathscr{I}^{n-1}_{1}$ and~$\hat{\mathscr{I}}_{n-1}$.

\begin{proposition}
\label{p:coincidence}
Under the previous assumption the measures $\mathscr{I}^{n-1}_1$ and $\hat{\mathscr{I}}_{n-1}$ satisfy
\begin{align*}
    \mathscr{I}^{n-1}_1(E) &= \inf_{\substack{E \subseteq B \\ B \in \mathcal{B}(\Om)}} \int_{\mathbb{S}^{n-1}} \eta_\xi(B) \, \di \mathcal{H}^{n-1}(\xi) \qquad \text{for every  $E \subseteq \Omega$\,,} \\
    \hat{\mathscr{I}}_{n-1}(F) &= \inf_{\substack{F \subseteq A \\ A \in \mathcal{B} ( \Om \times \mathbb{S}^{n-1}) }} \int_{\mathbb{S}^{n-1}} \eta_\xi(A_\xi) \, \di \mathcal{H}^{n-1}(\xi) \qquad  \text{for every  $F \subseteq \Omega \times \mathbb{S}^{n-1}$}.
\end{align*}
\end{proposition}

In the next two propositions we state two properties regarding the disintegration of~$\hat{\mathscr{I}}_{n-1}$ w.r.t.~$ \mathscr{I}^{n-1}_1$ and the measures~$\eta_{\xi}$. In particular, we fix $(P_{\xi})_{\xi \in \mathbb{S}^{n-1}}$ a transversal family of maps in~$\Om$.

\begin{proposition}
\label{p:fproposition}
Let $(P_{\xi})_{\xi \in \mathbb{S}^{n-1}}$ be a transversal family of maps in~$\Om$, let $(\eta_{\xi})_{\xi \in \mathbb{S}^{n-1}}$ be a family of Radon measures as in~\eqref{e:condigm2000}--\eqref{e:caratheodoryc2.1.0}. 
Assume that there exists $p \in (1, +\infty]$ such that $\mathscr{I}_p^{n-1}$ is finite and that
\begin{equation}
\label{e:abscont}
(P_{\xi})_{\sharp} \, \eta_\xi \ll \mathcal{H}^{n-1}  \restr \xi^{\bot}    \  \text{for $\mathcal{H}^{n-1}$-a.e. } \xi \in \mathbb{S}^{n-1}\,.
\end{equation}
Then, the measure $\hat{\mathscr{I}}_{n-1}$ can be disintegrated as 
\begin{equation*}
   \hat{\mathscr{I}}_{n-1} = (f_x \, \mathcal{H}^{n-1} \restr \mathbb{S}^{n-1}) \otimes  \mathscr{I}^{n-1}_1 ,  
\end{equation*}
where $f_x \colon \mathbb{S}^{n-1} \to \mathbb{R}$ are $\mathcal{H}^{n-1}$-measurable functions with $\int_{\mathbb{S}^{n-1}}f_x \, \di \mathcal{H}^{n-1}=1$ for $\mathscr{I}^{n-1}_{1}$-a.e. $x \in \Omega$. Moreover, the family of functions $(f_x)_{x \in \Omega}$ can be chosen in such a way that, defining $f \colon \Omega \times \mathbb{S}^{n-1} \to \mathbb{R}$ as $f(x,\xi):=f_x(\xi)$, then $f$ is a Borel function.
\end{proposition}

\begin{proof}
It follows by a combination of disintegration theorem with \cite[Proposition 4.7]{Tas22} and \cite[Remark 4.3]{Tas22}.
\end{proof}




We present the definition of \emph{integralgeometric measure}. 

\begin{definition}
\label{d:igm}
 Let $(P_\xi)_{\xi \in \mathbb{S}^{n-1}}$ be a family of transversal maps on~$\Om$, let the family~$(\eta_{\xi})_{\xi \in \mathbb{S}^{n-1}}$ be as in~\eqref{e:condigm2000}, and let the measure $\mathscr{I}^{n-1}_p$ be as in~\eqref{e:condigm1000}--\eqref{e:caratheodoryc2}. We say that~$\mathscr{I}^{n-1}_p$ is integralgeometric if and only if \eqref{e:abscont} holds true and there exists $E \in \mathcal{B}( \Omega)$ such that
\begin{align}
\label{e:condigm5}
&\eta_\xi(\Omega \setminus E)=0 \ \ \text{for $\mathcal{H}^{n-1}$-a.e. $\xi \in \mathbb{S}^{n-1}$}\\
\label{e:condigm6}
&\mathcal{H}^0(E \cap P^{-1}_\xi(y)) <  +\infty  \text{ for $\mathcal{H}^{n-1}$-a.e. }\xi \in \mathbb{S}^{n-1} \text{, $(P_{\xi})_{\sharp}\eta_\xi$-a.e. }y \in \xi^\bot.
\end{align}
\end{definition}

We conclude this section with a fundamental rectifiability criterion for integralgeometric measures (see \cite[Theorem 1.5]{Tas22}).

\begin{theorem}
\label{t:rectheorem}
 Let $\Om$ be an open subset of~$\R^{n}$, let $(P_\xi)_{\xi \in \mathbb{S}^{n-1}}$ be a family of transversal maps on~$\Om$, and let $(\eta_{\xi})_{\xi \in \mathbb{S}^{n-1}}$ and $\mathscr{I}^{n-1}_{p}$ be as in~\eqref{e:condigm2000}--\eqref{e:caratheodoryc2} for $p \in [1, +\infty]$. Assume that there exists $p \in (1, +\infty]$ such that $\mathscr{I}_p^{n-1}$ is a finite integralgeometric measure. Then, $\mathscr{I}^{n-1}_1$ is $(n-1)$-rectifiable.
\end{theorem}

\section{Curvilinear projections}
\label{s:curvilinear}
 
 This section is dedicated to the basic definitions of curvilinear projections (cf.~Definitions~\ref{d:param} and~\ref{d:CP}) and to the local construction of a family of curvilinear projections (see Section~\ref{sub:curvpro}). Further properties of curvilinear projections are studied in Section~\ref{s:technical}.

 \subsection{Assumptions and basic definitions}
 \label{sub:basic}
 
 From now on we fix a field $F \in C^{\infty} (\R^{n} \times \R^{n}; \R^{n})$ which is $2$-homogeneous in the second variable, namely, for every $x,v \in \mathbb{R}^n$ and $\alpha \in \mathbb{R}$ we have
\begin{equation}
    \label{e:quadratic}
    F(x,\alpha v) = \alpha^2 F(x,v) \,.
\end{equation}


We now give the definitions of parametrized \UUU maps and of curvilinear projections. \EEE

\begin{definition}[Parametrized maps]
\label{d:param-maps}
Let $\Omega$ be an open subset of~$\mathbb{R}^n$ and $\xi \in \mathbb{S}^{n-1}$. We say that a map $P \colon \Omega \to \xi^{\bot}$ is {\em parametrized} on~$\Omega$ if and only if there exist $\rho,\tau>0$ and a smooth Lipschitz map $ \varphi  \colon \{ y + t\xi: (y, t) \in [\xi^{\bot} \cap \mathrm{B}_{\rho}(0)] \times  (-\tau, \tau)\}  \to \R^{n}$  such that
\begin{enumerate}[label=(\arabic*), ref=(\arabic*)]
    \item $\Omega \subseteq \text{Im}(\varphi  )$;
    \item $\varphi^{-1}  \restr \Om$ is a bi-Lipschitz diffeomorphism with its image;
    \item $ P(\varphi ( y +  t\xi )) = y$ for every $(y, t) \in [\xi^{\bot} \cap \mathrm{B}_{\rho}(0)] \times  (-\tau, \tau)$ such that $y + t\xi \in  \varphi^{-1} (\Om)$;
\end{enumerate}
\end{definition}

\begin{remark}
\label{r:notation-velocity}
Given $\varphi \colon \{ y + t\xi: (y, t) \in [\xi^{\bot} \cap \mathrm{B}_{\rho}(0)] \times  (-\tau, \tau)\}  \to \R^{n}$ as in Definition~\ref{d:param-maps}, for simplicity of notation we denote \UUU by \EEE $\dot{\varphi}$ and $\ddot{\varphi}$ the first and second derivatives w.r.t.~$t$ of the function $t \mapsto \varphi (y + t\xi)$.
\end{remark}

\begin{remark}
 Conditions (2) \UUU and \EEE (3) of parametrized map imply
\begin{enumerate}[label=(\arabic*), ref=(\arabic*), ]
\setcounter{enumi}{3}
    \item
    $ \varphi ( P (x) +  (\varphi^{-1}(x) \cdot \xi) \xi \big) = x$ for every $x \in \Omega$.
\end{enumerate}
 We will more compactly denote by~$t^\xi_x$ the real number $ \varphi^{-1}(x) \cdot \xi$ for every $x \in \Omega$ and every $\xi \in \mathbb{S}^{n-1}$. Whenever~$\xi$ is fixed and there is no misunderstanding, we simply drop the index~$\xi$ and therefore write~$t_x$ instead of~$t^\xi_x$. 
\end{remark}

\begin{definition}[Velocity field]
Let~$\Om$ be a bounded open subset of~$\R^{n}$, $\xi \in \mathbb{S}^{n-1}$, and let $P \colon \Om \to \xi^{\bot}$ be a map parametrized by~$\varphi_{\xi}$ on~$\Om$. For every $x \in \Omega$ we define the {\rm velocity field} 
\begin{equation*}
\xi_{\varphi}(x):= \dot{\varphi}(P (x) +t_x \xi).
\end{equation*}
\end{definition}


\begin{definition}[Curvilinear projections]
\label{d:CP-maps}
Let $\Omega$ be an open subset of~$\mathbb{R}^n$ and $\xi \in \mathbb{S}^{n-1}$. We say that a smooth Lipschitz map $P \colon \Omega \to \xi^{\bot}$  is a \emph{curvilinear projection} (with respect to $F$) on~$\Om$ if \UUU the following holds: \EEE
\begin{enumerate}
\item $P$ is parametrized on~$\Om$ by $\varphi \colon \{ y + t\xi: (y, t) \in [\xi^{\bot} \cap \mathrm{B}_{\rho}(0)] \times  (-\tau, \tau)\}  \to \R^{n}$;
\item For every $(y, t) \in [\xi^{\bot} \cap {\rm B}_\rho(0)] \times (-\tau,\tau)$ we have
\begin{equation*}
    \ddot{\varphi} (y + t\xi) = F(\varphi(y + t\xi) ,\dot{\varphi}(y + t\xi) )\,.
    \end{equation*}
\end{enumerate}
\end{definition}


In our analysis, we will further need the following notions of parametrized \UUU family and of family of curvilinear projections. \EEE

\begin{definition}[Parametrized family]
\label{d:param}
Let $\Omega$ be an open subset of~$\mathbb{R}^n$. We say that a family $P_\xi \colon \Omega \to \xi^\bot$ for $\xi \in \mathbb{S}^{n-1}$ is {\em parametrized} on~$\Omega$ if and only if there exist $\rho,\tau>0$, an open subset~$A$ of~$\R^{n} \times \mathbb{S}^{n-1}$, and a smooth Lipschitz map $\varphi \colon  A \to \R^{n}$  such that
\begin{enumerate}[label=(\arabic*), ref=(\arabic*)]
    \item for every $\xi \in \mathbb{S}^{n-1}$ we have $A_{\xi} = \{ y + t\xi: (y, t) \in   [\xi^\bot \cap {\rm B}_\rho(0)] \times (-\tau,\tau) \} $;
    \item for every $\xi \in \mathbb{S}^{n-1}$, $P_{\xi}$ is parametrized on~$\Om$ by the map $\varphi_{\xi} := \varphi (\cdot, \xi) \colon A_{\xi} \to \R^{n}$. 
\end{enumerate}
\end{definition}



We also give the definition of families of curvilinear projections.

\begin{definition}[Family of curvilinear projections]
\label{d:CP}
Let $\Omega$ be an open subset of~$\mathbb{R}^n$. We say that a family of maps $P_\xi \colon \Omega \to \xi^\bot$ for $\xi \in \mathbb{S}^{n-1}$ is a family of \emph{curvilinear projections} on~$\Om$ if the following conditions hold:
\begin{enumerate}[label=(\arabic*), ref=(\arabic*)]
\item  the family $(P_{\xi})_{\xi \in \mathbb{S}^{n-1}}$ is parametrized by $\varphi \colon  A \to \R^{n}$; 
\item  for every $\xi \in \mathbb{S}^{n-1}$, $P_\xi$ is a curvilinear projection on~$\Om$ with parametrization $\varphi_{\xi} = \varphi(\cdot, \xi)$; 
\item $(P_\xi)_{\xi \in \mathbb{S}^{n-1}}$ is a transversal family of maps on~$\Om$;
\item for every $x \in \Om$, the map $\xi \mapsto \xi_{\varphi}(x)/|\xi_{\varphi}(x)|$ is a diffeomorphism from $\mathbb{S}^{n-1}$ onto itself.
\end{enumerate}
\end{definition}

We conclude this section by defining suitable slices of a measurable function~$u \colon \Om \to \R^{m}$ and of a subset~$B$ of~$\Om$ w.r.t.~a curvilinear projection. For this purpose we fix a map $g \colon \Omega \times \mathbb{R}^n \to \mathbb{R}^m$  satisfying the following properties:
\begin{enumerate}[label=(G.\arabic*),ref=G.\arabic*]
\item \label{G1} $g \in C(\Om \times \R^{n} ; \R^{m})$;
\item \label{G2} 
For every $x \in \Omega$ and $\Sigma \subset \mathbb{S}^{n-1}$ with $\mathcal{H}^{n-1}(\Sigma)>0$ we find an open neighborhood $U$ of $x$, an integer $0 \leq k \leq m$, and vectors $\{\xi_1,\dotsc,\xi_k\} \subset \Sigma$ such that
\begin{align*}
    \text{span}\{g(z,\xi_1), \dotsc,g(z,\xi_k)\} =\UUU \text{span}\{g(z,v) : \EEE v \in \mathbb{R}^n\} \qquad \text{for every $z \in U$}.
\end{align*} 
\end{enumerate}

\begin{remark}
We notice that for most of our arguments we only need~\eqref{G1}, while \eqref{G2} will be used in Corollary~\ref{c:int2} and in the corresponding definitions.
\end{remark}

We now give two definitions.

\begin{definition}
\label{d:mathfraku}
Let $\Om$ be an open subset of~$\R^{n}$ and $g \colon \Omega \times \mathbb{R}^n \to \mathbb{R}^m$ satisfying~\eqref{G1}--\eqref{G2}. Then, we define a continuous map $\pi \colon \Omega \times \mathbb{R}^m \to \mathbb{R}^m$ in such a way that $\pi(x,\cdot)$ coincides with the orthogonal projection of $\mathbb{R}^m$ onto $\text{span}\{ \text{Im}(g(x,\cdot))\}$ for every $x \in \Omega$. Moreover, for every $u \colon \Omega \to \mathbb{R}^m$ we define $\mathfrak{u} \colon \Omega \to \mathbb{R}^m$ as
\begin{equation*}
\mathfrak{u} (x):=\pi(x,u(x))\,.
\end{equation*}
\end{definition}

\begin{definition}[Slices]
\label{d:slices}
Let $\Om$ be an open subset of~$\R^{n}$, $\xi \in \mathbb{S}^{n-1}$, let $P \colon \Om \to \xi^{\bot}$ be a curvilinear projection on~$\Om$ parametrized by $\varphi \colon \{ y + t\xi: (y, t) \in [\xi^{\bot} \cap {\rm B}_{\rho} (0)] \times (-\tau, \tau)\} \to \R^{n}$, and let $g \colon \Om \times \R^{n} \to \R^{m}$ satisfy~\eqref{G1}. For every $B \subseteq \Om$ and every $y \in \xi^{\bot} \cap {\rm B}_{\rho} (0)$ we define
\begin{displaymath}
B^{\xi}_{y}:= \{ t \in \R: \, \varphi (y +  t\xi ) \in B\}\,.
\end{displaymath}
For every measurable function $u \colon \Om \to \R^{m}$, we define $\hat{u}^{\xi}_y \colon \Om^{\xi}_{y} \to \mathbb{R}$ as 
\begin{equation*}
\hat{u}^{\xi}_{y} (t) := u(\varphi (   y+ t\xi )) \cdot g(\varphi (y +  t\xi ),\dot{\varphi} (y + t\xi))\,,
\end{equation*}
We further define $u_\xi \colon \Omega \to \mathbb{R}$ by
\begin{equation*}
u_\xi(x) := u(x) \cdot g(x, \xi_\varphi(x)) \,,
\end{equation*}
and we notice the following identity
\begin{equation}
    \label{e:sliceide}
    u_\xi(\varphi ( y + t \xi ) ) = \hat{u}^\xi_y(t) \qquad  \text{for }\xi \in \mathbb{S}^{n-1} \text{ and } (y,t) \in [\xi^\bot \cap \mathrm{B}_\rho(0)] \times (-\tau,\tau).
\end{equation}
Eventually, for a measurable function $v \colon \Om \to \R^{m}$ we also set $v^{\xi}_{y} (t) := v(\varphi (y + t\xi))$ for $t \in \Om^{\xi}_{y}$.


\end{definition}

\EEE
\subsection{A technical result on curvilinear projections}
\label{s:technical}

This section is devoted to a technical result concerning curvilinear projections. 
We start by introducing the \emph{exponential map}.

\begin{definition}
\label{d:exp}
 Let $F \in C^{\infty} (\R^{n} \times \R^{n}; \R^{n})$ satisfy~ \eqref{e:quadratic}. For every $x \in \R^{n}$ we define, where it exists, the {\em exponential map} $\text{exp}_{x} \colon \mathbb{R}^n \to \mathbb{R}^n$ as $\text{exp}_{x}(\xi):= v_{\xi, x} (1)$, where $t\mapsto v_{\xi, x} (t)$ solves
\begin{equation}
\label{e:system-exponential}
\begin{cases}
\ddot{u}(t) = F(u(t),\dot{u}(t)), \ &t \in \mathbb{R}\,, \\
u(0)=x\,,&\\
\dot{u}(0)=\xi\,.&
\end{cases}
\end{equation}
\end{definition}

In the next definition we introduce the concept of injectivity radius.

\begin{definition}
\label{d:inj}
For every $x \in \mathbb{R}^n$ we define the injectivity radius $\text{inj}_{x}\in  [0, + \infty)$ as the supremum of all $r>0$ for which $\text{exp}_{x} \restr {\rm B}_{\overline{r}}(0)$ is well defined and $\text{exp}^{-1}_{x} \restr {\rm B}_{\overline{r}}(x)$ is a diffeomorphism with its image. 
\end{definition}

The definition of~$\text{exp}_{x}$ in a small ball ${\rm B}_{r} (0)$ is justified by the following lemma.

\begin{lemma}
\label{l:exp}
Let $F \in C^{\infty}( \mathbb{R}^n \times \mathbb{R}^n ; \mathbb{R}^n)$ satisfy~ \eqref{e:quadratic}. 
For every~$x \in \Omega$ we have $\emph{inj}_{x}>0$. 
\end{lemma}
\begin{proof}
 By the $2$-homogeneity of $F(x,\cdot)$ we get that 
\begin{equation}
\label{e:retr17}
v_{s\xi, x} (t) = v_{\xi, x} (st)\qquad  \text{for $s,t \in [0, + \infty)$,  $\xi \in \mathbb{R}^n$.}
\end{equation}
Hence, by the local well-posedness of ODEs we have that there exists~$r>0$ such that~$\text{exp}_{x}$ is well-defined on~${\rm B}_{r}(0)$. For every $i \in \{1,\dotsc,n \}$ we have that
\[
D\exp_{x}(0) e_i = \lim_{t \to 0^+} \frac{v_{te_{i}, x} (1) - v_{0, x} (1)}{t} = \lim_{t \to 0^+} \frac{v_{e_{i}, x} (t) - v_{e_{i}, x} (0) }{t} = \dot{v}_{e_{i}, x} ( 0) = e_i\,.
\]
Thus, the differential of $\text{exp}_{x}$ at $0$ is the identity. Applying the implicit function theorem, we find a sufficiently small $\tilde r>0$ such that $\text{exp}^{-1}_{x} \restr {\rm B}_{\tilde r}(x)$ is a diffeomorphism with its image. Therefore setting $\text{inj}_{x} \geq \min\{ r, \tilde{r}\} >0$.
\end{proof}

\begin{definition}
 Let $F \in C^{\infty}( \mathbb{R}^n \times \mathbb{R}^n ; \mathbb{R}^n)$ satisfy~ \eqref{e:quadratic}. and let $(P_\xi)_{\xi \in \mathbb{S}^{n-1}}$ be a family of curvilinear projections on~$\Omega$. Thanks to Lemma~\ref{l:exp} we may define for $0 < \overline{r} < \text{inj}_{x}$ the map $\phi_{x}\colon \mathrm{B}_{\overline{r}}(x) \setminus\{x\} \to \mathbb{S}^{n-1}$ as
\begin{equation}
\label{e:retr12}
\phi_{x}(z):=\psi_{x}^{-1}( \phi_0(\text{exp}^{-1}_{x}(z)))\qquad\text{for every }z \in \mathrm{B}_{\overline{r}}(x) \setminus \{x\}\,,
\end{equation}
where we set $\phi_0(z):=\frac{z}{|z|}$ and $\psi_{x}(\xi):=\phi_0(\xi_\varphi(x))$.
\end{definition}
The following proposition holds.

\begin{proposition}
\label{p:retr}
Let $F \in C^{\infty}( \mathbb{R}^n \times \mathbb{R}^n ; \mathbb{R}^n)$ satisfy~ \eqref{e:quadratic} and let~$(P_\xi)_{\xi \in \mathbb{S}^{n-1}}$ be a family of curvilinear projections on~$\Om$. For every $x \in \Om$, let $0 < \overline{r} < \emph{inj}_{x}$. Then, $\phi_{x}  \in C^{1}(  \mathrm{B}_{\overline{r}}(x) \setminus \{x\} ; \mathbb{S}^{n-1})$ and
\begin{align}
\label{e:chi1.1}
     &P_{\xi}(z)=P_{\xi}(x) \ \ \text{if and only if $\xi = \phi_{x}(z)$ for every $z \in \mathrm{B}_{\overline{r}}(x)\setminus \{x\}$.}\\
     \label{e:retr2}
     &  |\emph{J}\phi_{x}(z)| \leq \frac{C_{x}}{|z-x|^{n-1}} \qquad \text{for every $z \in \mathrm{B}_{\overline{r}}(x) \setminus \{x\}$},
\end{align} 
for some constant $C_{x}>0$. Moreover, if we assume that
\begin{equation}
\label{e:retr9.1.3}
\inf_{(z,\xi) \in {\rm B}_{\overline{r}}(x)\times \mathbb{S}^{n-1}}|\text{J}_z P_{\xi}(z)|>0 \ \,
\end{equation}
then we find a constant $C'_{x}>0$ such that
\begin{equation}
    \label{e:retr2.1}
    \frac{C'_{x}}{|z-x|^{n-1}}\leq |\emph{J}\phi_{x}(z)|  \qquad \text{for every $z \in \mathrm{B}_{\overline{r}}(x) \setminus \{x\}$.}
\end{equation}
\end{proposition}

\begin{proof}
For simplicity of notation, we drop the index~$x$ in the function~$\phi_{x}$. For every $z \in \mathrm{B}_{\overline{r}}(x) \setminus \{x\}$, choosing $\xi_{z} = \text{exp}^{-1}_{x}(z)$ in~\eqref{e:system-exponential} we get that 
\begin{equation}
\label{e:retr18}
v_{\xi_{z}, x} (1) = z\,.
\end{equation}
Therefore, arguing as in~\eqref{e:retr17}, the solution~$u$ of
\begin{equation*}
\begin{cases}
\ddot{u}(t) = F(u(t),\dot{u}(t))  &t \in \mathbb{R}\,, \\
u(0)=x\,,&\\
\dot{u}(0)=\phi_0(\text{exp}^{-1}_{x}(z))\,,&
\end{cases}
\end{equation*}
 satisfies $u(| \xi_{z} |t)= v _{\xi_{z}, x} (t)$.  Setting $\eta:=\psi_{x}^{-1}(\phi_0(\text{exp}^{-1}_{x}(z))$, from the definition of velocity field we have that
 \[
 \varphi_{\eta}( P_{\eta} ( x ) + t_{x}^{\eta} \eta)  = x \ \ \text{ and } \ \ \phi_0(\text{exp}^{-1}_{x}(z)) = \phi_0(\dot{\varphi}_{\eta}( P_{\eta} (x) + t_{x}^{\eta}  \eta)).
 \]
  Since the curve  $\gamma(t):= \varphi_{\eta}(P_{\eta}  (x) + t\eta)$ satisfies $\ddot{\gamma}(t) = F(\gamma(t),\dot{\gamma}(t)), \ \  t \in \mathbb{R}$, we infer from the $2$-homogeneity of $F(z,\cdot)$ the existence of a constant $\alpha$ for which $u(|\xi_z|t)=\varphi_{\eta}(P_{\eta}(x) + \alpha (t + t_{x}^{\eta}) \eta)$ for every~$t$. By~\eqref{e:retr18} and by the definition of~$\phi(z)$ we thus deduce  
 \[
  \varphi_{\phi(z)} ( P_{\phi(z) }(x) + \alpha( 1 + t_{x}^{\phi(z)}) \phi(z))= z\,.
 \]
Therefore, we can make use of (3) of Definition~\ref{d:param-maps} to infer the validity of the \emph{if} part of implication \eqref{e:chi1.1}. The \emph{only if} part of \eqref{e:chi1.1} can be obtained in a similar way.

From property \eqref{hp:H2} of transversal maps we have the existence of a constant $C>0$ such that
\begin{equation}
    \label{e:retr3}
    |\text{J}_\xi( P_{\phi(z)}(z)- P_{\phi(z)}(x))| \geq C |z-x|^{n-1} \qquad \text{for } \UUU (\xi,z) \EEE \in \mathbb{S}^{n-1}  \times {\rm B}_{\overline{r}}(x).
\end{equation}
 Therefore, we are in a position to apply the implicit function theorem and deduce that $z \mapsto \phi(z)$ is $C^1$-regular on the open set ${\rm B}_{\overline{r}}(x) \setminus \{x\}$. We can thus compute the jacobian of $\phi$ as 
\begin{equation}
\label{e:retr15}
    \text{J}\phi(z) = \frac{\text{J}_z P_{\phi(z)}(z)}{\text{J}_\xi( P_{\phi(z)}(z)- P_{\phi(z)}(x))} \qquad  \text{for every }z \in {\rm B}_{\overline{r}}(x)\setminus \{x\}.
\end{equation}
Since the maps $ z \mapsto P_{\xi}(z)$ are Lipschitz continuous in~$\Omega$ we can make use of property~\eqref{e:h3} to deduce that their Lipschitz constants are uniformly bounded with respect to $\xi \in \mathbb{S}^{n-1}$. Therefore, properties~\eqref{e:retr3}--\eqref{e:retr15} yield
\begin{equation*}
    |\emph{J}\phi(z)| \leq \frac{C_{x}}{|z-x|^{n-1}}\qquad \text{ for every } z \in {\rm B}_{\overline{r}}(x) \setminus \{x\},
\end{equation*}
for some constant $C_{x}>0$ depending on~$x$. This proves~\eqref{e:retr2}.

Now using property \eqref{hp:H3} of transversal maps, we have 
\begin{equation}
\label{e:retr9.1.2}
|\text{J}_\xi( P_{\phi(z)}(z)- P_{\phi(z)}(x))| \leq C'|z-x|^{n-1}\qquad \text{for every }z \in {\rm B}_{\overline{r}}(x)\setminus \{x\}.
\end{equation}
Therefore, by virtue of~\eqref{e:retr9.1.3} and of~\eqref{e:retr15}--\eqref{e:retr9.1.2}, we find $C'_{x}>0$ such that 
\begin{equation*}
    |\text{J} \phi(z)| \geq  \frac{C'_{x}}{|z-x|^{n-1}} \qquad  \text{for every }z \in {\rm B}_{\overline{r}}(x) \setminus \{x\}.
\end{equation*} 
This concludes the proof of~\eqref{e:retr2.1} and of the proposition.
\end{proof}

\subsection{Local existence of families of curvilinear projections}
\label{sub:curvpro}
In this subsection we show that it is always possible to locally construct a family of curvilinear projections. This is done by considering suitable flows of solutions of second order ODEs associated to the field $F$. A crucial point in our analysis is that $F$ is $2$-homogeneous in the second variable, which allows us to properly rescale the solution to the ODEs system driven by~$F$.

\begin{definition}
\label{d:curvpro}
Let $x_0 \in \mathbb{R}^n$ and $\rho_{0}>0$. For every $\xi \in {\rm B}_{2}(0)$ and every $y \in \xi^{\bot}\cap {\rm B}_{\rho_{0}}(0)$ we consider the solution $t \mapsto u_{\xi, y}(t)$ of the ODE system
\begin{equation*}
    \begin{cases}
    \ddot{u}(t) = F(u(t), \dot{u}(t)) &  t \in \mathbb{R},\\
    u(0)=y+x_0\,,\\
    \dot{u}(0)=\xi\,,
    \end{cases}
\end{equation*}
which is well-defined for $t \in (-\tau, \tau)$, for a suitable $\tau >0$ depending only on~$x_{0}$ and~$\rho_{0}$, but not on~$\xi$ and~$y$. Then, we define $\varphi_{\xi, x_{0}} \colon \mathbb{R}^n \to \mathbb{R}^n$ as follows: for every $x \in \R^{n}$, if $x = y + t\xi$ with $y \in  \xi^{\bot}\cap {\rm B}_{\rho_{0}}(0)$ and $t \in (-\tau, \tau)$, we set $\varphi_{\xi, x_{0}} (x) := u_{\xi, y}(t)$.

We further define $\varphi_{x_{0}} \colon \R^{n} \times \mathbb{S}^{n-1} \to \R^{n}$ as $\varphi_{x_{0}} (x, \xi) := \varphi_{\xi, x_{0}} (x)$ for $x \in \R^{n}$ and $\xi \in \mathbb{S}^{n-1}$.
\end{definition}

\begin{remark}\label{r:rx0}
Under the assumptions of Definition~\ref{d:curvpro}, the map $\varphi_{\xi, x_{0}}$ is well defined on the open ball ${\rm B}_{r_{x_{0}}} (0)$, for a suitable $r_{x_{0}}>0$ which only depends on~$x_{0}$, but not on~$\xi \in {\rm B}_{2}(0)$.
\end{remark}

\begin{remark}
\label{r:Axi}
Let $\tau$ and~$\rho_{0}$ as in Definition~\ref{d:curvpro}. For every $\xi \in \mathbb{S}^{n-1}$ we may set 
\begin{align*}
A_{\xi} & := \{ y + t\xi: (t, y) \in (-\tau, \tau) \times  (\xi^{\bot}\cap {\rm B}_{\rho_{0}}(0))\} \subseteq \R^{n}\,,
\\
 A&:= \{(x, \xi) \in \R^{n} \times \mathbb{S}^{n-1} : \, x \in A_{\xi},\, \xi \in \mathbb{S}^{n-1}\}\,.
 \end{align*}
Then, $\varphi_{x_{0}}$ is well defined on~$A \times \mathbb{S}^{n-1}$. 
\end{remark}


\begin{definition}
\label{d:varphi-xi-r}
Let $x_{0} \in \R^{n}$ and $r_{x_{0}}>0$ as in Remark~\ref{r:rx0}, and $r>0$. For every $\xi \in {\rm B}_{2}(0)$,  and every $x \in {\rm B}_{\frac{r_{x_{0}}}{ r}} (0)$ we define
\begin{align}
\label{e:varphi-xi-r}
\varphi_{\xi,x_0,r}(x) &  :=r^{-1}(\varphi_{ \xi, x_{0}} (rx) - x_0)\,,\\
\label{e:varphi-xi-r-2}
\Phi_{x_{0}, r} (x, \xi) &:= \big( \varphi_{\xi,x_0,r}(x) , \xi \big)\,.
\end{align}
\end{definition}

We prove a basic convergence property of~$\Phi_{ x_0, r}$.

\begin{lemma}
\label{l:curvpro1.3}
For every $x_{0} \in \Om$ it holds
\begin{align}
    \label{e:curvpro1.3}
  &  \Phi_{x_0,r} \to  id  
\end{align}
in $C^{\infty}_{loc} (\R^{n} \times {\rm B}_{2}(0) ; \R^{n}\times \R^{n})$ as $r\searrow0$.
\end{lemma}

\begin{proof}
Let $r_{x_{0}}>0$ be as in Remark~\ref{r:rx0}. We notice that for every $\xi \in \mathbb{S}^{n-1}$ and for $r>0$ the function $u_r(t):=\varphi_{\xi,x_0,r}(y+t\xi)$ satisfies
\begin{equation}
    \label{e:curvpro12}
    \begin{cases}
    \ddot{u}_r(t) =F_{r, x_{0}} (u_{r} (t) , \dot{u}_{r} (t)) &  t \in \mathbb{R}\,,\\
     u_r(0)=y  \,,\\
    \dot{u}_r(0)=\xi \,,
    \end{cases}
\end{equation}
where  $F_{r, x_{0}} (x, v):= rF(rx + x_{0} , v)$  for every $(x,v) \in \mathbb{R}^n \times \mathbb{R}^n$. Denoting by
\begin{align*}
& U_{r}:= \big \{ (t, y, \xi) \in \R \times \R^{n} \times \mathbb{S}^{n-1}:  y \in \xi^{\bot} \text{ and system~\eqref{e:curvpro12} admits solution }
\\
&
\qquad \qquad \qquad \qquad \qquad \qquad \qquad \quad \text{in~$[-|t|, |t|]$ with initial conditions~$(y, \xi)$} \big\}\,,
\\
& U_{\infty}:= \{ (t, y, \xi) \in \R \times \R^{n} \times \mathbb{S}^{n-1}: \, y \in \xi^{\bot}\}\,,
\end{align*}
we have that $U_{r} \nearrow U_{\infty}$ as $r \searrow 0$. Let $G_r \colon U_{r} \to \R^{n}$ be the flow relative to system~\eqref{e:curvpro12}, i.e., the map that to each $(t, y, \xi) \in U_{r}$ associates~$v_{r}(t)$, where $v_{r}$ is the unique solution of~\eqref{e:curvpro12} with initial data $(y, \xi)$. Since $F \in C^\infty(\mathbb{R}^n \times \mathbb{R}^n; \R^{n})$, $F_{r, x_{0}} \to 0$ in $C^{\infty}_{loc}(\mathbb{R}^n \times \mathbb{R}^n; \R^{n})$ as $r \searrow 0$. Thus, by the continuous and differentiable dependence of solutions of ODEs from the data of the system, we have that 
\begin{equation}
\label{e:curvpro14}
G_{r} \to G \qquad \text{ in $C^{\infty}_{loc} (U; \R^{n})$ as $r \searrow 0$,}
\end{equation}
 where $G \colon U \to \R^{n}$ is defined as $G(t, y, \xi) := y + t\xi$.  For every $x \in {\rm B}_{\frac{r_{x_{0}}}{r}} (0)$ we can write
\begin{equation}
\label{e:curvpro13}
\Phi_{x_{0}, r} (x, \xi) = \big( \varphi_{\xi,x_0,r}(x) , \xi \big) = \big( G_r(\xi \cdot x, \pi_\xi(x),\xi), \xi\big) \,.
\end{equation}
From \eqref{e:curvpro14} and \eqref{e:curvpro13} we immediately deduce~\eqref{e:curvpro1.3}.
\end{proof}

\begin{corollary}
\label{c:curvpro1.3}
For every $x_{0} \in \Om$ there exists $R_{x_{0}}>0$ such that for every $\xi \in {\rm B}_{2}(0)$ the map $\varphi_{\xi, x_{0}}^{-1} \lfloor {\rm B}_{R_{x_{0}}} (x_{0})$ is a diffeomorphism with its image.
\end{corollary}

\begin{proof}
The thesis follows directly from Lemma~\ref{l:curvpro1.3} and from Definition~\ref{d:varphi-xi-r} (see~\eqref{e:varphi-xi-r}).
\end{proof}

Corollary~\ref{c:curvpro1.3} justifies the following definition of~$P_{\xi, x_{0}}$ for $x_{0} \in \Om$ and~$\xi \in \mathbb{S}^{n-1}$.

\begin{definition}
\label{d:P-xi}
Let $x_{0} \in \Om$, $\xi \in {\rm B}_{2}(0)$, and $R_{x_{0}}>0$ be as in Corollary~\ref{c:curvpro1.3}. We define the map $P_{\xi,x_0} \colon {\rm B}_{R_{x_{0}}} (x_{0})  \to \xi^\bot$ as
\begin{equation}
    \label{e:curvproj}
    P_{\xi,x_0}:=\pi_\xi \circ \varphi_{\xi, x_{0}} ^{-1},
\end{equation}
where $\pi_\xi \colon \mathbb{R}^n \to \xi^\bot$ denotes the orthogonal projection onto the orthogonal to $\xi$. Furthermore, for $r>0$ we define $P_{\xi, x_{0}, r}\colon {\rm B}_{\frac{R_{x_{0}}}{r}} (0) \to \xi^{\perp}$ as
\begin{equation}
    \label{e:curvproj-2}
    P_{\xi,x_0, r}:=\pi_\xi \circ \varphi_{\xi, x_{0},r} ^{-1}.
\end{equation}
\end{definition}

%

The rest of this section is devoted to prove that, up to taking a smaller~$R_{0} \in (0, R_{x_{0}}]$ depending only on~$x_{0}$,~$(P_{\xi, x_{0}})_{\xi \in \mathbb{S}^{n-1}}$ is a family of curvilinear projections over~${\rm B}_{R_{0}} (x_{0})$ (see Theorem~\ref{p:curvpro}). To this aim, we start by showing that $(\pi_{\xi})_{\xi \in \mathbb{S}^{n-1}}$ is transversal in~$\R^{n}$.

\begin{proposition}
\label{p:trapro}
 The family of orthogonal projections  $\pi_{\xi} \colon \mathbb{R}^n \to \xi^\bot$ with $\xi \in \mathbb{S}^{n-1}$ is transversal. 
\end{proposition}
\begin{proof}
We only have to check \eqref{hp:H2} of Definition~\ref{d:transversal}. For this purpose we first claim that for every $x,x' \in \R^{n}$ with $x\neq x'$, every $\xi \in \mathbb{S}^{n-1}$, and every $C\in (0,1)$ it holds
\begin{equation}
\label{e:trapro1}
\frac{|\pi_\xi(x)-\pi_\xi(x')|}{|x-x'|} \leq C \ \ \text{ implies } \ \ \bigg|\text{J}_\xi \frac{\pi_\xi(x)-\pi_\xi(x')}{|x-x'|}\bigg| \geq (1-C^2)^{\frac{n-1}{2}}.
\end{equation}
Since for every $\mathrm{O} \in SO(n)$ we have 
\begin{equation}
\label{e:trapro4}
\pi_\xi(\mathrm{O}x)=\mathrm{O} \, \pi_{\mathrm{O}^t\xi}(x) \qquad  \text{for } x \in \R^{n}\,,
\end{equation}
condition \eqref{e:trapro1} is invariant by rotation. We can thus reduce ourselves to prove \eqref{e:trapro1} for $\xi=e_1$. Moreover, being $\xi \mapsto x -(x \cdot \xi)\xi$ a $C^\infty$-map on~$\mathbb{R}^n$ and being $\pi_\xi(x) = x -(x \cdot \xi)\xi$ for $\xi \in \mathbb{S}^{n-1}$, the tangential jacobian appearing in~\eqref{e:trapro1} can be computed as the jacobian of the map
\[
f_{xx'}(\eta) := \frac{x -(x \cdot (e_1+\eta))(e_1 +\eta)-[x' -(x' \cdot (e_1+\eta))(e_1 +\eta)]}{|x-x'|} \qquad \eta \in e_1^\bot
\]
evaluated at $\eta=0$. The $(j-1)$-th column of the $n \times (n-1)$ matrix representing the differential of $f_{xx'}$ at~$0$ with respect to the basis $\{e_1,\dotsc,e_n\}$ can be explicitly written as
\[
-\frac{(x_j-x'_j)e_1 + (x_1-x'_1)e_j }{|x-x'|}\qquad \text{for $j=2,\dotsc,n$.}
\]
A direct computation shows that 
\begin{equation}
\label{e:trapro3}
\text{J}_\eta f_{xx'}(0)=(-1)^{n-1}\frac{(x_1-x'_1)^{n-1}}{|x-x'|^{n-1}} \qquad \text{ for $x,x' \in \mathbb{R}^n$,  $x \neq x'$.} 
\end{equation}
 The invariance under rotation of~\eqref{e:trapro4} allow us to deduce from~\eqref{e:trapro3} that for every $x, x' \in \R^{n}$ with $x \neq x'$ and every $\xi \in \mathbb{S}^{n-1}$
\begin{equation*}
\text{J}_\xi \frac{\pi_\xi(x)-\pi_\xi(x')}{|x-x'|} = (-1)^{n-1} \frac{((x -x') \cdot \xi)^{n-1}}{|x-x'|^{n-1}}\,. 
\end{equation*}
Since $|\pi_\xi(x)-\pi_\xi(x')|/|x-x'| \leq C$ implies $|(x -x') \cdot \xi|/|x-x'| \geq \sqrt{1-C^2}$,
we conclude~\eqref{e:trapro1}. 

For $i=1,\dotsc,n$ we recall the notation $S_i:=\{\xi \in \mathbb{S}^{n-1} :\, |\xi\cdot e_i| \geq 1/\sqrt{n}  \}$. For every~$\xi \in S_i$ and $x,x' \in \Omega$ with $x \neq x'$ let
    \begin{equation*}
        T_{xx'}(\xi) := \frac{\pi_{e_i} \circ \pi_\xi(x) - \pi_{e_i} \circ \pi_\xi(x')}{|x-x'|}.
    \end{equation*}
Since $|\xi \cdot e_i| \geq 1/\sqrt{n}$ for $\xi \in S_{i}$, there exists a constant $c(n) \in (0,1)$ such that for every~$\xi \in S_i$ and for every $z,z' \in \xi^\bot$, 
    \[
    |\pi_{e_i}(z)-\pi_{e_i}(z')| \geq c(n)|z-z'|.
    \]
    Therefore, for every $\xi \in S_i$ and for every $x,x' \in \Om$ with $x \neq x'$ we have
    \begin{equation}
    \label{e:T-1}
    c(n)\frac{|\pi_\xi(x)-\pi_\xi(x')|}{|x-x'|} \leq |T_{xx'}(\xi)| \,.
    \end{equation}
    Moreover, since $|\text{J} (\pi_{e_i})(z)| = |\xi \cdot e_i|$ for every $\xi \in S_i$ and every $z \in \xi^\bot$, we further have that
    \begin{equation}
    \label{e:T-2}
    |\text{J}_\xi(T_{xx'}(\xi))| = |\xi \cdot e_i| \, \bigg| \text{J}_\xi \frac{\pi_\xi(x)-\pi_\xi(x')}{|x-x'|} \bigg|.
    \end{equation}
    Hence, from~\eqref{e:trapro1} and~\eqref{e:T-1}--\eqref{e:T-2} we infer that for every $x,x' \in \Omega$ with $x \neq x'$, every $\xi \in S_{i}$, and every $C \in (0,1)$:
    \begin{equation}
    \label{e:trapro2}
       |T_{xx'}(\xi)| \leq c(n)C\ \ \text{implies} \ \ |\text{J}_\xi(T_{xx'}(\xi))| \geq \frac{1}{\sqrt{n}} (1-C^2)^{\frac{n-1}{2}}. 
    \end{equation}
    Set $g_1(C):=c(n) C$, $g_2(C):= (1-C^2)^{\frac{n-1}{2}}/\sqrt{n}$, and $g(C):= g_1(C)-g_2(C)$. Since~$g$ is continuous, $g(0)<0$, and $g(1)>0$, we deduce that there exists $C_0 \in (0,1)$ such that $g(C_0)=0$. Finally, condition \eqref{hp:H2} of Definition~\ref{d:transversal} follows by setting $C:=C_0$ in~\eqref{e:trapro2}.
\end{proof}

The next lemma is a direct consequence of Definition~\ref{d:P-xi} and of Lemma~\ref{l:curvpro1.3}.

\begin{lemma}
\label{l:curvpro1.3-2}
For every $x_{0} \in \Om$ it holds
\begin{align}
    \label{e:curvpro1.3P}
  & P_{\xi,x_0,r} \to \pi_{\xi}
\end{align}
 in $C^{\infty}_{loc} ( \R^{n} \times {\rm B}_{2}(0) ; \R^{n})$ as $r\searrow0$, where the maps in~\eqref{e:curvpro1.3P} are intended as functions of~$(x, \xi)$.
\end{lemma}

\begin{proof}
It is enough to combine the definition of~$P_{\xi, x_{0}, r}$ given in~\eqref{e:curvproj-2} with the convergence shown in Lemma~\ref{l:curvpro1.3}.
\end{proof}

In order to show that~$(P_{\xi, x_{0}})_{ \xi \in \mathbb{S}^{n-1}}$ is a family of curvilinear projections, we will make use of the functions 
\begin{align*}
    T^i_{xx'}(\xi,r) & :=\frac{\pi_{e_i} \circ P_{\xi,x_0,r}(x) - \pi_{e_i}\circ P_{\xi,x_0,r}(x')}{|x-x'|}\,,\\
    T^i_{xx'}(\xi,0)& :=\frac{\pi_{e_i} \circ \pi_\xi(x) - \pi_{e_i}\circ\pi_\xi(x')}{|x-x'|}\,.
\end{align*}
defined for every $i=1,\dotsc,n$, every $r>0$, and every $x,x' \in {\rm B}_1(0)$ with $x \neq x'$. Notice that for $r>0$ sufficiently small, $\frac{R_{x_{0}}}{r}>1$, so that the maps $T^i_{xx'}(\xi,r)$ are well defined for  $x,x' \in {\rm B}_1(0)$ with $x \neq x'$.

\begin{theorem}
\label{p:curvpro}
Let $F \in C^{\infty}( \mathbb{R}^n \times \mathbb{R}^n ; \mathbb{R}^n)$ satisfy \UUU condition~\eqref{e:quadratic}. \EEE Then, for every $x_0 \in \Omega$ there exists $R_0>0$ such that the family of maps $\{ P_{\xi,x_0} \colon {\rm B}_{R_0}(x_0) \to \xi^\bot: \xi \in \mathbb{S}^{n-1}\}$ is a family of curvilinear projections on~${\rm B}_{R_{0}}(x_{0})$. Moreover, for every $k \in \mathbb{N}$, every $R>0$, and every $\epsilon>0$ there exists~$r_{\epsilon}>0$ such that 
\begin{equation}
\label{e:curvpro1.1}
   \|T^i_{xx'}(\cdot,r_{\epsilon}) -T^i_{xx'}(\cdot,0)\|_{C^{k} ({\rm B}_{2}(0))} \leq \epsilon \qquad \text{for $x,x' \in {\rm B}_1(0)$ with $x \neq x'$.}
\end{equation} 
\end{theorem}

\begin{proof}
To shorten the notation we further set $f_r(x,\xi):=\pi_{e_i} \circ P_{\xi,x_0,r}(x)$ and $f(x,\xi):=\pi_{e_i} \circ \pi_\xi(x)$. In view of Lemma~\ref{l:curvpro1.3-2} we have that
\begin{equation}
    \label{e:curvpro1.2}
    f_{r}(\cdot,\cdot) \to  f(\cdot,\cdot) \qquad \text{in $C^\infty_{loc}(\mathbb{R}^n \times {\rm B}_{2}(0); \R^{n})$ as $r \searrow 0$.}
\end{equation}
Let us denote by $\alpha$ a generic multi-index of order $k$ and by $\partial^\alpha_{\xi}$ the associated partial derivative. Then, for every $\xi \in {\rm B}_{2}(0)$, every $r>0$, and every $x, x' \in {\rm B}_{1}(0)$ with $x \neq x'$ we have that
\begin{align*}
    |\partial^{\alpha}_{\xi} T^i_{xx'}(\xi,r) - \partial^{\alpha}_{\xi} T^i_{xx'} (\xi,0)| &\leq \bigg| \frac{\partial^{\alpha}_{\xi} f_r( x, \xi ) -\partial^{\alpha}_{\xi} f_r ( x', \xi )}{|x-x'|} -D_x(\partial^{\alpha}_{\xi}f_r (x, \xi )) \cdot \frac{x'-x}{|x-x'|} \bigg| \\
    &\quad + \bigg| D_x(\partial^{\alpha}_{\xi} f_r (x, \xi )) \cdot \frac{x'-x}{|x-x'|} - D_x(\partial^{\alpha}_{\xi} f (x, \xi )) \cdot \frac{x'-x}{|x-x'|}\bigg| \\
    &\quad + \bigg|\frac{\partial^{\alpha}_{\xi}f(x, \xi ) -\partial^{\alpha}_{\xi}f(x', \xi)}{|x-x'|}-  D_x(\partial^{\alpha}_{\xi}f(x, \xi)) \cdot \frac{x'-x}{|x-x'|}\bigg| \\
    &\leq \|D^2_x(\partial^{\alpha}_{\xi} f_r)\|_{L^\infty({\rm B}_{1}(0) \times {\rm B}_{2}(0))}|x-x'| 
    \\
    &\quad+ \|D_x(\partial^{\alpha}_{\xi}f_r) - D_x(\partial^{\alpha}_{\xi}f)\|_{L^\infty({\rm B}_{1}(0) \times {\rm B}_{2}(0))}\\
    &\quad +\|D^2_x(\partial^{\alpha}_{\xi} f)\|_{L^\infty({\rm B}_{1}(0) \times {\rm B}_{2}(0))} |x-x'|.
    \end{align*}
The previous chain of inequalities together with~\eqref{e:curvpro1.2} implies that for every $\epsilon>0$ there exist $r_\epsilon, \rho_{\epsilon} >0$ (depending also on $\Xi$ and $k$) such that $|\partial^{\alpha}_{\xi} T^i_{xx'}(\xi,r_\epsilon) - \partial^{\alpha}_{\xi} T^i_{xx'}(\xi,0)| \leq \epsilon$ whenever $|x-x'|\leq \rho_\epsilon$. For points $x,x' \in {\rm B}_1(0)$ with $|x-x'|>\rho_\epsilon$ we use instead the following bound
\begin{align*}
    |\partial^{\alpha}_{\xi} T^i_{xx'}(\xi,r) - \partial^{\alpha}_{\xi} T^i_{xx'}(\xi,0)| &\leq \frac{2 \|\partial^{\alpha}_{\xi} f_r -\partial^{\alpha}_{\xi} f\|_{L^\infty ({\rm B}_{1}(0) \times {\rm B}_{2}(0)) }}{|x-x'|} 
    \\
    &
    \leq 2\rho^{-1}_\epsilon \|\partial^{\alpha}_{\xi} f_r -\partial^{\alpha}_{\xi} f\|_{L^\infty({\rm B}_{1}(0) \times {\rm B}_{2}(0)) }.
    \end{align*}
Thus, by eventually choosing a smaller $r_\epsilon$ we also have $|\partial^{\alpha}_{\xi} T^i_{xx'}(\xi,r_\epsilon) - \partial^{\alpha}_{\xi} T^i_{xx'}(\xi,0)| \leq \epsilon$ for every $\xi \in {\rm B}_{2}(0)$ and every $x,x' \in {\rm B}_1(0)$ with $|x-x'|>\rho_\epsilon$. This implies~\eqref{e:curvpro1.1}.

The arbitrariness of $\epsilon>0$ together with Remark~\ref{r:'} and Proposition~\ref{p:trapro} provides $R_0>0$ for which the family $(P_{\xi,x_0,R_0})_{\xi \in \mathbb{S}^{n-1}}$ is transversal on ${\rm B}_1(0)$. Since
\begin{equation*}
P_{\xi,x_0}(x)=R_0 P_{\xi,x_0,R_0}\bigg(\frac{x-x_0}{R_0}\bigg) \qquad \text{for $x \in {\rm B}_{R_0}(x_0)$,}
\end{equation*}
we deduce that $(P_{\xi,x_0})_{\xi \in \mathbb{S}^{n-1}}$ is a transversal family on ${\rm B}_{R_0}(x_0)$. Furthermore, $(P_{\xi,x_0})_{\xi \in \mathbb{S}^{n-1}}$ is parametrized by $\varphi_{x_{0}}$ by construction (cf.~Definitions~\ref{d:curvpro} and~\ref{d:P-xi}, Remark~\ref{r:Axi}, and Corollary~\ref{c:curvpro1.3}), so that conditions (1)--(3) of Definition~\ref{d:CP} are satisfied. Property~(4) follows directly from the convergence~\eqref{e:curvpro1.3}.
\end{proof}

\section{Proofs of Theorem~\ref{t:int1} and of Corollary~\ref{c:int2}}
\label{s:directional}

Throughout this section we assume that $\Omega$ is an open subset of $\mathbb{R}^n$ and we fix $F \in C^{\infty} ( \R^{n} \times \R^{n}; \R^{n})$ fulfilling~ \eqref{e:quadratic} and a map $g \colon \Omega \times \mathbb{R}^n \to \mathbb{R}^m$ satisfying~\eqref{G1}. Furthermore, we rely on the notation introduced in Section \ref{s:curvilinear}. We start by recalling the definition of jump set.

\begin{definition}
Let $u \colon \Omega \to \mathbb{R}^m$ be measurable, then $x \in \Omega$ belongs to $J_u$ if and only if there exists $(u^+(x),u^-(x),\nu(x)) \in \mathbb{R}^m \times \mathbb{R}^m \times \mathbb{S}^{n-1}$ such that
\[
\aplim_{\substack{z \to x \\ \pm(z-x) \cdot \nu (x) >0}} u(z)=u^\pm(x).
\]
\end{definition}

Here we present a fundamental property of the jump set (cf.~\cite{del}).

\begin{theorem}
\label{t:delnin}
Let $u \colon \Omega \to \mathbb{R}^n$ be measurable. Then,~$J_u$ is countably $(n-1)$-rectifiable.
\end{theorem}

We continue our analysis with two measurability results. Before \UUU doing this, \EEE it is convenient to introduce the notion of \emph{directional jump set}. 

\UUU
\begin{definition}
\label{d:dirjump}
Let $(P_\xi)_{\xi \in \mathbb{S}^{n-1}}$ be a family of curvilinear projections on $\Omega$ and let $u \colon \Omega \to \mathbb{R}^m$ be a measurable function. We introduce the directional jump set of $u$ in the direction~$\xi$ as
\begin{equation*}
   J_{\hat u_\xi}:=\{x \in \Omega :\, t^{\xi}_{x} \in J_{\hat{u}^\xi_{P_\xi(x)}}\}\,.
\end{equation*}
In addition, we define the collection of all directional jump sets of~$u$ by letting~$\xi$ varies in $\mathbb{S}^{n-1}$ as a subset of the product space $\Omega \times \mathbb{S}^{n-1}$ as follows
\begin{equation}
\label{e:Au}
    A_{\hat{u}} := \{ (x , \xi ) \in \Om \times \mathbb{S}^{n-1}: \, x \in J_{\hat{u}_{\xi}}\}\,.
\end{equation}
\end{definition}
\EEE






%

\begin{lemma}
\label{l:meas10000}
Let $u \colon \Omega \to \mathbb{R}^m$ be Borel measurable and let $(P_\xi)_{\xi \in \mathbb{S}^{n-1}}$ be a family of curvilinear projections on $\Omega$. Then, the sets~$J_{\hat u_\xi}$ and~$ A_{\hat{u}}$
are Borel measurable.
\end{lemma}

\begin{proof}
We limit ourselves to prove that $A_{\hat{u}}$ is Borel since the remaining assertion can be proved in the same way.  By replacing the function~$u_{\xi}$ with \UUU $v_{\xi} := \arctan (u_{\xi})$, \EEE we may further work with bounded functions. Hence, we define $s^\pm \colon \Omega \times \mathbb{S}^{n-1} \to \mathbb{R}$ and $i^\pm \colon \Omega \times \mathbb{S}^{n-1} \to \mathbb{R}$ as 
\begin{align*}
s^\pm(x,\xi) & := \limsup_{r \searrow 0} \mint_{0}^{r} \UUU v_{\xi} \EEE ( \varphi_{\xi} ( P_{\xi} (x) \pm t\xi ) ) \, \di t \,,
\\
i^\pm(x,\xi) & := \liminf_{r \searrow 0} \mint_{0}^{r} \UUU v_{\xi} \EEE ( \varphi_{\xi} ( P_{\xi} (x) \pm t\xi ) ) \, \di t \,.
\end{align*}
We notice that~$s^{\pm}$ and~$i^{\pm}$ are Borel measurable in~$\Om \times \mathbb{S}^{n-1}$. Indeed, the integrand functions are Borel measurable in the triple~$(x, \xi, t)$. Hence, an application of Fubini's theorem implies that the integral functions are Borel measurable in~$\Om \times \mathbb{S}^{n-1}$. Finally, both liminf and limsup can be computed by restricting $r \in \mathbb{Q}$ because of the continuity of the integrals with respect to~$r$. Since it holds that
\begin{align*}
A_{\hat{u}} = \Big\{(x,\xi) \in \Omega \times \mathbb{S}^{n-1} \, : & \, s^+(x,\xi)=i^+(x,\xi), \ s^-(x,\xi)=i^-(x,\xi),
\\
&
s^+(x,\xi)\neq s^-(x,\xi),\ s^{\pm} (x, \xi) \in \Big(-\frac{\pi}{2} , \frac{\pi}{2} \Big) \Big\}\,,
 \end{align*}
we immediately infer the Borel measurability of~$A_{\hat{u}}$.
\end{proof}


\begin{lemma}
\label{l:meas1000}
Let $u \colon \Omega \to \mathbb{R}^m$ be $\mathcal{L}^n$-measurable and let $(P_\xi)_{\xi \in \mathbb{S}^{n-1}}$ be a family of curvilinear projections on $\Omega$. Then, for every $B \in \mathcal{B} ( \mathbb{R}^n)$ and $A \in \mathcal{B} ( \mathbb{R}^n \times \mathbb{S}^{n-1})$ we have that
\begin{align}
\label{e:meas1000.1}
    &y \mapsto \sum_{ t \in B^\xi_y } (|[\hat{u}^\xi_y(t)]| \wedge 1)\qquad  \text{ is $\mathcal{H}^{n-1}$-measurable} \\
    \label{e:meas1000.2}
    &\xi \mapsto \int_{\xi^\bot} \sum_{ t \in (A_\xi)^\xi_y } \big ( |[\hat{u}^\xi_y(t)]| \wedge 1 \big) \, \di \mathcal{H}^{n-1}(y) \qquad\text{ is $\mathcal{H}^{n-1}$-measurable}.
\end{align}
\end{lemma}

\begin{proof}
We focus on~\eqref{e:meas1000.2} since the measurability of~\eqref{e:meas1000.1} can be obtained repeating the same argument with obvious modification. Given $\delta>0$ we consider $N_\delta \in \mathbb{N}$ and $\{\eta_1,\dotsc,\eta_{N_\delta}\}\subseteq \mathbb{S}^{n-1}$ for which the sets
\[
\Sigma_i:=\{\xi \in \mathbb{S}^{n-1} \, : \, \xi \cdot \eta_i \geq (1-\delta) \} \qquad  \text{for } i=1,\dotsc,N_\delta
\]
form a covering of~$\mathbb{S}^{n-1}$. We notice that we can reduce ourselves to show the measurability of~\eqref{e:meas1000.2} when restricted to $\Sigma_i$ for every $i=1,\dotsc,N_\delta$. 

 Let $A$ be an open subset of~$\R^{n} \times \mathbb{S}^{n-1}$ and $\varphi \colon A \to \R^{n}$ be the parametrization of~$(P_\xi)_{\xi \in \mathbb{S}^{n-1}}$, according to Definition~\ref{d:param}. In particular, there exist $\rho, \tau>0$ such that for every $\xi \in \mathbb{S}^{n-1}$ the map $\varphi_{\xi} = \varphi(\cdot, \xi) \colon \{ y + t\xi: \, (y, t) \in [  \xi^{\bot} \cap {\rm B}_{\rho} (0)] \times (-\tau, \tau)\} \to \R^{n}$ is a parametrization of~$P_{\xi}$. For every $\xi \in \Sigma_i$ we consider the map $P^i_\xi:= \pi_{\eta_i} \circ P_\xi$. Since for every $\xi \in \Sigma_i$ the map $\pi_{\eta_i} \restr \xi^\bot$ is an isomorphism between $\xi^\bot$ and $\eta_i^\bot$, the family $(P^i_\xi)_{\xi \in \Sigma_i}$ can be parametrized by 
\[
\varphi^i_{\xi} ( y + t \xi):= \varphi_{\xi}  ((\pi_{\eta_i}\restr \xi^\bot)^{-1}(y) +t\xi), \ \text{for } (y,t) \in [\eta_i^\bot \cap \mathrm{B}_{\tilde{\rho}}(0)] \times (-\tau,\tau),
\]
where  $\tilde{\rho} > 0$ is such that $\eta_i^\bot \cap \mathrm{B}_{\tilde{\rho}}(0) \subset \pi_{\xi}(\xi^\bot \cap \mathrm{B}_{\rho}(0))$ for every $\xi \in \Sigma^i$. Up to consider a smaller $\delta>0$ we may also assume that $\Omega \subseteq \text{Im}(\varphi^i_\xi)$ for every $i=1,\dotsc, N_\delta$. With this choice  we have that the family $(P^i_{\xi})_{\xi \in \Sigma_i}$ is parametrized on $\Omega$ for every $i=1,\dotsc,N_\delta$. The convenience in doing this relies on the fact that we are now working with a family of maps taking values in a common space. For the rest of the proof we drop the dependence from the index $i$.

Under this assumption the function in \eqref{e:meas1000.2} restricted to $\Sigma$ coincides with
\begin{equation}
\label{e:meas1000.2.1}
 \xi \mapsto \int_{\eta^\bot} \sum_{t \in (A_\xi )^\xi_y} \big ( |[\hat{u}^\xi_y(t)]| \wedge 1 \big) \, \di \mathcal{H}^{n-1}(y),
\end{equation}
where the slicing of sets and functions are now considered with respect to the new family $(P_\xi)_{\xi \in \Sigma}$ accordingly to Definition \ref{d:slices}. Notice that the $\mathcal{H}^{n-1}$-equivalence class of the function in \eqref{e:meas1000.2.1} does not depend on the Lebesgue equivalence class of $u$. Therefore, we may assume with no loss of generality that $u$ is Borel measurable. As in the proof of Lemma \ref{l:meas10000}, letting $s^\pm \colon \Omega \times \Sigma \to \overline{\mathbb{R}}$ and $i^\pm \colon \Omega \times \Sigma \to \overline{\mathbb{R}}$ be defined as
\[
s^\pm(x,\xi):= \aplims_{s \to t^\pm_x}  \, \hat{u}^\xi_{P_\xi(x)}(s), \qquad \text{and} \qquad i^\pm(x,\xi):= \aplimi_{s \to t^\pm_x} \,  \hat{u}^\xi_{P_\xi(x)}(s)
\]
and arguing as in the proof of Proposition~\ref{p:prodmeas} we infer that~$s^\pm$ and~$i^\pm$ are Borel measurable functions. Using the identity
\[
|[\hat{u}^\xi_{y}(t)]|= |s^+(\varphi_\xi(y+t\xi),\xi)-s^-(\varphi_\xi(y+t\xi),\xi)|,
\]
 whenever $(y,\xi,t) \in \eta^\bot \times \Sigma \times \mathbb{R}$ are such that $s^+(\varphi_\xi(y+t\xi),\xi)=i^+(\varphi_\xi(y+t\xi),\xi) \in \R$, $s^-(\varphi_\xi(y+t\xi),\xi)=i^-(\varphi_\xi(y+t\xi),\xi) \in \R$, and $t \in \Omega^\xi_{y}$. Therefore, from the Borel measurability of the maps $s^\pm$ and $i^\pm$ we infer the Borel measurability of $j \colon \eta^\bot  \times \Sigma \times \mathbb{R} \to \mathbb{R}$ defined as
 \[
  j(y,\xi,t):=
 \begin{cases}
 |[\hat{u}^\xi_{y}(t)]| \wedge 1, &\text{ if }(y,\xi,t) \in \eta^\bot \times \Sigma \times \mathbb{R} \text{ and } t \in \Omega^\xi_{y} \\
 0 &\text{ otherwise on } \eta^\bot \times \Sigma \times \mathbb{R}. 
 \end{cases}
 \]
  Now fix $A \in \mathcal{B}(\Omega \times \Sigma)$. By construction we have for every $\xi \in \Sigma$ 
 \[
 \int_{\eta^\bot} \sum_{t \in (A_\xi )^\xi_y} \big ( |[\hat{u}^\xi_y(t)]| \wedge 1 \big) \, \di \mathcal{H}^{n-1}(y)= \int_{\eta^\bot} \sum_{t \in (A_\xi )^\xi_y} j(y,\xi,t) \, \di \mathcal{H}^{n-1}(y).
 \]
 
  For every $m=1,2,\dotsc$ and $k=0,1,\dotsc$ let $A^m_k$ be the Borel subset of $\Omega \times \Sigma$ defined by $(\psi \circ j^{-1})([k2^{-m},(k+1)2^{-m}))$ where $\psi \colon \eta^\bot \times \Sigma \times \mathbb{R} \to \Omega \times \Sigma$ is defined as $\psi(y,\xi,t):= (\varphi_\xi(y+t\xi),\xi)$. Notice that 
  \begin{equation}
  \label{e:meas1000.3}
  \text{$(x,\xi) \in A^m_k$ \ and \ $\varphi_{\xi}(y+t\xi)=x$ implies \ $0 <j(y,\xi,t)-k2^{-m} \leq 2^{-m}$}. 
  \end{equation}
  Furthermore, we consider a sequence of countable Borel partitions of $A \cap A_{\hat{u}}$, say $(\mathcal{B}_m)_m$, such that $B' \in \mathcal{B}_m$ implies $\text{diam}(B') \leq 2^{-m}$ for every $m =1,2,\dotsc$. We claim that 
  \begin{equation}
  \label{e:meas1000.5}
  \sum_{k=1}^\infty \sum_{B' \in \mathcal{B}_m} k2^{-m}\mathbbm{1}_{P_\xi((A^m_k \cap B')_\xi)}(y)\nearrow \sum_{t \in (A_\xi)^\xi_y} j(y,\xi,t), \ \ \text{for $(y,\xi) \in \eta^\bot \times \Sigma$}
  \end{equation}
  as $m \to \infty$. The fact that the sequence is monotonically increasing follows by construction. Moreover, letting $s(\xi,y)$ coincide with $\sum_{t \in (A_\xi)^\xi_y} j(y,\xi,t)$, it is not difficult to show that for a fixed couple $(y,\xi)$ and for a given positive integer $M$ there exists a finite subset of $(A_\xi \cap J_{\hat{u}_\xi})^\xi_y = ((A \cap A_{\hat{u}})_\xi)^\xi_y$, say $\tilde{A}^\xi_y$, such that
  \[
  \begin{cases}
  |\sum_{t \in (A_\xi)^\xi_y} j(y,\xi,t) - \sum_{t \in \tilde{A}^\xi_y} j(y,\xi,t)| \leq M^{-1}, &\text{ if } s(\xi,y) < \infty\\
  \sum_{\tilde{A}^\xi_y} j(y,\xi,t) \geq M, &\text{ otherwise}.
  \end{cases}
  \]
  Letting 
  \[
  d:= \min_{\substack{t,t' \in \tilde{A}^\xi_y \\ \ t\neq t'}} |\varphi_\xi(y+t\xi)-\varphi_\xi(y+t'\xi) |,
  \]
  if $m$ is such that $2^{-m} < d$, then, using also that $(A^m_k \cap B')_{k,B'}$ forms a partition, to every $t \in \tilde{A}^\xi_y$ we can injectively associate $B' \in \mathcal{B}_m$ for which $\psi(y,\xi,t) \in (A^m_k \cap B')_\xi \times \{\xi \}$ for some $k=0,1,\dotsc$. Using also \eqref{e:meas1000.3} we infer
  \[
  \begin{split}
  \lim_{m \to \infty} &\sum_{k=1}^\infty \sum_{B' \in \mathcal{B}_m} k2^{-m}\mathbbm{1}_{P_\xi((A^m_k \cap B')_\xi)}(y) \geq \sum_{t \in \tilde{A}^\xi_y}(1-2^{-m}) j(y,\xi,t)\\
  &\geq \sum_{t \in (A_\xi)^\xi_y} j(y,\xi,t) -2^{-m}s(\xi,y) - M^{-1},
  \end{split}
  \]
  whenever $s(\xi,y) < \infty$ and 
  \[
  \lim_{m \to \infty} \sum_{k=1}^\infty \sum_{B' \in \mathcal{B}_m} k2^{-m}\mathbbm{1}_{P_\xi((A^m_k \cap B')_\xi)}(y) \geq \sum_{t \in \tilde{A}^\xi_y}(1-2^{-m}) j(y,\xi,t) \geq  (1-2^{-m})M,
  \]
otherwise. Thanks to the arbitrariness of $M$ we finally deduce for every $(y,\xi) \in \eta^\bot \times \Sigma$
\begin{equation}
\label{e:meas1000.4}
     \lim_{m \to \infty} \sum_{k=1}^\infty \sum_{B' \in \mathcal{B}_m} k2^{-m}\mathbbm{1}_{P_\xi((A^m_k \cap B')_\xi)}(y) \geq \sum_{t \in (A_\xi)^\xi_y} j(y,\xi,t),
\end{equation}
whenever $s(\xi,y) < \infty$ and
\begin{equation*}
     \lim_{m \to \infty} \sum_{k=1}^\infty \sum_{B' \in \mathcal{B}_m} k2^{-m}\mathbbm{1}_{P_\xi((A^m_k \cap B')_\xi)}(y) = \infty,
\end{equation*}
otherwise. In order to prove the opposite inequality, we simply observe that, since $(A^m_k \cap B')_{k,B'}$ forms a partition, given $m=1,2,\dotsc$, then for every $k=0,1,\dotsc$ and $B' \in \mathcal{B}_m$ for which  $\mathbbm{1}_{P_\xi((A^m_k \cap B')_\xi)}(y)=1$ we can injectively associate $t \in (A_\xi \cap J_{\hat{u}_\xi})^\xi_y=((A \cap A_{\hat{u}})_\xi)^\xi_y$ such that $\psi(y,\xi,t) \in (A^m_k \cap B')_\xi \times \{ \xi\}$; thanks to \eqref{e:meas1000.3} such a $t$ satisfies $k2^{-m} < j(y,\xi,t)$ from which we immediately deduce the validity of the opposite inequality in \eqref{e:meas1000.4}. Our claim is thus proved. 
  
  To conclude, we define the continuous map $P \colon \Omega \times \Sigma \to \eta^\bot \times \Sigma$ as $P(x,\xi):=(P_\xi(x),\xi)$ and notice that 
  \[
  \mathbbm{1}_{P_\xi((A^m_k \cap B')_\xi)}(y) = \mathbbm{1}_{P(A^m_k \cap B')}(y,\xi), \ \ \text{for every $(y,\xi) \in \eta^\bot \times \Sigma$}.
  \]
  Therefore, we are in position to apply the measurable projection theorem \cite[2.2.13]{fed}, and infer the $(\mathcal{H}^{n-1} \restr \eta^\bot  \otimes \mathcal{H}^{n-1} \restr \mathbb{S}^{n-1})$-measurability of $(y,\xi) \mapsto \mathbbm{1}_{P_\xi((A^m_k \cap B')_\xi)}(y)$. We can thus integrate both sides of \eqref{e:meas1000.5} with respect to $\mathcal{H}^{n-1} \restr \eta^\bot$ and apply Fubini's theorem together with the monotone convergence theorem for every $\xi \in \Sigma$ to finally infer the $\mathcal{H}^{n-1}$-measurability of \eqref{e:meas1000.2.1}. This concludes the proof. 
\end{proof}\EEE

\subsection{Rectifiability of the directional jump set}

Given an $\mathcal{L}^n$-measurable function $u \colon \Omega \to \mathbb{R}^m$, with the help of Lemma \ref{l:meas1000} consider for every $\xi \in \mathbb{S}^{n-1}$ the Borel regular measure $\eta_\xi$ of $\mathbb{R}^n$ given by
\begin{align}
    \label{e:defeta1}
    \eta_\xi(B) & := \int_{\xi^\bot} \sum_{ t \in B^\xi_y} \big( |[\hat{u}^\xi_y(t)]| \wedge 1 \big) \, \di \mathcal{H}^{n-1}(y) \qquad  B \in \mathcal{B} ( \mathbb{R}^n) \,,
    \\
    \label{e:defeta2}
    \eta_\xi(E) & := \inf \, \{\eta_\xi(B) : \, E \subseteq B, \ B \in \mathcal{B} (\Om)\}\,.
\end{align}

\begin{definition}
\label{d:defiu}
Let $u \colon \Omega \to \mathbb{R}^m$ be measurable and let $(P_{\xi})_{\xi \in \mathbb{S}^{n-1}}$ be a family of curvilinear projections on some open subset~$U$ of~$\Omega$, and let $(\eta_\xi)_{\xi \in \mathbb{S}^{n-1}}$ be the family of measures in~\eqref{e:defeta1}--\eqref{e:defeta2}.  Then, for  $1 \leq p \leq \infty $ we define~$\mathscr{I}_{u,p}$ as the resulting measure on~$U$ according to~\eqref{e:caratheodoryc2} and~$\hat{\mathscr{I}}_{u}$ as the resulting measure on $U \times \mathbb{S}^{n-1}$ according to~\eqref{e:caratheodoryc2.1.0}.
\end{definition}

 We show that $\mathscr{I}_{u,p}$ is concentrated on points $x \in \Omega$ such that $\mathcal{H}^{n-1}(\{\xi \in \mathbb{S}^{n-1} : [\hat{u}^\xi_{P_\xi(x)}(t_x)] \neq 0 \}) >0$.

\begin{proposition}
\label{p:keyprop}
Let $u \colon \Omega \to \mathbb{R}^m$ be measurable and let $(P_{\xi})_{\xi \in \mathbb{S}^{n-1}}$ be a family of curvilinear projections on $\Omega$. Assume that there exists $p \in (1, +\infty]$ such that $\mathscr{I}_{u, p}$ is finite. Then,
\begin{equation}
    \label{e:keyprop1000}
    \mathscr{I}_{u,1} \big ( \{x \in \Omega : \, \text{$ x \notin J_{\hat u_\xi} $ for $\mathcal{H}^{n-1}$-a.e.~$\xi \in \mathbb{S}^{n-1}$}  \} \big) = 0\,.
\end{equation}
\end{proposition}

\begin{proof}
 First of all the set defined in \eqref{e:keyprop1000} can be rewritten as 
\begin{equation}
\label{e:keyprop7000}
E:= \big\{x \in \Omega : \, \text{$t_x \notin J_{\hat{u}^\xi_{P_\xi(x)}}$ for $\mathcal{H}^{n-1}$-a.e.~$\xi \in \mathbb{S}^{n-1}$}  \big\}\,.
\end{equation}
We claim that the set $E$ does not depend on the Lebesgue representative of $u$. In order to verify this claim it is enough to prove that given $u_1,u_2$ in the same Lebesgue equivalence class, for $\mathcal{H}^{n-1}$-a.e.~$\xi$ we have that $u_1 \restr P_{\xi}^{-1}(P_\xi(x))$ and $u_2 \restr P_{\xi}^{-1}(P_\xi(x))$ are $\mathcal{H}^1$ equivalent. Let $\phi_x \colon \mathrm{B}_{\overline{r}}(x) \setminus \{x\} \to \mathbb{S}^{n-1}$ be the map given by Proposition~\ref{p:retr}. In particular, it holds that
\begin{equation}
\label{e:keyprop99999}
P_\xi(z)=P_\xi(x) \ \ \text{if and only if $\xi =\phi_x(z)$, for every $z \in {\rm B}_{\overline{r}}(x) \setminus \{x\}$}.
\end{equation}
Therefore, an application of Coarea Formula, together with the fact that estimate~\eqref{e:retr2} gives $\phi_{x} \in L^1(\mathrm{B}_{\overline{r}}(x))$, implies that 
\[
\int_{\mathbb{S}^{n-1}} \bigg( \int_{\phi^{-1}_x(\eta) \cap \mathrm{B}_{\overline{r}}(x)} |u_1-u_2| \wedge 1 \, \di \mathcal{H}^1 \bigg) \di \mathcal{H}^{n-1}(\eta)=0\,.
\]
By combining this last information with~\eqref{e:keyprop99999} we obtain the desired claim.

 Using the disintegration theorem, we write $\eta_\xi = \eta^\xi_y \otimes (P_{\xi})_{\sharp} \eta_\xi$ for a suitable family of probability measures $(\eta^\xi_y)_{y \in \xi^{\bot}}$ concentrated for $\mathcal{H}^{n-1}$-a.e.~$y \in \xi^\bot$ on the level set~$P^{-1}_{\xi}(y)$. From the definition~\eqref{e:defeta1} of~$\eta_\xi$ we deduce that
\begin{equation}
\label{e:keyprop6000}
(P_{\xi})_{\sharp} \eta_\xi \ll \mathcal{H}^{n-1} \restr \xi^{\bot} \qquad  \text{ and } \qquad \eta^\xi_y= \eta^\xi_y \restr J_{\hat u_\xi} \,,
\end{equation}
for $\mathcal{H}^{n-1}$-a.e.~$\xi \in \mathbb{S}^{n-1}$ and for $\mathcal{H}^{n-1}$-a.e.~$y \in \xi^\bot$. From~\eqref{e:keyprop6000} we deduce in particular that the measures $\eta^\xi_y$ are atomic. Furthermore, being $\mathscr{I}_{u,p}$ finite for some $1<p\leq \infty$, using Proposition~\ref{p:fproposition} we find a disintegration of $\hat{\mathscr{I}}_{u}$ of the form
\begin{equation}
\label{e:keyprop5000}
\hat{\mathscr{I}}_{u} = (f_x\, \mathcal{H}^{n-1} \restr \mathbb{S}^{n-1}) \otimes \mathscr{I}_{u,1}\, , 
\end{equation}
for a Borel measurable real-valued function $(x,\xi) \mapsto f_x(\xi)$. 
 Defining $S_\xi := \{x \in \Omega  : \,  x \notin J_{\hat{u}_\xi} \}$ and using $\eta_\xi = \eta^\xi_y \otimes (P_\xi)_\sharp \eta_\xi$ together with \eqref{e:keyprop6000}, we have that
\begin{equation}
\label{e:keyprop3000}
\eta_\xi(S_\xi) =0 \qquad \text{for $\mathcal{H}^{n-1}$-a.e. $\xi \in \mathbb{S}^{n-1}$}.
\end{equation}
Recalling the notation~\eqref{e:Au} and
\begin{align*}
(A_{\hat{u}})_\xi & =\{ x \in \Om : \, (x, \xi) \in A_{\hat{u}} \} = \{ x \in \Omega  : \, x \in J_{\hat u_{\xi}}  \} \qquad \text{for $\xi \in \mathbb{S}^{n-1}$}\,,
\\
(A_{\hat{u}})_x & = \{\xi \in \mathbb{S}^{n-1} : \, (x, \xi) \in A_{\hat{u}}\} = \{ \xi \in \mathbb{S}^{n-1}: x \in J_{\hat u_{\xi}}  \} \qquad \text{for $x \in \Om$}\,,
\end{align*}
equality~\eqref{e:keyprop3000} can be rewritten as
\begin{equation*}
\eta_\xi(\Omega\setminus (A_{\hat{u}})_\xi) =0  \qquad \text{for $\mathcal{H}^{n-1}$-a.e. $\xi \in \mathbb{S}^{n-1}$}.
\end{equation*}
By Lemma~\ref{l:meas10000} the set $A_{\hat{u}}$ is Borel. Thus, Proposition~\ref{p:coincidence} yields that 
\[
\hat{\mathscr{I}}_{u}([\Omega \times \mathbb{S}^{n-1}] \setminus  A_{\hat{u}}) = \int_{\mathbb{S}^{n-1}} \eta_\xi(\Omega \setminus (A_{\hat{u}})_\xi) \, \di \mathcal{H}^{n-1}(\xi) =0\,.
\]
Using disintegration~\eqref{e:keyprop5000} we obtain that 
\begin{equation}
\label{e:keyprop4001}
\int_{\mathbb{S}^{n-1} \setminus (A_{\hat{u}})_x} f_x(\xi) \, \di \mathcal{H}^{n-1}(\xi) =0 \qquad \text{for $\mathscr{I}_{u,1}$-a.e.~$x \in \Omega$}.
\end{equation}
We notice that the set~$E$ in~\eqref{e:keyprop1000}--\eqref{e:keyprop7000} can be rewritten as
\[
E = \{x \in \Omega : \, \mathcal{H}^{n-1}((A_{\hat{u}})_x) = 0  \}\,.
\]
Hence, by~\eqref{e:keyprop4001} for $\mathscr{I}_{u,1}$-a.e.~$x \in E$ condition $f_x\equiv 0$ is guaranteed. Since we know from Proposition \ref{p:fproposition} that $\int_{\mathbb{S}^{n-1}} f_x \, \di \mathcal{H}^{n-1}=1$ for $\mathscr{I}_{u,1}$-a.e.~$x \in \Omega$, we finally infer $\mathscr{I}_{u,1}(E)=0$. This concludes the proof of~\eqref{e:keyprop1000}.
\end{proof}

 We now give the notion of one-dimensional radial oscillation around~$x$ and rigorously define the set $\text{Osc}_{u}( \rho)$.

\begin{definition}[One dimensional radial oscillation around~$x$]
\label{d:oscillation}
Let $f \colon \mathbb{R} \to \mathbb{R}$ be measurable. We introduce the oscillation of $f$ at scale $r>0$ around the origin as
\begin{equation*}
    \text{Osc}_r(f,\rho) := \inf_{\text{Lip}(\theta)\leq 1}\int_{-\rho/4}^{\rho/4} (|f(rt)-\theta| \wedge 1) \, t^{n-1} \,  \di t \, . 
\end{equation*}
 For $\Omega \subseteq \mathbb{R}^n$ open and $u \colon \Om \to \R^{m}$ measurable, setting $\exp_{x,\xi}(t):= \exp_x(t\xi)$ and 
 \begin{displaymath}
 \mathring{u}^\xi_x (t) := u( \exp_{x,\xi}(t) ) \cdot g( \exp_{x,\xi}(t), \dot{\exp}_{x,\xi}(t))\,,
 \end{displaymath}
 we define the \emph{oscillation of $u$ around $x \in \Omega$} as
\begin{equation*}
\label{e:defosc}
   \text{Osc}(u,x,\rho):=  \limsup_{r \searrow 0}  \int_{\mathbb{S}^{n-1}} \text{Osc}_r(\mathring{u}^\xi_x,\rho) \, \di \mathcal{H}^{n-1} (\xi) \,.
\end{equation*}
\end{definition}

\begin{definition}
\label{d:oscillation2}
Given $\Omega  \subseteq \mathbb{R}^n$ open, $u \colon \Omega \to \mathbb{R}^m$ measurable, and $\rho >0$ we define 
\begin{equation*}
    \text{Osc}_u(\rho) := \{x \in \Omega : \text{Osc}(u,x,\rho)>0 \}.
\end{equation*}
\end{definition}

 We are now in position to prove Theorem~\ref{t:int1}.

\begin{proof}[Proof of Theorem \ref{t:int1}]
We observe the validity of the following implication for every $x \in \Omega$:
\begin{equation}
\label{e:rectiu1.1}
     \mathcal{H}^{n-1}(\{\xi \in \mathbb{S}^{n-1} :  x \in J_{\hat{u}_\xi} \})>0 \  \text{ implies }  x \in \text{Osc}_{u} (\rho) \,. 
\end{equation}
 Indeed, setting $\psi_x(\xi):= \xi_{\varphi}(x)/|\xi_{\varphi}(x)|$, we know from property (4) of family of curvilinear projections (see Definition~\ref{d:CP}) \EEE that the jacobian of~$\psi_x$ is bounded away from zero for every $\xi \in \mathbb{S}^{n-1}$. Thus, we can write
 \[
\int_{\mathbb{S}^{n-1}}  \text{Osc}_r(\mathring{u}^\xi_x, \rho)  \, \di \mathcal{H}^{n-1} (\xi) = \int_{\mathbb{S}^{n-1}}  \text{Osc}_r(\mathring{u}^{\psi_{x} (\eta)}_x, \rho) J_\xi\psi_x(\eta) \, \di \mathcal{H}^{n-1} (\eta)
 \]
 In addition, if we denote by $\Lambda_x := \{\xi \in \mathbb{S}^{n-1} : \, \xi = \psi_x(\eta) \text{ and } x \in J_{\hat{u}_\eta} \text{ for some }\eta \in \mathbb{S}^{n-1}\}$, condition $\xi \in \Lambda_x$ implies by Ascoli-Arzel\'a that
 \begin{equation}
 \label{e:rectiu4000.1}
  \lim_{r \searrow 0 }  \text{Osc}_r(\mathring{u}^\xi_x, \rho) >0 \,.
 \end{equation}
  Being $\psi_x(\cdot)$ a diffeomorphism of $\mathbb{S}^{n-1}$ onto itself, condition $\mathcal{H}^{n-1}(\{\xi \in \mathbb{S}^{n-1} :\, x \in J_{\hat{u}_\xi} \})>0$ implies $\mathcal{H}^{n-1}(\Lambda_x)>0$.  Therefore, inequality \eqref{e:rectiu4000.1} together with Fatou's Lemma allows us to infer 
 \begin{equation*}
     \liminf _{r \searrow 0}  \int_{\mathbb{S}^{n-1}}  \text{Osc}_r(\mathring{u}^\xi_x, \rho)  \, \di \mathcal{H}^{n-1} (\xi)  >0\,,
 \end{equation*}
 which yields $\text{Osc} (u, x, \rho) >0$ and $x \in \Omega \setminus \text{Osc}_u( \rho) $.  We have thus proved~\eqref{e:rectiu1.1}. 
 
 Eventually, we notice that \eqref{e:rectiu1.1} in combination with Proposition \ref{p:keyprop} tells us that $\mathscr{I}_{u,1}(\Omega \setminus \text{Osc}_u (\rho))=0$. This implies that $\mathscr{I}_{u, q}$ is integralgeometric for every $q \in [1, +\infty]$. Since we assumed the finiteness of~$\mathscr{I}_{u,p}$, the $(n-1)$-rectifiability of~$\mathscr{I}_{u,1}$ follows now from Theorem~\ref{t:rectheorem}.
  
 In particular, we know that there exists a countably $(n-1)$-rectifiable subset $R$ of~$\Om$ such that
\begin{equation*}
\label{e:euju1000}
   \mathscr{I}_{u,1}(\Omega \setminus R)=0\,.
\end{equation*}
This condition together with the formula for integralgeometric measures given in Proposition \ref{p:coincidence} and applied to the Borel set $\Omega \setminus R$, yields that 
\begin{equation*}
    \label{e:euju6.1}
    \eta_\xi(\Omega \setminus R)=0\qquad \text{for $\mathcal{H}^{n-1}$-a.e. $\xi \in \mathbb{S}^{n-1}$.}
\end{equation*}
From the definition of $\eta_\xi$ (cf. \eqref{e:defeta1}) we immediately infer
\begin{align*}
\label{e:euju1}
     J_{\hat u^\xi_y}  \cap \big ( \Omega^{\xi}_{y} \setminus R^{\xi}_{y} \big) =\emptyset \qquad \text{for $\mathcal{H}^{n-1}$-a.e. $\xi \in \mathbb{S}^{n-1}$, for $\mathcal{H}^{n-1}$-a.e. $y \in \xi^\bot$.}
\end{align*}
 Hence, we have that~\eqref{e:int6.1} holds.
This concludes the proof of the theorem.
 \end{proof}

\subsection{Slicing the jump set}
We start with general proposition which relates the trace on rectifiable sets of a measurable function $u\colon \Om \to \R^{m}$ with the traces of its one dimensional slices. Since its proof is an adaptation of a quite standard argument, we postpone its proof in Appendix~\ref{appendix}.

\begin{proposition}
\label{c:relje}
Let $u \colon \Omega \to \mathbb{R}^m$ be measurable, let $\xi \in \mathbb{S}^{n-1}$, and let $P \colon \Omega \to \xi^\bot $ be a parametrized map. Assume that there exists a diffeomorphism $\tau \colon \mathbb{R} \to (-1,1)$ such that 
$D_\xi \tau (u_\xi \circ \varphi) \in \mathcal{M}_b(\varphi^{-1}(\Omega))$. Then, for every countably $(n-1)$-rectifiable set $R \subseteq \Omega$ it holds true
\begin{equation}\label{e:corollary-relje}
    \aplim_{\substack{z \to x \\ \pm(z-x) \cdot \nu_{R} (x) >0}} u_\xi(z) = \aplim_{s \to t^{ \pm \sigma(x) }}\hat{u}^\xi_y(s) \qquad  \text{for $\mathcal{H}^{n-1}$-a.e.~$y \in \xi^{\bot}$ and for $t \in R^\xi_y$}\,,
\end{equation}
whenever at least one between the above approximate limits exists and where $x = \varphi(y+t\xi)$, $\nu_{R} \colon R \to \mathbb{S}^{n-1}$ is a Borel measurable orientation of $R$, and $\sigma(x) := \text{\emph{sign}}(\nu_{R}(x) \cdot \xi_{\varphi} (x))$.
\end{proposition}

\begin{remark}
We notice that the equality \eqref{e:corollary-relje} does not depend on the choice of the orientation~$\nu_{R}$.
\end{remark}

 Combining Theorem~\ref{t:int1} and Proposition~\ref{c:relje} we infer the following general structure result for the jump set of a measurable function~$u \colon \Om \to \R^{m}$. 

\begin{theorem}
\label{t:slicecoe}
 Let $\Om$ be an open subset of~$\R^{n}$, let $F \in C^{\infty} (\R^{n} \times \R^{n}; \R^{n})$ fulfilling~ \eqref{e:quadratic}, let $g \colon \Om \times \R^{n} \to \R^{m}$ satisfy \eqref{G1}, let $u \colon \Omega \to \mathbb{R}^m$ be measurable, let $\tau \colon \mathbb{R} \to (-1,1)$ be a diffeomorphism, and let $(P_{\xi})_{\xi \in \mathbb{S}^{n-1}}$ be a family of curvilinear projections on $\Omega$. Suppose that the following conditions hold:
\begin{enumerate}
\item There exists~$p \in (1, +\infty]$ such that~$\mathscr{I}_{u,p}$ is finite;
\item There exists $\rho>0$ such that $\emph{Osc}_u (\rho)$ is $\sigma$-finite w.r.t.~$\tilde{\mathcal{I}}^{n-1}$;
\item $D_\xi \tau(u_\xi \circ \varphi_\xi) \in \mathcal{M}_b(\varphi_\xi^{-1}(\Omega))$ for $\mathcal{H}^{n-1}$-a.e. $\xi \in \mathbb{S}^{n-1}$.
\end{enumerate}
Then, we have that
\begin{equation}
    \label{e:slicing1}
    J_{\hat{u}^\xi_y} = (J_{u_\xi})^\xi_y, \ \ \text{for $\mathcal{H}^{n-1}$-a.e. $\xi \in \mathbb{S}^{n-1}$,  $\mathcal{H}^{n-1}$-a.e. $y \in \xi^\bot$}.
\end{equation}
\end{theorem}

\begin{proof}
We start by showing that
\begin{equation}
    \label{e:slicing2}
    J_{\hat{u}^\xi_y} \subseteq (J_{u_\xi})^\xi_y \qquad  \text{for $\mathcal{H}^{n-1}$-a.e.~$\xi \in \mathbb{S}^{n-1}$,  $\mathcal{H}^{n-1}$-a.e.~$y \in \xi^\bot$}.
\end{equation}
 By Theorem~\ref{t:int1} we know that there exists a countably $(n-1)$-rectifiable subset~$R$ of~$\Om$ such that~\eqref{e:int6.1} holds.
By~\eqref{e:corollary-relje} applied to~$R$ we have that 
\begin{equation}
\label{e:slicing3}
J_{\hat{u}^\xi_y} \cap R^\xi_y \subseteq (J_{u_\xi})^\xi_y, \ \ \text{for $\mathcal{H}^{n-1}$-a.e. $\xi \in \mathbb{S}^{n-1}$,  $\mathcal{H}^{n-1}$-a.e. $y \in \xi^\bot$}.
\end{equation}
 Hence,~\eqref{e:int6.1} and~\eqref{e:slicing3} imply~\eqref{e:slicing2}. In order to show the opposite inclusion we first make use of Theorem~\ref{t:delnin} to infer the countably $(n-1)$-rectifiability of $J_{u_\xi}$. Thus, by applying again~\eqref{e:corollary-relje} to~$J_{u_\xi}$ we immediately obtain that
\begin{equation*}
    J_{\hat{u}^\xi_y} \supseteq (J_{u_\xi})^\xi_y, \ \ \text{for $\mathcal{H}^{n-1}$-a.e. $\xi \in \mathbb{S}^{n-1}$,  $\mathcal{H}^{n-1}$-a.e. $y \in \xi^\bot$}.
\end{equation*}
This concludes the proof.
\end{proof}

The remaining part of this section is devoted to the relation between the one-dimensional jump set~$J_{\hat{u}^{\xi}_{y}}$ and the slices~$(J_{\mathfrak{u}})^{\xi}_{y}$ of the jump set of the function~$\mathfrak{u}$ introduced in Definition~\ref{d:mathfraku}.

\begin{definition}
Let~$\Om$ be an open subset of~$\R^{n}$ and let $(P_{\xi})_{\xi \in \mathbb{S}^{n-1}}$ be a family of curvilinear projections on~$\Om$. Given an $(n-1)$-rectifiable set $R \subseteq \Omega$ and $\xi \in \mathbb{S}^{n-1}$ we denote
\begin{equation*}
    R^\xi := \left\{x \in R  :  \text{ there exists }\nu_{R}(x) \text{ and } \nu_{R}(x) \cdot \xi_{\varphi}(x) \neq 0 \right\}.
\end{equation*}
\end{definition}

We state two technical propositions whose proofs can be found in the Appendix~\ref{appendix}.

\begin{proposition}
\label{p:r=rxi}
Let~$\Om$ be an open subset of~$\R^{n}$ and let $(P_{\xi})_{\xi \in \mathbb{S}^{n-1}}$ be a family of curvilinear projections on~$\Om$. Assume that $R \subseteq \Omega$ is $(n-1)$-rectifiable. Then, we have that
\begin{equation}
\label{e:r=rxi}
    \mathcal{H}^{n-1} \big( \{\xi \in \mathbb{S}^{n-1} : \, \mathcal{H}^{n-1}(R  \setminus R^\xi)>0 \} \big) = 0\,.
\end{equation}
\end{proposition}

\begin{proposition}
\label{p:prodmeas}
Let $u \colon \Omega \to \mathbb{R}^m$ be Borel measurable, let $(P_\xi)_{\xi \in \mathbb{S}^{n-1}}$ be a family of curvilinear projections on $\Omega$, let $R \subseteq \Omega$ be countably $(n-1)$-rectifiable, and let $\nu \colon R \to \mathbb{S}^{n-1}$ be a Borel measurable orientation.  Assume that there exists a diffeomorphism $\tau \colon \mathbb{R} \to (-1,1)$ such that $D_\xi\tau(u_\xi \circ \varphi_\xi) \in \mathcal{M}_b(\varphi_\xi^{-1}(\Omega))$ for $\mathcal{H}^{n-1}$-a.e. $\xi \in \mathbb{S}^{n-1}$. If we set 
\begin{equation}
\label{e:nrelje1}
 \Delta  := \bigg\{   (x,\xi) \in R \times \mathbb{S}^{n-1} : \  \aplim_{\substack{z \to x \\ \pm(z-x) \cdot \nu_{R} (x) >0}} u_\xi(z)  = \aplim_{s \to t^{\pm\sigma(x)}_x}\hat{u}^\xi_{P_\xi(x)}(s) \bigg\}\,,
\end{equation}
 (the existence of at least one of the above approximate limits in \eqref{e:nrelje1} is tacitly assumed) then 
\begin{align}
\label{e:nrelje100}
(\mathcal{H}^{n-1} \restr R  \otimes \mathcal{H}^{n-1} \restr \mathbb{S}^{n-1}) \big ( (R \times \mathbb{S}^{n-1}) \setminus \Delta \big) = 0\,.
\end{align}
\end{proposition}

Finally, we are in position to prove Corollary \ref{c:int2}. We recall that, besides~\eqref{G1}, we now assume also condition~\eqref{G2} for $g$.

\begin{proof}[Proof of Corollary \ref{c:int2}]
By virtue of Theorem \ref{t:slicecoe} it is enough to prove that
\begin{equation}
    \label{e:mainslicepro2}
    (J_{u_\xi})^\xi_y = ( J_{\mathfrak{u}}  )^\xi_y \qquad \text{for $\mathcal{H}^{n-1}$-a.e.~$\xi \in \mathbb{S}^{n-1}$,  $\mathcal{H}^{n-1}$-a.e.~$y \in \xi^\bot$}.
\end{equation}
Since condition~\eqref{e:mainslicepro2} does not depend on the Lebesgue representative of~$u$ we may suppose that~$u$ is a Borel measurable function. Let $R \subseteq \Omega$ be the countably $(n-1)$-rectifiable set provided by Theorem~\ref{t:int1}.

 Thanks to Proposition~\ref{p:prodmeas} and to Fubini's theorem we know that for $\mathcal{H}^{n-1}$-a.e.~$x \in R$ we have (remember that the existence of both approximate limits below is tacitly guaranteed)
\begin{equation}
    \label{e:mainslicepro3}
     \aplim_{\substack{z \to x \\ \pm(z-x) \cdot \nu_{R} (x) >0}} u_\xi(z)  = \aplim_{s \to t^{\pm\sigma(x)}_x}\hat{u}^\xi_{P_\xi(x)}(s) \qquad \text{for $\mathcal{H}^{n-1}$-a.e.~$\xi \in \mathbb{S}^{n-1}$}.
\end{equation}
 Therefore, we infer from \eqref{e:mainslicepro9} and \eqref{e:mainslicepro3} that for $\mathcal{H}^{n-1}$-a.e.~$x \in R$ we have
\begin{equation}
\label{e:mainslicepro5}
    \mathcal{H}^{n-1}(\{\xi \in \mathbb{S}^{n-1} \, : \, x \in J_{u_\xi} \})>0.
\end{equation}
 Using property \eqref{G2} of~$g$ and property (4) of family of curvilinear projections, we infer from~\eqref{e:mainslicepro5} that for $\mathcal{H}^{n-1}$-a.e.~$x \in R$ we find $\{\xi^1,\dotsc,\xi^k\}$ (for some $1 \leq k \leq m$) and an open neighborhood $U$ of $x$ such that $x \in J_{u_{\xi_j}}$ for $j=1,\dotsc,k$ and such that 
\begin{align}
\label{e:mainslice99999}
\text{span}\{g(z,\xi^1_{\varphi}(x)),\dotsc,g(z,\xi^k_{\varphi}(x))\} = \text{span}\{g(z,v) : v \in \mathbb{R}^n \}, \ \ \text{for }z \in U 
\end{align}
With no loss of generality, we may assume $k=\text{dim} (\text{span}\{g(x,\xi^1_{\varphi}(x)),\dotsc, g(x,\xi^k_{\varphi}(x))\})$, otherwise we would just remove some of the $\xi_j$-s. Therefore, using the continuity of $g(\cdot,\cdot)$ and $ \xi_{\varphi}(\cdot)$, up to consider a smaller neighborhood $U$, we infer from \eqref{e:mainslice99999} 
\[
k=\text{dim} (\text{span}\{g(z,\xi^1_{\varphi}(z)),\dotsc, g(z,\xi^k_{\varphi}(z))\})=
\text{dim}(\text{span}\{g(z,v) : v \in \mathbb{R}^n \}), \ \ \text{for }z \in U.
\]
In particular we deduce
\begin{equation}
\label{e:mainslice99999.1}
   \text{span}\{g(z,\xi^1_{\varphi}(z)),\dotsc, g(z,\xi^k_{\varphi}(z))\} = \text{span}\{g(z,v) : v \in \mathbb{R}^n \}, \ \ \text{for }z \in U. 
\end{equation}
Using again the continuity of~$g$, condition \eqref{e:mainslice99999.1} gives continuous coefficients $\alpha_j \colon U \to \mathbb{R}$ such that $\mathfrak{u} (z) = \sum_j \alpha_j(z)  u(z) \cdot  g(z,\xi_{\varphi}^j(z))$ for $z \in U$. Therefore, we can write
\[
\begin{split}
\aplim_{\substack{z \to x \\ \pm(z-x) \cdot \nu_{R} (x) >0}} \mathfrak{u} (z) &= \sum_{j=1}^k \aplim_{\substack{z \to x \\ \pm(z-x) \cdot \nu_{R} (x) >0}} \alpha_j(z) \, u(z) \cdot g(z,\xi^j_\varphi(z))\\
&=\sum_{j=1}^k \alpha_j(x) \aplim_{\substack{z \to x \\ \pm(z-x) \cdot \nu_{R} (x) >0}}  \, u(z) \cdot g(z,\xi^j_\varphi(z)).
\end{split}
\]
This gives $\mathfrak{u}^\pm (x) \in \mathbb{R}^m$ for which
\[
\aplim_{\substack{z \to x \\ \pm(z-x) \cdot \nu_{R} (x) >0}} \mathfrak{u} (z)= \mathfrak{u}^\pm (x)\,.
\]
Moreover, it cannot be $\mathfrak{u}^+(x)= \mathfrak{u}^-(x)$, otherwise we would get a contradiction with the fact that $x \in J_{u_{\xi^j}}$ for $j=1,\dotsc,k$.  Hence, we have that  $x \in J_{\mathfrak{u}}$. Therefore, we have obtained that 
\[
x \in J_{\mathfrak{u}} \qquad \text{for $\mathcal{H}^{n-1}$-a.e.~$x \in R$}.
\]
As a consequence we can infer that
\begin{equation}
    \label{e:mainslicepro6}
    (J_{u_\xi} \cap R)^\xi_y \subseteq (J_{\mathfrak{u}})^\xi_y \qquad \text{for $\mathcal{H}^{n-1}$-a.e.~$\xi \in \mathbb{S}^{n-1}$,  $\mathcal{H}^{n-1}$-a.e.~$y \in \xi^\bot$}.
\end{equation}
Furthermore, the set~$R$ also satisfies~\eqref{e:int6.1}, which together with~\eqref{e:slicing1} gives
\begin{equation}
    \label{e:mainslicepro7}
    (J_{u_\xi} \cap R)^\xi_y = (J_{u_\xi})^\xi_y \qquad \text{for $\mathcal{H}^{n-1}$-a.e.~$\xi \in \mathbb{S}^{n-1}$,  $\mathcal{H}^{n-1}$-a.e.~$y \in \xi^\bot$}.
\end{equation}
 Combining \eqref{e:mainslicepro6} with \eqref{e:mainslicepro7} yields
 \begin{equation}
    \label{e:mainslicepro8}
    (J_{u_\xi})^\xi_y \subseteq (J_{\mathfrak{u}})^\xi_y  \qquad \text{for $\mathcal{H}^{n-1}$-a.e. $\xi \in \mathbb{S}^{n-1}$,  $\mathcal{H}^{n-1}$-a.e. $y \in \xi^\bot$}.
\end{equation}

It remains to prove the opposite inclusion. We claim that 
\begin{equation}
    \label{e:mainslicepro10}
    x \in J_{u_\xi} \qquad  \text{for $\mathcal{H}^{n-1}$-a.e. $x \in J_{\mathfrak{u}}$, $\mathcal{H}^{n-1}$-a.e. $\xi \in \mathbb{S}^{n-1}$}.
\end{equation}
Indeed, suppose by contradiction that \eqref{e:mainslicepro10} does not hold true.  Then, we find a set $B \subseteq J_{\mathfrak{u}}$  such that $\HH^{n-1}(B)>0$ and
\begin{equation}
\label{e:mainslicepro13}
 \mathcal{H}^{n-1}(\{ \xi \in \mathbb{S}^{n-1} \, : \, x \notin J_{u_\xi} \}) >0 \qquad \text{for $x \in B$}.
 \end{equation}
 In view of~\eqref{e:mainslicepro3} applied to the rectifiable set~$J_{\mathfrak{u}}$ (cf.~Theorem~\ref{t:delnin}), we find $x \in B$ and $\Sigma \subseteq \mathbb{S}^{n-1}$ with $\mathcal{H}^{n-1}(\Sigma)>0$ for which $x$ has to be a Lebesgue point for $u_\xi$ for every $\xi \in \Sigma$.  Arguing as above, we find $\{\xi^1,\dotsc,\xi^k\}$ and continuous coefficients $\alpha_j \colon U \to \mathbb{R}$ such that $\mathfrak{u} (z) = \sum_j \alpha_j(z)  u(z) \cdot  g(z,\xi_{\varphi}^j(z))$ for every~$z$ in some open neighborhood of~$x$. Therefore, we can write
\[
\begin{split}
\aplim_{z \to x}  \mathfrak{u} (z) &= \sum_{j=1}^k \aplim_{z \to x } \alpha_j(z) \, u(z) \cdot g(z,\xi^j_\varphi(z))\\
&=\sum_{j=1}^k \alpha_j(x) \aplim_{z \to x}  \, u(z) \cdot g(z,\xi^j_\varphi(z)),
\end{split}
\]
from which we immediately deduce that $x$ is a Lebesgue point of~$\mathfrak{u}$. This gives a contradiction with the assumption $x \in J_{\mathfrak{u}}$ and proves claim~\eqref{e:mainslicepro10}.

Since~$u$ is assumed to be Borel measurable, arguing as in Proposition~\ref{p:prodmeas} we infer that the set $\{(x,\xi) \in J_{\mathfrak{u}} \times \mathbb{S}^{n-1} \, : \, x \in J_{u_\xi} \}$ is Borel measurable. Therefore, we infer from~\eqref{e:mainslicepro10} and from Fubini's theorem the following property
\begin{equation}
    \label{e:mainslicepro11}
    x\in J_{u_\xi} \qquad  \text{for $\mathcal{H}^{n-1}$-a.e.~$\xi \in \mathbb{S}^{n-1}$, $\mathcal{H}^{n-1}$-a.e.~$x \in J_{\mathfrak{u}}$}.
\end{equation}
Condition~\eqref{e:mainslicepro11} immediately gives the opposite inclusion in~\eqref{e:mainslicepro8} and concludes the proof.
\end{proof}

 \section{An example}
 \label{s:applications}
 
 In this section we want to show how the hypotheses of Corollary \ref{c:int2} are guaranteed in the $BV^{\mathcal{A}}$-setting. In particular, we show how condition (2) can be ensured by means of Poincar\'e type of inequalities. 
 
We start with some general preliminaries, which will be useful also in Section~\ref{s:GBD}. We consider a field $F \in C^\infty(\Omega \times \mathbb{R}^n; \mathbb{R}^n)$ satisfying~\eqref{e:quadratic}, a function $g \in C(\Omega \times \mathbb{R}^n; \mathbb{R}^m)$ fulfilling \eqref{G1}--\eqref{G2}, and an operator $\mathcal{E} \colon C^\infty(\Omega; \mathbb{R}^m) \to C^\infty(\Omega; \mathbb{R}^k)$. Suppose that $\mathcal{E}$ satisfies the following condition; given $a \in C^\infty(\Omega; \mathbb{R}^m)$ and $\gamma \colon (-\tau,\tau) \to \mathbb{R}^n$ solution of $\ddot{\gamma} = F(\gamma,\dot{\gamma})$ there exists an increasing continuous function $c_{\mathcal{E}} \colon [0,+\infty) \to [0,+\infty)$ such that 
 \begin{equation}
     \label{e:operator1}
     \frac{\di}{\di t} [a(\gamma(t)) \cdot g(\gamma(t),\dot{\gamma}(t) )] \leq c_{\mathcal{E}}(|\dot{\gamma}(t)|)  |\mathcal{E}(a)(\gamma(t))|, \qquad \text{for every } t \in (-\tau,\tau). 
 \end{equation}
 
 We introduce the function space $\Xi(U)$. Given an open set $U \subset \Omega$ we define
\begin{equation*}
   \Xi(U) := \Big\{a \in C^{\infty}(U;\mathbb{R}^m) :\,  \|\mathcal{E}(a)\|_{L^\infty(U;\mathbb{R}^{k})} \leq \frac{1}{c_{\mathcal{E}} (2)} \Big\}.
\end{equation*}
In verifying condition (2) of Theorem \ref{e:int1}, instead of looking at $\text{Osc}_u (\rho)$ it is usually easier to obtain a control on the size of the following set 
\begin{equation*}
  [\text{Osc}]_u  (\rho) := \{ x \in \Omega : [\text{Osc}](u,x,\rho)>0\},
\end{equation*} 
where
\begin{equation*}
 [\text{Osc}](u,x,\rho):= \limsup_{r \to 0^+}  \inf_{a \in \Xi(\mathrm{B}_{\rho/2} (0))} \int_{\mathrm{B}_{\rho/2}(0)} |u_{r,x}-a| \wedge 1 \, \di z,
\end{equation*}
and $u_{r,x} \colon \mathrm{B}_1(0) \to \mathbb{R}^m$ is defined as $u_{r,x}(z):=u(x+rz)$. We have the validity of the following proposition.
\begin{proposition}
\label{p:application1}
Let $\rho>0$ and $(P_\xi)_{\xi \in \mathbb{S}^{n-1}}$ be a family of curvilinear projections on~$\Omega$. Then, we have
\begin{equation*}
     \emph{Osc}_u (\rho)  \subseteq  [\emph{Osc}]_u (\rho)\,.
\end{equation*}
\end{proposition}

\begin{proof}
By applying Coarea Formula with the map $\phi_x$ given in \eqref{e:retr12} and property \eqref{e:operator1}, it is not difficult to verify that we have for every but sufficiently small $r>0$
\begin{equation*}
   \int_{\mathbb{S}^{n-1}}  \text{Osc}_r(\mathring{u}^\xi_x, \rho)\, \di \mathcal{H}^{n-1}(\xi) \leq  \inf_{a \in \Xi(\mathrm{B}_{\rho/2}(0))} \int_{\mathrm{B}_{\rho/2}(0)} |u_{r,x} - a| \wedge 1 \, \di z,
\end{equation*}
where we have used that $|\dot{\exp}_{r,x}(t\xi)| \leq 2$ for $(\xi,t) \in \mathbb{S}^{n-1} \times (-1,1)$ whenever $r$ is sufficiently small.
\end{proof}

The previous proposition tells us that the $\sigma$-finiteness of $\text{Osc}_u (\rho) $ can be deduced from the $\sigma$-finiteness of $[\text{Osc}]_u (\rho)$. This latter condition is typically guaranteed by the validity of Poincar\'e's type of inequalities. 

\subsection{Complex-elliptic operators satisfying a mixing condition} A first example is provided by choosing $\mathcal{E}:= \mathcal{A}$ whenever $\mathcal{A}$ is a (first-order) complex-elliptic operator. We briefly recall that an operator of the form 
\[
\mathcal{A}(a)(x)=\sum_{i=1}^n A_i \partial_i a(x) \in C^\infty(\Omega;\mathbb{R}^k), \ \ \text{for }a \in C^\infty(\Omega;\mathbb{R}^m),
\]
for suitable linear map $A_i \colon \mathbb{R}^m \to \mathbb{R}^k$, is called complex-elliptic if and only if the complexification of its principal symbol $\mathbb{A}(\zeta) := \sum_{i=1}^n \zeta_i A_i $ ($\zeta \in \mathbb{C}^n$) satisfies the following inequality
\[
|\mathbb{A}(\zeta)v| \geq c |\zeta|\,|v|, \ \ \text{for every } \zeta \in \mathbb{C}^n \text{ and } v \in \mathbb{C} \otimes \mathbb{R}^m,
\]
for some constant $c>0$. In this case the kernel of $\mathcal{A}$ can be completely characterized in the sense that there exists a positive integer $\ell = \ell(\mathcal{A})$ such that whenever $\mathcal{A}(u)=0$ holds true in the sense of distribution on $\Omega$ then $u$ is a polynomial map from $\Omega$ with values in $\mathbb{R}^m$ of degree at most $\ell$ (cf.~\cite{Arr-Sk, smith}). As shown in \cite{Gme19} such a characterization leads to the following Poincar\'e's inequality; for any $x \in \Omega$ we find $\ell = \ell(\mathcal{A}) \in \mathbb{N} \setminus \{0\}$ and a polynomial $p_{x,r}$ of degree at most $\ell -1$ such that
\[
\|u -p_{x,r}\|_{L^{\frac{n}{n-1}}(B_r(x))} \leq c |\mathcal{A}(u)|({\rm B}_r(x)), \ \ \text{for $u \in BV^{\mathcal{A}}(\Omega)$ and ${\rm B}_r(x) \subset \Omega$},
\]
for some constant $c=c(n,\mathcal{A})>0$. In particular, by investigating the asymptotic behaviour of the coefficients of $p_{x,r}$ for $r \to 0^+$ it is possible to prove the following proposition regarding the oscillation of $u$ (cf.~\cite{Arr-Sk}).
\begin{proposition}
\label{p:application2}
Let $\mathcal{A}$ be a first-order complex-elliptic operator and let $u \in BV^\mathcal{A}(\Omega)$. Then for every $x \in \Omega$ satisfying $\Theta^{*n-1}(|\mathcal{A}(u)|,x)=0$ we have
\begin{equation*}
    \lim_{r \searrow 0} \, \inf_{a \in \mathbb{R}} \int_{\mathrm{B}_1(0)} |u_{r, x} - a |^{\frac{n}{n-1}} \, \di z =0\,,
\end{equation*}
where $\Theta^{*n-1}$ denotes the $(n-1)$-dimensional upper-density and $u_{r, x}(z):=u(x+rz)$. 
\end{proposition}
In addition, if $\mathcal{A}$ satisfies the \emph{mixing condition} introduced in~\cite{arr, Spe-VS}, it is possible to prove that given $\xi \in \mathbb{R}^{n}$ then there exist $e \in \mathbb{R}^m$ and $w \in \mathbb{R}^{k}$ such that $w$ is a rank-one covector and $(\xi,e)$ forms a \emph{spectral pair} (cf.~\cite{arr}). In particular, this implies that for every $x \in \Omega$ and $t \in \mathbb{R}$ for which $x +t\xi \in \Omega$ we have 
\begin{equation}
\label{e:spectral1}
    \frac{\di}{\di t} [a(x+t\xi) \cdot e] = w \cdot \mathcal{A}(a)(x +t\xi), \qquad \text{for } a \in C^\infty(\Omega;\mathbb{R}^m).
\end{equation}
By eventually dividing both sides of the previous inequality by the product $(|e| \vee 1) \, (|w| \vee 1)$ we may suppose with no loss of generality that  $|e| \leq 1$ and $|w| \leq 1$. Therefore, we can consider a map $p \colon \mathbb{R}^n \to \mathbb{R}^m$ which selects for each $\xi \in \mathbb{R}^n$ a vector $e_\xi \in \mathbb{R}^m$ with $|e_\xi| \leq 1$ for which \eqref{e:spectral1} holds true for some $w \in \mathbb{R}^k$ with $|w| \leq 1$. By choosing the field $F :=0$ and the map $g(x,\xi):=p(\xi)$, we see that condition~\eqref{e:operator1} is satisfied with $c_{\mathcal{E}}(\cdot)$ constantly equal to~$1$. This means that condition~(2) in Corollary~\ref{c:int2} is guaranteed by Propositions~\ref{p:application1} and~\ref{p:application2} together with the $\sigma$-finiteness of $\{x \in \Omega : \Theta^{*n-1}(|\mathcal{A}(u)|,x)>0 \}$ w.r.t.~$\mathcal{H}^{n-1}$ and the general formula (cf.~\cite[Corollary 2.10.11.]{fed})
\[
\int_{\xi^\bot} \mathcal{H}^0(E \cap P^{-1}(y))\, \di \mathcal{H}^{n-1}(y) \leq \text{Lip}^{n-1}(P) \mathcal{H}^{n-1}(E) \qquad E \subseteq \Omega\,.
\]
Conditions~(1) and~(3) are instead a direct consequence of the slicing representation~\cite{arr}. In particular, Corollary~\ref{c:int2} applies for every (first-order) complex-elliptic operator satisfying the above mentioned mixing condition. We further notice that in this case we can make use of \cite[Remark~2.2]{arr} to infer that the map~$\mathfrak{u}$ introduced in Definition~\ref{d:mathfraku} coincides with~$u$.

\section{Generalised bounded deformation}
\label{s:GBD}

In this section we show how Corollary \ref{c:int2} can be applied to the space of vector fields having generalised bounded deformation on manifolds. We assume that $\Omega$ is an open subset of $\mathbb{R}^n$, $\{e_1,\dotsc,e_n\}$ is the canonical basis of $\mathbb{R}^n$, and that $g \colon \Omega \times \mathbb{R}^n \to \mathbb{R}^n$ is the projection on the second component, namely, $g(x,z):=z$. We point out that with this choice of $g$ the local constant rank property \eqref{G2} is trivially satisfied. Moreover we fix a field $F \in C^{\infty} (\R^{n} \times \R^{n}; \R^{n})$ satisfying 
 \begin{enumerate}[label=(Q),ref=Q]
 \item \label{hp:F}  $F$ is a quadratic form in the second variable, that is, for every $x \in \R^{n}$ and every $v_1,v_2 \in \mathbb{R}^n$
\begin{equation*}
    F(x,v_1 + v_2) + F(x,v_1 - v_2) = 2 F(x,v_1) +2 F(x,v_2)\,.
\end{equation*}
\end{enumerate}
For later convenience we associate to $F$ a map $F^q \colon \mathbb{R}^n \to \text{Lin}(\mathbb{R}^n \otimes \mathbb{R}^n \otimes \mathbb{R}^n;\mathbb{R})$ in the following way
 \begin{equation*}
     F^q(x)(v_1 \otimes v_2 \otimes v_3)  := \frac{v_3}{2} \cdot (F(x,v_1+v_2)  -F(x,v_1)  -F(x,v_2) ) \qquad  v_1,v_2,v_3 \in \mathbb{R}^n.
 \end{equation*}
 It is worth noting that, under our hypothesis~\eqref{hp:F}, for every $v_3 \in \mathbb{R}^n$ the map $(v_1,v_2) \mapsto F^q(x)(v_1 \otimes v_2 \otimes v_3)$ is symmetric and hence can be represented as an element of $\mathbb{M}^{n \times n}_{sym}$. For this reason we can write
 \begin{equation*}
     F^q(x)(v_1 \otimes v_2 \otimes v_3)= (v_3 \cdot F^q(x))v_1 \cdot v_2,  \ \ v_1,v_2,v_3 \in \mathbb{R}^n, 
     \end{equation*}
     for a suitable $(v_3 \cdot F^q(x))\in \mathbb{M}^{n \times n}_{sym}$ depending on $v_3$. Given $r>0$ and a point $x \in \mathbb{R}^n$ we define $F_{r, x} \colon \mathbb{R}^n \times \mathbb{R}^n \to \mathbb{R}^n$ as $F_{r, x}(z,v):= r F ( x + r z , v )$ and analogously $F^q_{r, x} \colon \mathbb{R}^n \to \text{Lin}(\mathbb{R}^n \otimes \mathbb{R}^n \otimes \mathbb{R}^n; \mathbb{R})$ as $F^q_{r, x}(z):= rF^q(x+rz)$.

\subsection{The Rigid Interpolation condition}
As it will be shown in Section \ref{sub:poincare}, in order to apply Corollary \ref{c:int2} we need to assume a further condition on the field $F$ which we call \emph{Rigid Interpolation}. At this point it is convenient to introduce some notation. Given $z \in {\rm B}_1(0)$ and $r>0$, the symbol $\mathcal{S}_{0,z}$ denotes the set $\{z+e_0, \dotsc,z+e_n  \}$, where we have set $e_0:=0$, and for $0 \leq i < j \leq n$ we define $t \mapsto \ell_{z, r, ij}(t)$ as the curve $\gamma(\cdot)$ (whenever it is well defined) satisfying 
\begin{equation}
\label{e:poincare15000}
    \begin{cases}
    \ddot{\gamma}(t) = F_{r, x}(\gamma(t),\dot{\gamma}(t)), \ t \in [0,t_{ij}], \ \text{for some }t_{ij}>0 & \\
    \gamma(0)=z+e_i, \ \gamma(t_{ij})=z+e_j &\\
     |\dot{\gamma}(0)|= 1,  &
    \end{cases}
\end{equation}
where $F_{r,x}(z,v):=rF(x+rz,v)$. 

\begin{remark}
The existence of the curves $\ell_{z,r,ij}$ for sufficiently small $r$ depending on $x$ can be made rigorous following the lines of Lemma \ref{l:exp}. More precisely, if we denote by $\text{inj}_{r,x}(z)$ the injectivity radius defined in Definition \ref{d:inj} with $F$ replaced by $F_{r,x}$, we have that $\text{inj}_{r,x}(z) \to \infty$ as $r \to 0^+$ uniformly for $z \in \mathrm{B}_1(0)$.
\end{remark}

The symbol $\mathcal{S}_{r,1,z}$ denotes the 1-dimensional geodesic skeleton of $\mathcal{S}_{0,z}$, namely,  
\[
\mathcal{S}_{r,1,z} := \{ h \in \mathbb{R}^n : \, h = \ell_{z, r, ij}(t) \ \text{for some }t \in [ 0,t_{ij}] \text{ and } \ i \neq j \}.
\]
Moreover for $0 \leq i <j \leq n$ we define 
\[
\xi_{r,ij}(z):= \dot{\ell}_{z, r,ij}(0) \ \ \text{ and } \ \ \xi_{r,ji}(z):= \dot{\ell}_{z, r,ij}(t_{ij}).
\]
 We consider a semi-norm on $E_{r, z} \colon \mathbb{R}^{n+1} \times \mathbb{R}^n \to [0,\infty)$ defined as follows
\[
E_{r, z} (w) := \sum_{0 \leq i < j \leq n} |w^j \cdot \xi_{r,ji}(z) - w^i \cdot \xi_{r,ij}(z)|  \qquad \text{for $w \in \mathbb{R}^{(n+1) \times n}$,}
\]
where $w^{i}$ denotes the $i$-th column of the matrix~$w$.  Eventually, we denote by $\mathcal{S}_{n,z}$ the convex hull of $\mathcal{S}_{0,z}$. Observing that every~$z \in \mathrm{B}_1(0)$ with $z \cdot e_i < 0$ for every $i=1,\dotsc,n$ satisfies $\mathcal{S}_{n,z}\subset B_1(0)$ and that $
\mathcal{L}^n(\{ z \in \mathrm{B}_1(0) : \, z \cdot e_i < 0, \ i=1,\dotsc,n  \})=2^{-n}\omega_n$ we infer from elementary geometric considerations that there exists a dimensional constant $0<\rho(n) \leq 1$ such that $2^{n+1}\mathcal{L}^n(Q(n)) \geq \omega_n$ whenever
\begin{equation}
\label{e:rip1}
Q(n):= \{z \in \mathrm{B}_1(0) :\, \mathrm{B}_{\rho(n)}(0) \subset \mathring{\mathcal{S}}_{n,z} \subset \mathcal{S}_{n,z} \subset \mathrm{B}_1(0) \} \,.
\end{equation}
 We are now in position to state the required Rigid interpolation property of $F$;
 \begin{enumerate}[label=(RI),ref=RI]
   \item \label{hp:F2}  Given $x \in \mathbb{R}^n$ there exists a radius $r_x>0$ such that for every $z \in Q(n)$, $w \in \mathbb{R}^{(n+1)  \times n}$, and $0 < r \leq r_x$, we find a smooth map $a_r \colon \mathrm{B}_1(0) \to \mathbb{R}^n$ for which 
 \begin{align}
    \label{e:rip4}
    & \ \ \ \ \ \ \ \  a_r(h)=w^i  \qquad \text{ for every $h \in \mathcal{S}_{0,z}$,} \\
    \label{e:rip5}
    &\|\tilde{e}(a_r) - a_r \cdot F^q_{r , x} \|_{L^{\infty}(\mathcal{S}_{n,z}; \mathbb{M}^{n}_{sym})} \leq c(n) E_{r, z} (w)\,,
\end{align}
where $c(n)>0$ is a dimensional constant and where $\tilde{e}(a_r)$ denotes the symmetric gradient of $a_r$. 
\end{enumerate}

\begin{remark}
In the manifolds-setting considered in the next section, we will see that the operator $\mathcal{E}:= \tilde{e}(\cdot) - (\cdot) \cdot F^q \colon C^{\infty}(\Omega;\mathbb{R}^n) \to C^\infty(\Omega;\mathbb{M}^n_{sym})$ coincides with the curvilinear symmetric gradient. We further point out that $\mathcal{E}$ satisfies~\eqref{e:operator1}.
\end{remark}

\subsection{Definition of the space}

We start with a preliminary proposition.

\begin{proposition}
Let $\Om$ be an open bounded subset of~$\R^{n}$, $\xi \in \mathbb{S}^{n-1}$, let $P_\xi\colon \Om \to \xi^{\bot}$ be a curvilinear projection on~$\Omega$, and let $u \colon \Om \to \R^n$ be a measurable function. Then, for every $B \in \mathcal{B}( \Omega)$ the function
\[
y \mapsto |{\rm D} \hat{u}^{\xi}_{y} | (B^{\xi}_{y} \setminus J^{1}_{\hat{u}^{\xi}_{y}})  + \HH^{0} (B^{\xi}_{y} \cap J^{1}_{\hat{u}^{\xi}_{y} }) 
\]
is $\mathcal{H}^{n-1}$-measurable on $\xi^\bot$.
\end{proposition}
\begin{proof}
 Letting $v(x):= (u_\xi \circ \varphi_\xi)(x)$ for every $x \in \varphi^{-1}(\Omega)$ and using identity \eqref{e:sliceide}, namely, $v(y+t\xi)=\hat{u}^\xi_y$ for $y \in \xi^\bot$ and $t \in \Omega^\xi_y$ (notice that $\Omega^\xi_y= \{t \in \mathbb{R} : \ y+t\xi \in \varphi^{-1}(\Omega)\}$, the claimed measurability follows from~\cite[Lemma 3.6]{dal} 
\end{proof}

We are now in a position to define~$GBD_{F}(\Om)$.

\begin{definition}[The space $GBD_{F}(\Om)$]
\label{d:GBD}
Let $\Om$ be an open subset of~$\R^{n}$. We say that a measurable function~$u \colon \Om \to \R^{n}$ belongs to $ GBD_{F}(\Om)$ if there exists $\lambda \in \mathcal{M}_{b}^{+}(\Om)$ such that for every open subset $U$ of~$\Om$, every $\xi \in \mathbb{S}^{n-1}$, and every curvilinear projection~$P_{\xi}\colon U \to \xi^{\bot}$ on~$U$, the following facts hold true: 
\begin{enumerate}
\item\label{e:slice-1} $\hat{u}^{\xi}_{y} \in BV_{loc} (U^{\xi}_{y})$  for $\HH^{n-1}$-a.e.~$y \in \xi^{\bot}$;
\vspace{1mm}
\item\label{e:slice-2} for every Borel subset $B \in \mathcal{B}(U)$
\begin{align*}
\int_{\xi^{\bot}} \Big( |{\rm D} \hat{u}^{\xi}_{y} | (B^{\xi}_{y} \setminus J^{1}_{\hat{u}^{\xi}_{y}}) & + \HH^{0} (B^{\xi}_{y} \cap J^{1}_{\hat{u}^{\xi}_{y} } ) \Big)\, \di \HH^{n-1}(y) 
\leq \|\dot{\varphi}_\xi\|^2_{L^\infty}{\rm Lip}(P_{\xi};U)^{n-1} \lambda(B)\,. 
\end{align*}
\end{enumerate}
\end{definition}

\begin{remark}
\label{r:nontrivial}
We notice that, thanks to the construction in Section~\ref{sub:curvpro} (see Theorem~\ref{p:curvpro}), for every open set $\Om \subseteq \R^{n}$ there exists an at most countable family 
\[
\big\{ \big(x_{i}, r_{i}, (P_{\xi, x_{i}})_{\xi \in \mathbb{S}^{n-1}}\big): \, i \in I \big\}
\] 
such that $\{ {\rm B}_{r_{i}} (x_{i})\}_{i \in I}$ is a cover of~$\Om$ and $(P_{x_{i}, \xi})_{\xi \in \mathbb{S}^{n-1}}$ is a family of curvilinear projections in~${\rm B}_{r_{i}} (x_{i})$. This justifies the requests~$(1)$ and~$(2)$ in Definition~\ref{d:GBD} and that the set~$GBD_{F}(\Om)$ is nontrivial.
\end{remark}


%


\subsection{A weak Poincar\'e's inequality}
\label{sub:poincare}
In this subsection we show that a Poincar\'e inequality holds true in $GBD_F$ as a consequence of \eqref{hp:F2}.
\begin{theorem}
\label{t:poincare}
 Let $F \in C^{\infty} ( \mathbb{R}^n \times \mathbb{R}^n ; \mathbb{R}^n)$ satisfy~\eqref{hp:F}--\eqref{hp:F2} and $u \in GBD_F(\Omega)$ with $\lambda \in \mathcal{M}^+_b(\Omega)$ given by Definition \ref{d:GBD}. There exists a dimensional constant $c(n)>0$ such that if $\Theta^{*n-1}(\lambda,x)=0$ for some $x \in \Omega$ then we find a radius $r_x>0$ and for every $0<r\leq r_x$ a smooth map $a_r \colon \mathrm{B}_1(0) \to \mathbb{R}^n$ satisfying 
\begin{align}
\label{e:poincare1000}
&\|e(a_r) - a_r \cdot F^q_{r, x} \|_{L^{\infty}({\rm B}_{\rho(n)/2}(0); \mathbb{M}^{n}_{sym})} \leq c(n) r^{1-n}\lambda_r ({\rm B}_1(0))  \\
\label{e:poincare2000}
    & \ \ \ \  \int_{{\rm B}_{\rho(n)/2}(0)} |u_{r, x}-a_r| \wedge 1 \, \di z \leq c(n) r^{1-n} \lambda_r ({\rm B}_1(0))\, ,
\end{align} 
where $\rho(n)>0$ is the dimensional constant defined by condition \eqref{e:rip1}, $\lambda_r := \psi_{r, x \sharp}\lambda$, and $\psi_{r, x} \colon \mathrm{B}_r(x) \to \mathrm{B}_1(0)$ is defined as $\psi_{r, x}(z):= (z-x)/r$. 
\end{theorem}

\begin{proof}
We notice that $u_{r, x} \in GBD_{F_{r, x}}({\rm B}_1(0))$ and that $\lambda_r$ can be chosen in the definition of $u_{r, x} \in GBD_{F_{r, x}}({\rm B}_1(0))$ (see (2) of Definition~\ref{d:GBD}). 

It is convenient to fix some notation. For $z \in Q(n)$ (cf.~ \eqref{e:rip1}) and $v \colon {\rm B}_1(0) \to \mathbb{R}^n$ we define $w_{v(z)} \in \mathbb{R}^{(n+1) \times n}$ as
\[
w_{v(z)}:= (v(z+e_0), v(z+e_1), \dotsc, v(z+e_n))
\]
and for every $i=0,\dotsc,n$ we set
\[
w^i_{v(z)}:= v(z+e_i).
\]
For $z \in {\rm B}_{1}(0)$, $r>0$, and $\xi \in \R^{n}$, similarly to Definition~\ref{d:exp} we set $\exp_{r, z}(\xi) := \gamma(1)$, where $\gamma$ is the unique solution of
\begin{displaymath}
\left\{\begin{array}{lll}
\ddot{\gamma} (t) = F_{r, x} (\gamma(t), \dot{\gamma} (t)) \,,\\[1mm]
\gamma(0) = z\,,\\[1mm]
\dot{\gamma} (0) = \xi\,.
\end{array}\right.
\end{displaymath}
As $F_{r, x} \to 0$ in $C^{\infty}_{loc} (\R^{n} \times \R^{n}; \R^{n})$ as $r \searrow 0$, there exists $r_{x}>0$ such that~$\exp_{r, z} (\xi)$ is well-defined for every $\xi \in {\rm B}_{4} (0)$ and every $z \in {\rm B}_{1}(0)$. Arguing as in Lemma~\ref{l:exp}, we may as well assume that $\exp_{r, z}$ is a diffeomorphism of~${\rm B}_{4}(0)$ onto its image, and that $\exp_{r, z}({\rm B}_{4}(0)) \supseteq {\rm B}_{2}(z)$ for $0 < r < r_{x}$ and $z \in {\rm B}_{1}(0)$. Hence, we may as well define for $w \in {\rm B}_{2}(z)$
\begin{align*}
\phi_{r, z} (w) := \frac{ \exp^{-1}_{r, z} (w)}{ |\exp^{-1}_{r, z} (w)|} \,, \qquad 
\chi_{r, z} (w)  :=  \dot{\exp}_{r, z} (t \phi_{r, z} (w) ) |_{t = |\exp^{-1}_{r, z} (w)|} \,.
\end{align*}
Finally, we set
\begin{align}
\label{e:not1}
\hat{u}^\xi_{r, z }(t)&:=u_{r, x}(\exp_{r, z}(t\xi)) \cdot \dot{\exp}_{r, z}(t\xi) \qquad \text{for $t >0$ and $\xi \in \mathbb{S}^{n-1}$,} \\
B^\xi_{r, z} &:= \{t >0 : \, \exp_{r, z}(t\xi) \in B \}  \qquad \text{for $B \in \mathcal{B}( {\rm B}_1(0) )$.} \label{e:not2}
\end{align}
We notice that, up to fix a smaller $r_{x}>0$, $\hat{u}^{\xi}_{r, z}(t)$ is well posed for $t \in {\rm B}_{1}(0)^{\xi}_{r, z}$. 

We subdivide the proof of the theorem in 3 main steps.

\noindent\underline{{\em Step 1}.} We now prove that, up to redefine $r_{x}>0$, there exists a constant $c(n) >0$ such that for every $0 < r \leq r_{x}$
\begin{equation}
\label{e:poincare3000}
    \int_{Q(n)} E^1_{r, z}(w_{u_{r, x}(z)}) \, \di z \leq c(n) \lambda_r({\rm B}_1(0))\,,
\end{equation}
where $E^1_{r,z}$ denotes the truncated version of $E_{r,z}$, namely,
\[
E^1_{r, z} (w) := \sum_{0 \leq i < j \leq n} |w^j \cdot \xi_{r,ji}(z) - w^i \cdot \xi_{r,ij}(z)| \wedge 1  \qquad \text{for $w \in \mathbb{R}^{(n+1) \times n}$}.
\]

 Fix $i,j \in \{0,\dotsc,n  \}$ with $i < j$ and define $\xi := (e_j -e_i)/|e_j-e_i|$. Given $s \in (-1,1)$ we consider the $(n-1)$-dimensional affine space $\xi^\bot_s:= \xi^\bot + s\xi$.  Consider the vector field $v_{s, r} \colon \xi^\bot \to \mathbb{S}^{n-1}$ given by $v_{s,r}(y):= \phi_{r, y+s \xi}(y+s\xi+(e_j-e_i))$. If we consider the flow $G_r \colon \mathbb{R}^n \times \mathbb{R} \times \mathbb{R}^n \to \mathbb{R}^n \times \mathbb{R}^n$ relative to the system 
\begin{equation}
\label{e:poincare14000}
\begin{cases}
\ddot{\gamma}(t) = F_{r, x}(\gamma(t),\dot{\gamma}(t))&\\
\gamma(0)= x, \ x \in \mathbb{R}^n &\\
\dot{\gamma}(0)= v, \ v \in \mathbb{R}^n,
\end{cases}
\end{equation}
then because of the convergence $F_{r, x} \to 0$ in $C^\infty_{loc}(\mathbb{R}^n \times \R^{n} ; \R^{n})$ as $r \searrow 0$, we see that 
\begin{equation}
\label{e:Gr}
G_r( w,t,v) \to  (w+tv,v) \qquad \text{in } C^\infty_{loc}(\mathbb{R}^n \times \mathbb{R} \times \mathbb{R}^n ; \R^{n} \times \R^{n} ) \text{ as $r \searrow 0$}.
\end{equation}
As a consequence, if we define $\varphi_{\xi, r} \colon \mathbb{R}^n  \to \mathbb{R}^n$ in such a way that $\varphi_{ \xi, r}(y+(t+s)\xi) = \pi_1(G_r(y+s\xi,t,v_{ s, r}(y)))$ for $y \in \xi^\bot \cap {\rm B}_1(0)$ and $t \in (-\tau,\tau)$ (for a suitable $\tau>0$ which can be chosen independently of~$r$ and where $\pi_1 \colon \mathbb{R}^n \times \mathbb{R}^n \to \mathbb{R}^n$ denotes the orthogonal projection on the first coordinate), then for every but sufficiently small $r>0$ we find a Lipschitz map $P_{\xi,s,r} \colon {\rm B}_1(0) \to \xi^\bot$ whose Lipschitz constants converge to $1$ as $r \searrow 0$ and whose level sets are exactly the curves $t \mapsto \varphi_{\xi , r }(y+(t+s)\xi)$ for every $y \in \xi^\bot \cap {\rm B}_1(0)$. Arguing as in Section~\ref{sub:curvpro} and using~\eqref{e:Gr}, we get that for every but sufficiently small $r>0$ and for every $s \in (-1,1)$ and $\xi \in \mathbb{S}^{n-1}$ the map~$P_{\xi,s,r}$ is a curvilinear projection with respect to 
 $F_{r, x}$ parametrized by~$\varphi_{\xi, r}$. In addition, writing~$z \in Q(n)$ as $z = y  + s\xi - e_{i}$ for some $y \in \xi^{\bot} \cap {\rm B}_{1}(0)$ we deduce from the uniqueness of the solution of~\eqref{e:poincare15000} and~\eqref{e:poincare14000} that  $\varphi_{\xi, r}(y+(t+s)\xi)= \ell_{z, r,ij}(t)$ for every $t \in [0,t_{ij}]$. In particular, we have
\begin{align*}
\varphi_{\xi, r }(y+s\xi)=z+e_i  \qquad &\text{ and } \qquad  \varphi_{\xi,r}(y+(t_{ij}+s)\xi)= z+e_j \\
\xi_{r,ij}(z) = \dot{\varphi}_{\xi, r}(y+s\xi) \qquad &\text{ and } \qquad  \xi_{r,ji}(z) = \dot{\varphi}_{\xi,r}(y+(t_{ij}+s)\xi).
\end{align*}
 
Therefore, from the definition of $GBD_{F_{r, x}}({\rm B}_1(0))$ we can write by Fubini's Theorem
\[
\begin{split}
&\int_{Q(n)}|w_{u_{r, x}(z)}^j \cdot \xi_{r,ji}(z) - w_{u_{r, x}(z)}^i \cdot \xi_{r,ij}(z)| \wedge 1 \, \di z\\
&= \int_{-1}^1 \bigg( \int_{\xi^\bot+s\xi-e_i} |w_{u_{r, x}(h)}^j \cdot \xi_{r,ji} ( h )  - w_{u_{r, x} (h)}^i \cdot \xi_{r,ij} ( h) | \wedge 1 \, \di \mathcal{H}^{n-1}(h) \bigg) \di s \\
&  \ \ \ \ \ \ \ \ \ \ \leq 2 \|\dot{\varphi}_{\xi, r }\|^2_{L^\infty} \text{Lip}(P_{\xi, s, r} ; {\rm B}_1(0))^{n-1} \lambda_r({\rm B}_1(0)) \,.
\end{split}
\]
The arbitrariness of the indices $i$ and $j$ leads to the existence of a radius $r_x >0$ and a constant $c(n) >0$ for which \eqref{e:poincare3000} holds true.

\noindent \underline{\em Step 2.} Let us compactly write 
\[
\begin{split}
|D\hat{u}^\xi_{r, z}|:=&|D\hat{u}^\xi_{r, z}|({\rm B}_1(0)^\xi_{r, z} \setminus J^1_{\hat{u}^\xi_{r, z}} ) + \mathcal{H}^0({\rm B}_1(0)^\xi_{r, z} \cap J^1_{\hat{u}^\xi_{r, z}} ) \\
&O_{r, z}(u):= \int_{\mathbb{S}^{n-1}} |D\hat{u}^\xi_{r, z}| \, \di \mathcal{H}^{n-1}(\xi) \,.
\end{split}
\]
We claim that, up to a redefinition of~$r_{x}>0$ and of~$c(n)>0$, it holds 
\begin{align}
    \label{e:poincare4000}
    \int_{{\rm B}_1(0)} O_{r, z}(u) \, \di z &\leq c(n) \lambda_r({\rm B}_1(0)) \qquad \text{for $z \in {\rm B}_{1}(0)$}\,.
\end{align}
%
%
%

For $z \in {\rm B}_{1}(0)$, $i = 1, \ldots, n$, and $j \in \mathbb{N}$ we define the open sets 
\[
U^\pm_{z,i,j} := \bigg\{w \in {\rm B}_{2^{-j+1}}(z) \setminus \overline{{\rm B}}_{2^{-j}}(z) : \, (w-z) \cdot e_i > \pm \frac{|w-z|}{2\sqrt{n}}  \bigg\}\,. 
\]

For every $j \in \mathbb{N}$ and $z \in \rm {\rm B}_1(0)$ the map $P_{r, i, j} \colon U^\pm_{z,i,j} \to e_i^\bot$ defined as $P_{r, i, j}(w) := (\pi_{e_i} \circ \phi_{r, z})(w) $ is a curvilinear projection on both open sets $U^+_{z,i,j}$ and $U^-_{z,i,j}$ with respect to~$F_{r, x}$ for every but sufficiently small $r >0$ (uniformly in~$j$ and $z$). In order to verify this, we have to find a parametrization of  $\varphi_{r, i, j} \colon \{y + te_i :   (y,t) \in [e_i^\bot \cap {\rm B}_{\rho}(0)] \times (-\tau,\tau) \} \to \mathbb{R}^n$ of~$P_{r, i, j}$ (for suitable $\tau,\rho>0$) in the sense of Definition~\ref{d:param-maps}, for every $0 < r < r_{x}$.

 We focus on $U^+_{z,i,j}$. We set $\rho := \sqrt{1-(4\sqrt{n})^{-2}}$, $\tau := 2^{-j}$,  and
 $$\xi_y := (\pi_{e_i} \restr \{\eta \in \mathbb{S}^{n-1} : \, \eta \cdot e_i >(4\sqrt{n})^{-1} \})^{-1}(y)
 \qquad \text{for $y \in e_i^\bot \cap {\rm B}_{\rho}(0)$}\,.
 $$
   For $(y,t) \in [e_i^\bot \cap {\rm B}_{\rho}(0)] \times (-\tau,\tau)$ we let 
\[
\varphi_{r, i, j} (y + te_i)= \exp_{r, z} \big((t + \tau +2^{-j-2}) \xi_y \big)\,.
\]
 To verify (1) of Definition \ref{d:param-maps} we notice that $U^{+}_{z,i,j} \subset \text{Im}(\varphi_{e_i})$ is equivalent to
\begin{equation}
\label{e:radialslice1}
\begin{split}
& \{w \in {\rm B}_1(0) \setminus \overline{{\rm B}}_{2^{-1}}(0) : \, w \cdot e_i > \pm|w|/2\sqrt{n}  \} 
\\
&
\subset \left\{\frac{\exp_{r, z}(2^{-j}(t +1 + 2^{-2})\xi_y) - z}{2^{-j}} : \, (y,t) \in  [e_i^\bot \cap {\rm B}_{\rho}(0)] \times (-1,1) \right\}\,.
\end{split}
\end{equation}
Since $(\exp_{r, z} ( 2^{-j} \, \cdot ) - z) / 2^{-j} \to \text{Id}$  in $C^\infty_{loc} (\R^{n}; \R^{n})$ as $r \searrow 0$ uniformly w.r.t.~$j \in \mathbb{N}$, by our choice of~$\rho$ and~$\tau$ we deduce that, up to redefine $r_{x}>0$, inclusion \eqref{e:radialslice1} holds true for every $0 < r < r_{x}$, every $z \in {\rm B}_{1}(0)$, and every $j \in \mathbb{N}$. In a similar way we get that $\varphi_{r, i, j}^{-1} \restr U^+_{z,i,j}$ is a bi-Lipschitz diffeomorphism for every $i$, $j$, and $0 < r < r_{x}$. This gives property~(2) of Definition~\ref{d:param-maps}. Property~(3) follows by construction. To conclude, we have to show that~$\varphi_{r, i, j}$ satisfies $\ddot{\varphi}_{r, i, j} = F_{r, x}(\varphi_{r, i, j}, \dot\varphi_{r, i, j})$. Again this follows by construction, since the level sets~$\phi_{r, z}^{-1}(\xi)$ are described exactly by the curve $t \mapsto \exp_{r, z}(t\xi)$. 

Let us fix~$j \in \mathbb{N}$ and $i = 1 , \ldots, n$. Applying the definition of~$GBD_{F_{r,x}} ({\rm B}_{1}(0))$ (see Definition~\ref{d:GBD}) with curvilinear projection $P_{r, i, j} \colon U^+_{z,i,j} \to e_i^\bot$ we get for every $B \in \mathcal{B}(U_{z, i, j}^{+})$
\begin{align}
\label{e:1000}
   \int_{e_i^\bot}\!\!\! \big(  | D (\hat{u}_{r, x})^{e_i}_y|(B^{e_i}_y \setminus J^1_{(\hat{u}_{r, x})^{e_i}_y}) & + \mathcal{H}^0(B^{e_i}_y \cap J^1_{(\hat{u}_{r, x})^{e_i}_y} \big) \, \di \mathcal{H}^{n-1}(y) 
   \\
   &
   \leq \| \dot{\varphi}_{r, i, j} \|_{L^\infty}^2 \text{Lip}(P_{r, i, j} ; U^+_{z,i,j})^{n-1} \lambda_{r}(B) \,. \nonumber
\end{align}
Since $F_{r, x} \to 0$ in $C^{\infty}_{loc} (\R^{d} \times \R^{d}; \R^{d})$, we have that, up to redefine $r_{x}>0$ and $c(n)>0$, it holds
\begin{align*}
 \text{Lip}(P_{r, i, j} ;U^+_{z,i,j})^{n-1} \leq c(n) 2^{(n-1)(j+1)} \qquad \text{and} \qquad  \|\dot{\varphi}_{r, i, j}\|_{L^\infty}^2 \leq 2\,.
\end{align*}
Hence, it follows from~\eqref{e:1000} that
\begin{align}
    \label{e:radialslice2}
    \int_{e_i^\bot} \big( |D(\hat{u}_{r, x})^{e_i}_y | (B^{e_i}_y \setminus J^1_{(\hat{u}_{r, x})^{e_i}_y}) & + \mathcal{H}^0(B^{e_i}_y \cap J^1_{(\hat{u}_{r, x})^{e_i}_y} \big)  \, \di \mathcal{H}^{n-1}(y) 
    \\
    & 
    \leq 2 c(n) 2^{(n-1)(j+1)} \lambda_{r}(B) \leq 2 \int_{B} \frac{c(n)}{|w-z|^{n-1}} \, \di \lambda_{r} (w)\,, \nonumber
\end{align}
for every  $B \in \mathcal{B}( U^+_{z,i,j})$ and every $j\in \mathbb{N}$. The same inequality for the set $U^-_{z,i,j}$ can be obtained.

 Let us set $S_{i}(z):= \{w \in  \mathrm{B}_1(0) : \, |(w-z) \cdot e_i| > |w-z|/(2\sqrt{n})  \}$. Choosing $B = U^+_{z,i,j}  \cap \rm B_1(0) $ and $B= U^-_{z,i,j}  \cap \rm B_1(0) $ in~\eqref{e:radialslice2} and summing both sides with respect to $j \in \mathbb{N}$, we obtain that
\begin{equation}
    \label{e:radialslice3}
    \begin{split}
    \int_{e_i^\bot} \big(  |D(\hat{u}_{r, x})^{e_i}_y|(S_{i}(z)^{e_i}_y \setminus J^1_{(\hat{u}_{r, x})^{e_i}_y}) &+ \mathcal{H}^0(S_{i}(z)^{e_i}_y \cap J^1_{(\hat{u}_{r, x})^{e_i}_y}  \big) \, \di \mathcal{H}^{n-1}(y)\\
    &\leq 2\int_{S_{i}(z)} \frac{c(n)}{|w - z|^{n-1}} \, \di \lambda_{r}(w)\,.
    \end{split}
\end{equation}
  Recalling the notation~\eqref{e:not1}--\eqref{e:not2}, 
we can perform the change of variable induced by the map $\pi_{e_i} \colon \{ \xi \in \mathbb{S}^{n-1} : |\xi \cdot e_i| > 1/(2\sqrt{n}) \} \to e_i^\bot$  on the integral on the left hand-side of~\eqref{e:radialslice3} and up to redefine $c(n)>0$ we get
\begin{align}
    \label{e:radialslice4}
    \int_{\mathbb{S}^{n-1}} \big( |D\hat{u}^{\xi}_{ r, z} |(S_{i}(z)^{\xi}_{r, z} \setminus J^1_{\hat{u}^{\xi}_{r, z}}) &+ \mathcal{H}^0(S_{i}(z)^{\xi}_{r, z} \cap J^1_{\hat{u}^{\xi}_{r, z}} \big) \, \di \mathcal{H}^{n-1}(\xi)\\
    &\leq 2\int_{S_{i}(z)} \frac{c(n)}{|w-z|^{n-1}} \, \di \lambda_{r}(w)\,. \nonumber
\end{align}
By possibly redefining again $c(n)>0$ and by summing both sides of~\eqref{e:radialslice4} with respect to $i=1,\dotsc,n$, we obtain
\begin{align}
    \label{e:radialslice}
    \int_{\mathbb{S}^{n-1} }  | D\hat{u}^{\xi}_{r, z} | ( \mathrm{B}_{1}(0)^{\xi}_{r, z} \setminus J^1_{\hat{u}^{\xi}_{r, z}}) &+ \mathcal{H}^0( \mathrm{B}_{1}(0)^{\xi}_{r, z}  \cap J^1_{\hat{u}^{\xi}_{r, z}}) \, \di \mathcal{H}^{n-1}(\xi)\\
    &\leq \int_{{\rm B}_{1}(0)} \frac{c(n)}{|w-z|^{n-1}} \, \di \lambda_{r}(w)\,. \nonumber
\end{align}
Then, inequality~\eqref{e:poincare4000} follows by integrating both side of~\eqref{e:radialslice} with respect to $z \in {\rm B}_{1}(0)$ and by using Fubini's Theorem.

\noindent{\underline{\em Step 3: proof of~\eqref{e:poincare1000} and~\eqref{e:poincare2000}.}} 
We claim that, up to redefine $c(n) >0$, for every $0 < r \leq r_x$ we find $z_r \in Q(n)$ satisfying for every $z \in S_{0,z_r}$
\begin{align}
    \label{e:poincare5000}
    &  E^1_{r, z_r}(w_{u_{r, x}(z_r)}) \leq c(n) \lambda_r({\rm B}_1(0)) \,, \\
    \label{e:poincare6000}
    &  O_{r, z}(u)\leq c(n) \lambda_r({\rm B}_1(0))\,, \\
    \label{e:poincare7000}
    &|u_{r, x}(h) \cdot \chi_{r, z}(h)- u_{r, x}(z) \cdot \phi_{r, z}(h)| \wedge 1 \leq c(n) |D\hat{u}^{\phi_{r, z}(h)}_{r, x}| \ \text{a.e. } h \in {\rm B}_{\rho(n)/2}(0)\,,
\end{align}
where in \eqref{e:poincare7000} we are also assuming that every $z \in S_{0,z_r}$ is a Lebesgue's point of~$u_{r, x}$. Indeed,~\eqref{e:poincare5000}--\eqref{e:poincare6000} can be obtained from \eqref{e:poincare3000}--\eqref{e:poincare4000} via Chebychev's inequality and appealing to the lower bound on the measure of $Q(n)$ in \eqref{e:rip1}. Notice also that in view of $\Theta^{*n-1}(\lambda,x)=0$, inequality \eqref{e:poincare5000} becomes for every but sufficiently small $r>0$
\begin{equation}
\label{e:poincare5000.1}
      E_{r, z_r}(w_{u_{r, x}(z_r)}) \leq c(n) \lambda_r({\rm B}_1(0)) \,.
\end{equation}

For what concerns~\eqref{e:poincare7000}, we notice that if $z \in {\rm B}_1(0)$ satisfies $O_{r, z}(u) < \infty$, then for $\mathcal{H}^{n-1}$-a.e. $\xi \in \mathbb{S}^{n-1}$ the function $t \mapsto \hat{u}^\xi_{r, z}(t)$ belongs to $BV_{loc} ({\rm B}_1(0)^\xi_{r, z})$. In addition, if~$z$ is also a Lebesgue point of $u_{r, x}$, we can apply the Fundamental Theorem of Calculus to write exactly
\[
\begin{split}
|u_{r, x}(h) \cdot \chi_{r, z}(h)- u_{r, x}(z) \cdot \phi_{r, z}(h)| \wedge 1 \leq |D\hat{u}^{\phi_{r, z}(h)}_{r, z}| \,,
\end{split}
\]
where we used the concavity of the truncation function when restricted to the positive real line. Eventually, for every $0 < r \leq r_x$ we associate to each $h \in {\rm B}_{\rho(n)/2}(0)$ the $n$-tuple $(\chi_{z_r+e_i,r}(h))_{i=1}^n$ and notice that, since $F_{r, x} \to F_{0, x} = 0$ in $C^\infty_{loc}(\mathbb{R}^n \times \mathbb{R}^n; \R^{n})$ as $r \searrow 0$, the following convergence holds true for every $i=1,\dotsc,n$
\begin{equation}
    \label{e:poincare8000}
    \lim_{r \searrow0} \bigg\| \chi_{r, z_r+e_i} - \frac{(\cdot) -(z_r+e_i)}{|(\cdot) -(z_r+e_i)|} \bigg\|_{L^\infty({\rm B}_{\rho(n)/2}(0))} = 0\,. 
\end{equation}
Condition \eqref{e:poincare8000} implies that, up to redefine $r_x>0$ and $c(n)>0$, we have
\begin{equation}
    \label{e:poincare9000}
    \inf_{h \in {\rm B}_{\rho(n)/2}(0)} |\chi_{r, z_r+e_1}(h) \wedge \dotsc \wedge \chi_{r, z_r+e_n } ( h ) | \geq \frac{1}{c(n)} \qquad \text{ for every } 0<r\leq r_x \,.
\end{equation}
By the Rigid interpolation property~\eqref{hp:F2}, up to redefine again $r_x>0$ and $c(n)>0$, we find a smooth map $a_r \colon {\rm B}_{\rho(n)/2}(0) \to \mathbb{R}^n$ satisfying for every $0 < r \leq r_x$
\begin{align}
    \label{e:poincare10000}
    &a_r(z)=u_{r, x}(z) \qquad \text{ for every $z \in S_{0,z_r}$,} \\
    \label{e:poincare11000}
    &\|e(a_r) - a_r \cdot F^q_{r, x} \|_{L^{\infty}(S_{n,z_r}; \mathbb{M}^{n}_{sym})} \leq c E_{z_r,r}(w_{u_{r, x}(z_r)}) \,. 
\end{align}
In particular, \eqref{e:poincare11000} in combination with \eqref{e:poincare5000} gives immediately \eqref{e:poincare1000}. 

It remains to prove \eqref{e:poincare2000}. To this purpose, we define the vector field $X \colon {\rm B}_{\rho(n)/2}(0) \to \mathbb{R}^n$ as
\[
X(h):=((u_{r, x}(h) - a_r(h)) \cdot \chi_{r, z_r+e_1}(h), \dotsc, (u_{r, x}(h)-a_r(h)) \cdot \chi_{r, z_r+e_n}(h)).
\]
By virtue of \eqref{e:poincare7000} and \eqref{e:poincare10000} we can write for every $i=1,\dotsc,n$ and for a.e.~$h \in {\rm B}_{\rho(n)/2}(0)$
\[
\begin{split}
    |X_i(h)| \wedge 1 &= |(u_{r, x } ( h ) - a_r( h ) ) \cdot \chi_{r, z_r+e_i} ( h ) - (u_{r, x} (z_r+e_i) - a_r ( z_r + e_i ) ) \cdot \phi_{r, z_r+e_i} ( h ) | \wedge 1 \\
    & \leq | D\hat{u}^{\phi_{r, z_r+e_i}(h)}_{r, z_r+e_i}| + \|e(a_r) - a_r \cdot F^q_{r, x} \|_{L^{\infty}(S_{n,z_r}; \mathbb{M}^{n}_{sym})} \sup_{\xi,t } |\dot{\exp}_{r, z_r+e_i} (t\xi)|\,, 
\end{split}
\]
where the supremum is considered for all $\xi \in \mathbb{S}^{n-1}$ and all positive~$t$ for which the map $\exp_{r, z_r+e_i}((\cdot)\xi)$ takes value in~${\rm B}_1(0)$. In view of the previous inequality and of~\eqref{e:poincare11000}, we can redefine $r_x$ and $c(n)>0$ in such a way that 
\begin{equation}
\label{e:poincare12000}
|X_i(h)| \wedge 1 \leq  |D\hat{u}^{\phi_{r, z_r+e_i}(h)}_{r, z_r+e_i} | + c(n) E_{r, z_r}(w_{u_{r, x}(z_r)}) \,, 
\end{equation}
for a.e.~$h \in {\rm B}_{\rho(n)/2}(0)$, for every $0< r \leq r_x$, and for every $i=1,\dotsc,n$.
Notice that Coarea formula together with estimate \eqref{e:retr2} implies
\[
\begin{split}
    \int_{{\rm B}_{\rho(n)/2}(0)} & |D\hat{u}^{\phi_{r, z_r+e_i}(h)}_{r, z_r+e_i} | \, \di h 
     \leq  c'(n) \int_{{\rm B}_{\rho(n)/2}(0)} |D\hat{u}^{\phi_{r, z_r+e_i}(h)}_{r, z_r+e_i}| J\phi_{r, z_r+e_i}(h) \, \di h \\
    &=c''(n)\int_{\mathbb{S}^{n-1}} |D\hat{u}^{\xi}_{r, z_r+e_i}|\big( \mathcal{H}^1(\phi^{-1}_{r, z_r+e_i} (\xi) \cap {\rm B}_{\rho(n)/2}(0))\big) \, \di \mathcal{H}^{n-1}(\xi) \\
    &\leq c'''(n) \int_{\mathbb{S}^{n-1}} |D\hat{u}^{\xi}_{r, z_r+e_i}| \, \di \mathcal{H}^{n-1}(\xi)
\end{split}
\]
for suitable constants $c'(n),c''(n),c'''(n)$ and for every but sufficiently small $r>0$. Therefore, by~\eqref{e:poincare5000}--\eqref{e:poincare6000} we can possibly redefine $r_x>0$ and $c(n)>0$ so that by integrating both sides of \eqref{e:poincare12000} we obtain
\begin{equation}
    \label{e:poincare13000}
    \int_{{\rm B}_{\rho(n)/2}(0)} |X_i(h)| \wedge 1 \, \di h \leq c(n) \lambda_r({\rm B}_1(0))  \qquad \text{ for every } 0 <r \leq r_x\,.
\end{equation}
Finally, we notice that condition \eqref{e:poincare9000} implies that, up to redefine $c(n)>0$, we have also 
\[
|u_{r, x}(h) - a_r(h)| \leq c(n) |X(h)| \qquad \text{ for every $h \in {\rm B}_{\rho(n)/2}(0)$ and $0 <r \leq r_x$},
\]
which together with \eqref{e:poincare13000} immediately gives the validity of \eqref{e:poincare2000}.
\end{proof}

We conclude this subsection with the following proposition.

\begin{proposition}
\label{p:lll}
Let $F \in C^{\infty} ( \mathbb{R}^n \times \mathbb{R}^n ; \mathbb{R}^n)$ satisfy~\eqref{hp:F}--\eqref{hp:F2}, $u \in GBD_F(\Omega)$, and $x \in \Omega$ such that $\Theta^{*n-1}(\lambda,x)=0$. Then
\begin{equation*}
    \lim_{r \searrow 0} \, \inf_{a \in \Xi (\rm B_{\rho(n)/2}(0))} \int_{\mathrm{B}_{\rho(n)/2}(0)} |u_{r, x} - a | \wedge 1 \, \di z =0\,.
\end{equation*}
\end{proposition}

As a consequence of Proposition \ref{p:lll}, the set $[{\rm Osc}]_{u} (\rho(n))$ is $\sigma$-finite w.r.t.~$\HH^{n-1}$.

\subsection{The case of Riemannian manifolds}
\label{sub:Riemann}

In this subsection we verify that, in the case $F$ is related to the geometry of a Riemannian manifold, then condition \eqref{hp:F2} is satisfied. For this purpose consider an $n$-dimensional manifold $\rm (M,g)$ which for simplicity we suppose embedded in $\mathbb{R}^m$ (actually this is not restrictive in view of the celebrated Nash's imbedding Theorem \cite{Nash}). Let $(U,\psi)$ denote a chart with $\psi \colon U \to \Omega \subset \mathbb{R}^n$. Given $q \in \Omega$ and $0 < \delta < \text{dist}(z, \partial \Omega )$ we denote by $\Psi$ a $C^\infty$-regular extension of $\psi \restr \psi^{-1}(\mathrm{B}_\delta(q))$ to the whole of $\mathbb{R}^m$ and such that $\mathrm{B}_\delta(q) \subset \Psi(\mathbb{R}^m) \subset \Omega$. Letting $\varphi \colon \Omega \to U$ denote the inverse of $\psi$, we define $g_i \colon \Omega \to \mathbb{R}^m$ as $g_i(h):= \partial_i \varphi(h)$ for every $i=1,\dotsc,n$. Since \begin{equation*}
\text{rank}\{g_1(h),\dotsc,g_n(h)\} = n \qquad  \text{for every $h\in \Omega$},
\end{equation*}
we can define the $i$-th element of the dual basis $g^i \colon U \to \mathbb{R}^{m*}$ in such a way that
\begin{equation}
\label{e:appendix2000}
\langle g^i(h), g_j(h)\rangle = \delta_{ij} \qquad  \text{for every }i,j=1,\dotsc,n \ \text{ and } \ h \in \Omega\,,
\end{equation}
where $\langle \cdot, \cdot \rangle$ denote the duality paring between $\mathbb{R}^{m*}$ and $\mathbb{R}^m$.

Now we want to write the covariant derivative of $a \in \mathscr{D}^1(\rm M)$, namely, a differential one-form on $\rm M$, locally in terms of the dual basis $\{g^1,\dotsc,g^n\}$. For this purpose we denote by $a_i \colon \Omega \to \mathbb{R}$ the \emph{curvilinear coordinates} of $a$, namely, $a(h)= \sum_i a_i(h) g^i(h)$ for every $h \in \Omega$. For $i=1,\dotsc,n$, the $i$-th \emph{covariant derivatives} $\nabla_i a(h) $ of~$a$ at $h \in \Omega$ is the element of $\mathbb{R}^{m*}$ defined as
\begin{align}
\label{e:A6}
    &\langle \nabla_i a(h), g_j(h) \rangle := \partial_i a_j(h) -\sum_k a_k(h) \Gamma^k_{ij}(h) \qquad  j=1,\dotsc,n \,, \\
    \label{e:A6.1}
    & \ \ \ \ \ \ \ \ \ \langle \nabla_i a(h), v \rangle :=0 \qquad \text{if } v \notin \text{span} \{g_1(h),\dotsc,g_n(h)\},
\end{align}
where $\Gamma^k_{ij}$ denotes the Christoffel's symbols with respect to the basis $\{g_1,\dotsc,g_n\}$; we recall that they can be computed in coordinates as
\begin{equation*}
\Gamma^k_{ij}(h):=-\langle \partial_i g^k(h), g_j(h) \rangle \qquad \text{for } i,j,k=1,\dotsc,n.
\end{equation*}
The gradient of $a$ can be locally represented as an operator $\nabla \colon C^{\infty}(\Omega;\mathbb{R}^n) \to C^\infty(\Omega;\mathbb{M}^n)$ called \emph{curvilinear gradient} and defined as
\[
\begin{split}
[\nabla(a)]_{ij}(h) &:= \langle \nabla_i a(h), g_j(h) \rangle \\  &:=\partial_i a_j(h)  - \sum_{\ell=1}^n a_\ell(h) \Gamma^\ell_{ij}(h), \qquad i,j=1,\dotsc,n. 
\end{split}
\]
 Analogously, the symmetric gradient of $a$ can be locally represented as an operator $e \colon C^{\infty}(\Omega;\mathbb{R}^n) \to C^\infty(\Omega;\mathbb{M}^n_{sym})$ called \emph{curvilinear symmetric gradient} and defined as
\[
\begin{split}
2[e(a)]_{ij}(h) &:= \langle \nabla_i a(h), g_j(h) \rangle + \langle \nabla_j a(h), g_i(h) \rangle\\  &:=\partial_i a_j(h) + \partial_j a_i(h) - 2\sum_{\ell=1}^n a_\ell(h) \Gamma^\ell_{ij}(h), \qquad i,j=1,\dotsc,n. 
\end{split}
\]
 We continue by fixing~$(\tilde{e}_i)_{i=1}^n$ and~$(e_i)_{i=1}^m$ two orthonormal bases of~$\mathbb{R}^n$ and of~$\mathbb{R}^m$, respectively. We further denote by $(\tilde{e}^i)_{i=1}^n$ and $(e^i)_{i=1}^m$ their dual bases. For convenience of notation we also set~$\tilde{e}_0:=0$. 
 By definition of~$\Psi$ and~$\varphi$, for $h \in \mathrm{B}_\delta(q)$ we have the identity 
\[
\partial_i (\Psi \circ \varphi)(h) \cdot \tilde{e}_j = \delta_{ij} \qquad  \text{for }i,j=1,\dotsc,n\,.
\]
Thus, for every $i,j=1,\dotsc,n$ and every $h \in \mathrm{B}_\delta(q)$ we get that 
\begin{equation}
    \label{e:appendix1000}
    \delta_{ij}=\sum_{\ell=1}^m (\partial_\ell \Psi(\varphi(h)) \cdot \tilde{e}_j) (\partial_i \varphi(h) \cdot e_\ell) =\sum_{\ell=1}^m ( \partial_\ell \Psi(\varphi(h)) \cdot \tilde{e}_{j}) (g_i (h ) \cdot e_\ell)\,.
\end{equation}
Combining \eqref{e:appendix2000} and~\eqref{e:appendix1000} we obtain for every $h \in \mathrm{B}_\delta(q)$
\begin{equation}
    \label{e:appendix3000}
    \partial_\ell\Psi(\varphi(h)) \cdot \tilde{e}_j = \langle g^j(h),e_\ell \rangle  \qquad  \text{for }\ell=1,\dotsc,m  \text{ and }  j=1,\dotsc,n\,.
\end{equation}

We notice that, because of the fact that $\rm M$ is isometrically embedded in $\mathbb{R}^m$, there is a natural homomorphism $\mathrm{i} \colon \mathscr{D}^1(U) \to C^\infty(\Omega;\mathbb{R}^{m*})$ acting as
\[
\mathrm{i}(a)(x) := \sum_i a_i(h)g^i(h), \qquad \text{ for } x \in U \text{ and }h = \psi(x). 
\]
Therefore we can consider coefficients $\tilde{a}_k:= \langle a,e_k \rangle$, so that $a(h)= \sum_{k=1}^m \tilde{a}_k(h) \, e^k$ and the (euclidean) gradient $\tilde{\nabla}(a)(h) \in \mathbb{M}^{m \times n}$ of~$a$ at $h \in \Omega$ is defined as
\begin{equation}
    \label{e:appendix1.1}
    [\tilde{\nabla}(a)]_{ij}(h) := \langle \partial_i \tilde{a}(h), e^j \rangle =  \partial_i \tilde{a}_j(h) \qquad \text{for $i=1,\dotsc,n$ and $j=1\dotsc,m$.}
\end{equation}
We want to find a precise relation between the gradient and the curvilinear gradient. In order to do so, we first write the following identity for every smooth function $w \colon \mathbb{R}^n \to \mathbb{R}$ and every $j=1,\dotsc,m$ 
 \begin{equation}
 \label{e:appendix4000}
 \partial_j (w \circ\Psi)(z) = \sum_{\ell=1}^n \partial_\ell w(\Psi(z)) \partial_j \Psi(z) \cdot \tilde{e}_\ell\,.
 \end{equation}
Now let $v \colon \mathbb{R}^m \to \mathbb{R}^{m*}$ be defined as $v(z) := a(\Psi(z))$. By~\eqref{e:appendix3000}, by~\eqref{e:appendix4000}, and by definition of~$\Gamma^{k}_{ij}$, for every $z \in \psi^{-1} ({\rm B}_{\delta} (q))$ and $i,j=1,\dotsc,m$ we compute 
\begin{align}
\label{e:appendix10099}
\langle \partial_j v(z),e_i\rangle & = \sum_{k=1}^n\partial_j (a_k \circ \Psi)(z)\langle g^k(\Psi(z)) ,e_i\rangle +\sum_{k=1}^n a_k(\Psi(z))\langle \partial_j(g^k \circ \Psi)(z),e_i \rangle\\
&= \sum_{k,\ell=1}^n\partial_\ell a_k(\Psi(z)) (\partial_j \Psi(z) \cdot \tilde{e}_\ell)  \langle g^k(\Psi(z)) ,e_i\rangle \nonumber
\\
&
\qquad  +\sum_{k,\ell=1}^n  a_k(\Psi(z))\langle \partial_\ell g^k(\Psi(z)),e_i \rangle (\partial_j \Psi(z) \cdot \tilde{e}_\ell) \nonumber \\
&= \sum_{k,\ell=1}^n\partial_\ell a_k(\Psi(z)) \langle g^\ell(\Psi(z)), e_j \rangle \langle g^k(\Psi(z)) ,e_i\rangle \nonumber
\\
&
\qquad +\sum_{k,\ell=1}^n a_k(\Psi(z))\langle \partial_\ell g^k(\Psi(z)),e_i \rangle \langle g^\ell(\Psi(z)), e_j \rangle \nonumber \\
&= \sum_{k,\ell=1}^n\partial_\ell a_k(\Psi(z)) \langle g^\ell(\Psi(z)), e_j \rangle \langle g^k(\Psi(z)) ,e_i\rangle \nonumber
\\
&
\qquad -\sum_{k,\ell,p=1}^n  a_k(\Psi(z)) \Gamma^k_{\ell p}(\Psi(z)) \langle g^p(\Psi(z)),e_i\rangle  \langle g^\ell(\Psi(z)),e_j \rangle \nonumber \\
&= \sum_{k,\ell=1}^n\big[\partial_\ell a_k(\Psi(z)) - \sum_{p=1}^n a_p(\Psi(z))\Gamma^p_{\ell k}(\Psi(z))\big]\langle g^\ell(\Psi(z)), e_j \rangle \langle g^k(\Psi(z)) ,e_i\rangle. \nonumber
\end{align}
Hence, for every $h \in \mathrm{B}_\delta(q)$ and $i,j=1,\dotsc,m$ we can compactly write
\begin{equation}
    \label{e:appendix5000}
    \langle \partial_j v(\varphi(h)),e_i\rangle= \sum_{k,\ell=1}^n\langle \nabla_{\ell} a(h),g_k(h)  \rangle\langle g^\ell(h), e_j \rangle \langle g^k(h) ,e_i\rangle.
\end{equation}

Therefore, letting $G \colon \mathrm{B}_\delta(q) \to \mathbb{M}^{n \times m}$ be defined as $G_{ij}(h):= \langle g^i(h), e_j \rangle$ for $i=1,\dots,n$ and $j=1,\dotsc,m$, in view of~\eqref{e:appendix1.1} and of~\eqref{e:appendix5000}, the gradient~$\tilde{\nabla} (v)$ and the curvilinear gradient $\nabla(a)$ are related by
\begin{equation}
    \label{e:appendix3.1}
    \tilde{\nabla}(v)(\varphi(h)) =G^\top(h) \, \nabla(a)(h) \, G(h) \qquad \text{for every }h \in \mathrm{B}_\delta(q)\,.
\end{equation}
Since the (euclidean) symmetric gradient $\tilde{e}(v)$ of~$v$ and the curvilinear symmetric gradient $e(a)$ of $a$ can be compactly written as
\[
\tilde{e}(v) := \frac{\tilde{\nabla}(v)+\tilde{\nabla}(v)^{\top}}{2} \qquad \text{and} \qquad e(a) := \frac{\nabla(a)+\nabla(a)^{\top}}{2},
\]
we may similarly relate $\tilde{e}(v)$ with the curvilinear symmetric gradient~$e(a)$ of~$a$ by
 \begin{equation}
     \label{e:appendix100999}
     \tilde{e}(v)(\varphi(h))  = G^{\top}(h) \,  e(a)(h)  \, G(h) \qquad  \text{for every }h \in \Omega\,.
 \end{equation}
Furthermore, it follows from~\eqref{e:appendix1000} that~$G(h)$ admits a right-inverse~$G^{-1}(h) \in \mathbb{M}^{n \times m}$ such that $G^{-1}_{ij}(h)=g_{j}(h) \cdot e_i$ for $i=1,\dotsc,n$ and $j=1,\dotsc,m$. Hence,~\eqref{e:appendix100999} can be inverted as
\begin{equation}
     \label{e:appendix100.1}
     G^{-\top}(h) \, \tilde{e}(v)(\varphi(h)) \, G^{-1}(h)  =  e(a)(h) \qquad \text{for every }h \in \mathrm{B}_\delta(q)\,.
 \end{equation}

 Now we let the field $F$ be defined as
\begin{equation}
\label{e:f=christoffel}
 F(h,v):= -\bigg(\sum_{i,j=1}^n\Gamma^1_{ij}(h)v_iv_j,\dotsc,\sum_{i,j=1}^n\Gamma^n_{ij}(h)v_iv_j \bigg), \qquad (h,v) \in \Omega \times \mathbb{R}^n.
\end{equation}
Choosing the function $g \colon \Omega \times \mathbb{R}^n \to \mathbb{R}^n$ to be the projection onto the second component, a straightforward computation shows that \eqref{e:operator1} is satisfied with $\mathcal{E}=e$ and with function $c_{\mathcal{E}}(\cdot) = |\cdot|^2$. With the above geometrical preliminaries in mind we want to verify that \eqref{hp:F2} is satisfied whenever $F$ has the form in \eqref{e:f=christoffel}. In particular, thanks to the weak-Poincar\'e's inequality, condition (2) of Corollary \ref{c:int2} will be satisfied. For this purpose we fix $x \in \mathrm{B}_\delta(q)$ and set $\varphi_{r,x}(h):= \varphi(x+rh)$ for $h \in B_{1}(0)$ and for every but sufficiently small $r>0$. Then, we have that
\begin{equation}\label{e:appendix7}
g_{i,r,x}(h) := r g_i(x+rh) = \partial_i \varphi_{r,x}(h) \qquad \text{for $i=1,\dotsc,n$ and  $h \in \mathrm{B}_1(0)$,}
\end{equation}
Similarly to~\eqref{e:A6} we set
\begin{equation}
\label{e:appendix6}
\Gamma^k_{ij,r,x}(h):= -\langle \partial_i g_{r,x}^k(h), g_{j,r,x}(h)\rangle \qquad \text{for $i,j,k=1,\dotsc,n$ and $h \in \mathrm{B}_1(0)$,}
\end{equation}
so that the function $F_{r,x}(h,v)= rF(x+rh,v)$ satisfies for $h \in \mathrm{B}_1(0)$ 
\begin{equation*}
 F_{r,x}(h,v)= -\Big( \sum_{i,j = 1}^{n} \Gamma^1_{ij,r,x}(h)v_iv_j, \dotsc, \sum_{i,j= 1}^{n} \Gamma^n_{ij,r,x}(h)v_iv_j\Big)\,.
\end{equation*}

Notice that the computation presented in~\eqref{e:appendix1000}--\eqref{e:appendix100.1} can be repeated with~$\varphi$ replaced by~$\varphi_{r, x}$. Thus, there exists $G_{r,x}  \colon {\rm B}_1(0) \to \mathbb{M}^{n \times m}$ satisfying relation \eqref{e:appendix3.1} with $\{g_1,\dotsc,g_n\}$ replaced by $\{g_{1,r,x},\dotsc,g_{n,r,x}\}$. Since by definition we have that 
\begin{align}
\label{e:appendix23.1}
& r(G_{r,x})_{ij} \to  \langle g^i(x), e_j \rangle  \qquad   r^{-1}(G^{-1}_{r,x})_{ji}\to  \langle g_i(x), e_j \rangle  \ \ \text{ in } C^\infty({\rm B}_1(0))\,, \\
\label{e:appendix23}
& \ \ \ \ \ \ \ \ \ \ \ \ \ \ \ \ \ \ \ \ \ \ \ \ \  \Gamma^k_{ij,r,x} \to 0 \qquad \text{in } C^\infty({\rm B}_1(0))\,,
\end{align}
as $r\searrow 0$, we find a radius $r_x>0$ such that for $0< r \leq r_x$
\begin{align}
    \label{e:appendix16}
   &  \|\nabla G^{-\top}_{r,x}\|_{L^{\infty} (B_{1}(0))} \| G^{\top}_{r,x} \|_{L^{\infty} (B_{1}(0))} 
    \\
    &
    \qquad \qquad + \| G_{r,x} \|_{L^{\infty} (B_{1}(0))} \|\nabla G^{-1}_{r,x}\|_{L^{\infty} (B_{1}(0))} \leq \frac{1}{64(n^2+n)}\,. \nonumber
\end{align}

Let $z \in Q(n)$ and $r\in (0,r_x]$. To each $w \in \mathbb{R}^{(n+1)\times n} $ we associate $\underline{w}_{r, z} \in \mathbb{R}^{(n+1)\times m}$ by
\begin{equation}
\label{e:appendix15}
 \underline{w}^i_{r, z} := \sum_{k=1}^n w^i_k \, g_{r,x}^k(z+\tilde{e}_i)  \qquad \text{for $i=0,\dotsc,n$.}
\end{equation}
Let $v_r \colon \mathbb{R}^m \to \mathbb{R}^{m*}$ be any linear-affine interpolation of the values $\{\underline{w}^0_{r, z},\dotsc,\underline{w}^n_{r, z} \}$, namely, a function fulfilling
\begin{equation}
    \label{e:appendix14}
   \tilde{\nabla}v_r(\zeta) \text{ is constant in $\zeta \in \mathbb{R}^m$} \ \ \text{and} \ \  v_r(\varphi_{r, x } ( z + \tilde{e}_i  ) ) = \underline{w}^i_{r, z} \qquad \text{for $i=0,\dotsc,n$.}
\end{equation}
Let us define $a_r \colon \mathrm{B}_1(0) \to \mathbb{R}^{m*}$ as $a_r(h):= v_r(\varphi_{x,r}(h))$. We infer from~\eqref{e:appendix15}--\eqref{e:appendix14} that the curvilinear coordinates~$\{(a_{r})_{j}\}_{j =1}^{n}$ of~$a_r$ satisfy
\begin{equation}
    \label{e:appendix14.1}
    (a_r)_j( z + \tilde{e}_i) = w^i_j \qquad \text{for $i=0,\dotsc,n$ and  $j=1,\dotsc,n$.}
\end{equation}
In view of \eqref{e:appendix14.1}, to show the validity of \eqref{hp:F2} we are left to prove that $a_r$ fulfills \eqref{e:rip5}.

\begin{theorem}
\label{t:RI-manifold}
The function $a_{r}$ satisfies~\eqref{e:rip5} of the Rigid interpolation property~\eqref{hp:F2}.
\end{theorem} 

\begin{proof}
We notice that the function~$a_{r}$ appearing in~\eqref{e:rip4}--\eqref{e:rip5} is the vector $((a_{r})_{1}, \ldots, (a_{r})_{n})$ of curvilinear coordinates of~$a_{r}$. However, below we will continue considering the map $a_{r} \colon {\rm B}_{1}(0) \to \R^{m*}$, in order to use relations~\eqref{e:appendix3.1}--\eqref{e:appendix100.1}. In particular, for $x$ fixed let us start with~$r_{x}>0$ such that~\eqref{e:appendix16} holds.

We claim that, up to redefine the value of $r_x>0$, we have for every $(z,w,r) \in Q(n) \times\mathbb{R}^{(n+1)\times n} \times (0,r_x]$ 
\begin{equation}
    \label{e:appendix21}
    |e(a_r)(h)-e(a_r)(h')| \leq \frac{\|e(a_r)\|_{L^\infty(\mathrm{B}_1(0))}}{  16  (n+1)n} \qquad \text{for } h,h' \in \mathrm{B}_1(0)\,.
\end{equation}
Let us fix $h, h' \in {\rm B}_{1}(0)$ and let $\sigma_{h, h'} (s) :=  sh+(1-s)h'$ for $s \in [0,1]$. By~\eqref{e:appendix100999}--\eqref{e:appendix100.1} and by the fact that~$v_r$ is affine, we estimate
\begin{align*}
    | e(a_r)(h) & - e(a_r)(h') |  \leq \int_0^1 \bigg|\frac{\di}{\di s} e(a_r)(\sigma_{h,h'} (s) )\bigg| \, \di s 
    \\
    &
    = \int_0^1 \bigg|\frac{\di }{\di s} G_{r,x}^{-\top}(\sigma_{h,h'} (s))  \tilde{e}(v_r)(\varphi_{r, x} ( \sigma_{h,h'} (s)))  G_{r,x}^{-1} ( \sigma_{h,h'} (s) )\bigg| \, \di s \\
    &\leq \int_0^1 | \nabla G_{r,x}^{-\top} ( \sigma_{h,h'} (s) ) \cdot (h-h')  \tilde{e}(v_r) ( \varphi_{r, x} (\sigma_{h,h'} (s)) )  G_{r,x}^{-1} ( \sigma_{h,h'} (s) ) | \, \di s \\
    &\qquad + \int_0^1 | G_{r,x}^{-\top} ( \sigma_{h,h'} (s) ) \tilde{e}(v_r)(\varphi_{r, x} (\sigma_{h,h'} (s)) ) \nabla G_{r,x}^{-1}(\sigma_{h,h'} (s) )\cdot (h-h') | \, \di s \\
    &= \int_0^1 | \nabla G_{r,x}^{-\top}(\sigma_{h,h'} (s)) \cdot (h-h') G_{r,x}^{\top} ( \sigma_{h,h'} (s) ) e(a_r)(\sigma_{h,h'} (s)) | \, \di s \\
    &\qquad + \int_0^1 | e ( a_r ) ( \sigma_{h,h'} (s) )G_{r,x}(\sigma_{h,h'} (s)) \nabla G_{r,x}^{-1} ( \sigma_{h,h'} (s) )\cdot (h-h') | \, \di s \\
    &\leq   4  \| e(a_r) \|_{L^\infty(\mathrm{B}_1(0))} \big( \|\nabla G_{r,x}^{-\top}\|_{L^\infty(\mathrm{B}_1(0))} \|G_{r,x}^{\top} \|_{L^\infty(\mathrm{B}_1(0))}
    \\
    &
    \qquad + \|G_{r,x}\|_{L^\infty(\mathrm{B}_1(0))}\|\nabla G_{r,x}^{-1}\|_{L^\infty(\mathrm{B}_1(0))} \big)\,.
    \end{align*}
    The above inequality, together with~\eqref{e:appendix16}, implies~\eqref{e:appendix21}.
 We can thus infer from~\eqref{e:appendix21} the validity of the following estimate
\begin{equation}
\label{e:appendix22}
     \|e(a_{r})\|_{L^\infty(\mathrm{B}_1(0))} \leq |e(a_{r} (h))| + \frac{\|e(a_{r})\|_{L^\infty(\mathrm{B}_1(0))}}{ 16  (n+1)n} \qquad \text{for every }h \in \mathrm{B}_1(0)\,.
    \end{equation}
    
We further claim that, up to redefine $r_x>0$, the following inequality holds true for $(z,w,r) \in Q(n) \times\mathbb{R}^{(n+1)\times n} \times (0,r_x]$
\begin{equation}
    \label{e:appendix19}
    |e(a_{r} ( z ) ) | \leq c(n) \Big(  E_{r,z}(w) + \frac{ 5  }{8}  \|e(a_{r})\|_{L^\infty(\mathrm{B}_1(0))} \Big)\,,
\end{equation}
for a dimensional constant~$c(n)>0$. Indeed, given $i,j=0,1,\dotsc, n$, let $\ell_{z,r,ij} \colon [0,t_{ij}] \to \mathbb{R}^n$ and $\xi_{r,ij} \colon \mathrm{B}_1(0) \to \mathbb{R}^n$ be defined in Section~\ref{s:curvilinear} (see \eqref{e:poincare15000}). By~\eqref{e:appendix21} we infer that for every $0 <r\leq r_x$
\begin{align}
\label{e:appendix-120}
w^j \cdot \xi_{r,ji}(z) - w^i \cdot \xi_{r,ij}(z) & = \int_0^{t_{ij}} \frac{\di}{\di t} \big( a_{r} (\ell_{z,r,ij}(t)) \cdot \dot{\ell}_{z,r,ij}(t) \big) \, \di t
\\
&
=\int_0^{t_{ij}} e(a_{r} (\ell_{z,r,ij}(t))) \dot{\ell}_{z,r,ij}(t) \cdot \dot{\ell}_{z,r,ij}(t) \, \di t \nonumber
\\
&
\geq t_{ij} e(a_{r} (z+e_i))\xi_{r,ij}(z) \cdot \xi_{r,ij}(z) \nonumber
\\
& \qquad -  \|e(a_{r})\|_{L^\infty(\mathrm{B}_1(0))}\bigg(\int_0^{t_{ij}} |\xi_{r,ij}(z) -\dot{\ell}_{z,r,ij}(t)| \, \di t \nonumber
\\
&
 \qquad +  \int_{0}^{t_{ij}} | \dot{\ell}_{z, r, ij} (t) | \, | \xi_{r,ij}(z) -\dot{\ell}_{z,r,ij}(t)| \, \di t   +\frac{t_{ij}}{16(n+1)n} \bigg) \nonumber 
 \end{align}

Setting $\tilde{e}_{ij}:= (\tilde{e}_j-\tilde{e}_i)/|\tilde{e}_j-\tilde{e}_i|$ we have, by~\eqref{e:appendix23}, that 
\[
\begin{split}
&t_{0j} \to 1  \ \ \text{ for } 1\leq  j \leq n\\
&t_{ij} \to \sqrt{2} \ \ \text{ for } 1\leq i<  j \leq n\\
&\dot{\ell}_{z,r,ij}(t) \to \tilde{e}_{ij} \ \ \text{ for } 0\leq i < j \leq n 
\end{split}
\]
as $r \searrow 0$, uniformly w.r.t.~$t \in [0, t_{ij}]$ and $z \in Q(n)$. Hence, up to further redefine $r_x>0$, we infer from~\eqref{e:appendix-120} that
\begin{equation}
    \label{e:appendix24}
     |e(a_{r} (z+e_i)) \xi_{r,ij}(z) \cdot \xi_{r,ij}(z)| \leq  |w^j \cdot \xi_{r,ji}(z)-w^i \cdot \xi_{r,ij}(z)| + \frac{\|e(a_{r} ) \|_{L^\infty(\mathrm{B}_1(0))}}{2(n+1)n},
\end{equation}
for every $(z,w,r) \in Q(n) \times\mathbb{R}^{(n+1)\times n} \times (0,r_x]$ and every $0 \leq i < j \leq n$. Thanks to~\eqref{e:appendix21} and~\eqref{e:appendix24}, we may further estimate for $(z,w,r) \in Q(n) \times\mathbb{R}^{(n+1)\times n} \times (0,r_x]$,
\begin{align*}
    &|e(a_{r}(z))\tilde{e}_{ij} \cdot \tilde{e}_{ij}| \leq |e(a_{r} ( z + e_i) ) \xi_{r,ij} (z) \cdot \xi_{r,ij}(z)| \\
    \vphantom{\frac12}&\quad  + | [ e ( a_{r} ( z + e_i) ) - e (a_{r} ( z ) ) ] \xi_{r,ij} (z) \cdot \xi_{r,ij}(z) |
    + | e ( a_{r} ( z ) ) [ \xi_{r,ij}(z) \cdot \xi_{r,ij}(z) - \tilde{e}_{ij} \cdot \tilde{e}_{ij} ]|\\
    &  \leq |w^j \cdot \xi_{r,ji}(z)-w^i \cdot \xi_{r,ij}(z)| + \frac{\|e(a)\|_{L^\infty(\mathrm{B}_1(0))}}{2(n+1)n} + \frac{\|e(a)\|_{L^\infty(\mathrm{B}_1(0))}}{16(n+1)n}
    \\ 
    \vphantom{\frac12}&\quad +2\|e(a)\|_{L^\infty(\mathrm{B}_1(0))}|\xi_{r,ij}(z)-\tilde{e}_{ij}|,
\end{align*}
which together with convergence \eqref{e:appendix23} gives, up to possibly redefine once again $r_x>0$, that
\begin{equation}
    \label{e:appendix25}
    \begin{split}
    |e(a_{r} (z))\tilde{e}_{ij} \cdot \tilde{e}_{ij}| &\leq |w^j \cdot \xi_{r,ji}(z)-w^i \cdot \xi_{r,ij}(z)|
    + \frac{5\|e(a_{r})\|_{L^\infty(\mathrm{B}_1(0))}}{8(n+1)n}\,, 
    \end{split}
\end{equation}
whenever $(z,w,r) \in Q(n) \times\mathbb{R}^{(n+1)\times n} \times (0,r_x]$ and $0 \leq i < j \leq n$. Noticing that
\[
|e(a_{r}(z))|  \leq c(n) \sum_{0 \leq i < j \leq n} |e(a_{r}(z))\tilde{e}_{ij} \cdot \tilde{e}_{ij}|\,,
\]
for some dimensional constant $c(n)>0$, combining \eqref{e:appendix24} with \eqref{e:appendix25} we infer the validity of \eqref{e:appendix19}. Using \eqref{e:appendix19} in \eqref{e:appendix22} we obtain for $(z,w,r) \in Q(n) \times\mathbb{R}^{(n+1)\times n} \times (0,r_x]$
\begin{equation*}
\|e(a_{r})\|_{L^\infty(\mathrm{B}_1(0))} \leq \frac{32c(n)}{11} \, E_{r,z}(w)
\end{equation*}
 which is exactly~\eqref{e:rip5}.
\end{proof}


\appendix

\section{Proofs of Propositions \ref{c:relje} and \ref{p:r=rxi}}
\label{appendix}

\begin{proof}[Proof of Proposition~\ref{c:relje}]
 For $\mathcal{H}^{n-1}$-a.e.~$x \in R$ we set  
\[ 
R^\pm(x) := \{z\in \mathbb{R}^n : \, \pm (z-x) \cdot \nu_{R}(x) > 0 \}\,.
\]
We divide the proof into two steps.

\underline{\em Step 1}. Let us first suppose that $R$ is a $C^1$-manifold of dimension $n-1$. In order to ease the notaion we let  $v:=\tau(u_\xi)$. We can rewrite~\eqref{e:corollary-relje} in integral terms as
\begin{equation}
    \label{e:relje4}
    \lim_{r \searrow 0}  \mint_{{\rm B}_r(x) \cap R^\pm(x)}  |v(z)-\arctan(a)| \, \di z =0 \ \ \text{ and }\ \  a=\aplim_{s \to t^{\pm\sigma(x)}}  \hat{u}^{\xi}_{y}(t) \,, 
\end{equation}
whenever $x=\varphi(y+t\xi)$. If we perform the change of variable induced by $\varphi$, the right-hand side of \eqref{e:relje4} becomes
\begin{equation}
    \label{e:relje2}
    \lim_{r \searrow 0} D(x,r) \mint_{\varphi^{-1}({\rm B}_r(x) \cap R^\pm(x))}    |v(\varphi(z)) -\arctan(a)|\,|\text{J}\varphi(z)| \, \di z = 0\,,
\end{equation}
where
\[
D(x,r):=\frac{\mathcal{L}^n(\varphi^{-1}({\rm B}_r(x) \cap R^\pm(x)))}{\mathcal{L}^n({\rm B}_r(x) \cap R^\pm(x))}.
\]
Since Area Formula gives
\[ 
\lim_{r \searrow 0} \, D(x,r) =  J \varphi^{-1}(x)\,,
\]
by virtue of~\eqref{e:relje2} we can reduce ourselves to prove that
\begin{equation}
    \label{e:relje3}
    \lim_{r \searrow 0}  \mint_{\varphi^{-1}({\rm B}_r(x) \cap R^\pm(x))} |v(\varphi(z))-\arctan(a)| \, \di z=0 \ \ \text{ and }\ \  a=\aplim_{s \to t^{\pm\sigma(x)}}\hat{u}_{y}(s)\,.
\end{equation}
It is not difficult to verify that for every $x'$ in the domain of $\varphi$ we have 
\[
[D\varphi(x')]^{\top} \nu_{R} (\varphi(x')) \in \text{Tan}^\bot(\varphi^{-1}(R),x') \,.
\]
Using this last information together with the following identity
\begin{equation*}
[D\varphi(x')]^{\top} \nu_{R} (\varphi(x')) \cdot \xi =  \nu_{R} (\varphi(x')) \cdot [D\varphi(x')]\xi = \nu_{R} (\varphi(x')) \cdot \xi_{\varphi} (\varphi(x')) \,,
\end{equation*}
we infer that the set $\varphi^{-1}(\{x \in R \cap U : \, \nu_{R}(x) \cdot \xi_{\varphi} (x) \neq 0\})$ has no vertical part with respect to the direction~$\xi$. Now let us set
\[
R_0 :=\{x \in R \cap U :  \, \nu_{R} (x) \cdot \xi_\varphi (x) \neq 0\}\,.
\] 
Applying the Implicit Function Theorem, we find a covering if $\varphi^{-1}(R_0)$ made of at most countably many open cylinders $C_i$ such that $C_i \cap \varphi^{-1}(R_0)$ is the graph of a $C^1$-regular functions with values in $\mathbb{R}$ and defined on $\pi_\xi(C_i)$ (being $\pi_\xi \colon \mathbb{R}^n \to \xi^\bot$ the orthogonal projection). Since for every $i \in \mathbb{N}$ we have from hypothesis $v(\varphi(\cdot)) \in L^1(C_i)$ and $D_\xi v(\varphi(\cdot)) \in \mathcal{M}_b(C_i)$, we are in position to apply \cite[Theorem 5.1]{dal} to find $v^\pm_R \in L^1(\varphi^{-1}(R_0) \cap C_i;\mathcal{H}^{n-1})$ satisfying
\begin{align}
    \label{e:relje5}
    & v^\pm_R(x') = \aplim_{s \to t'^{\pm}} \, v( \varphi(y'+s\xi)) \,,
    \\
    \label{e:relje5.1.1}
    & \lim_{r \searrow 0} \, \frac{1}{r^n}\int_{{\rm B}_r(x') \cap C_i^\pm} |v(\varphi(z)) - v^\pm_R(x')| \, \di z =0 \, ,
\end{align}
for $\mathcal{H}^{n-1}$-a.e.~$y' \in \pi_\xi(C_i)$ and every $t' \in R^\xi_{y'}$, whenever $x'=y'+t'\xi$ and  $C_i^{+}$ and~$C^{-}_{i}$ are the supgraph and subgraph of $\varphi^{-1}(R_0) \cap C_i$, respectively. Using identity~\eqref{e:sliceide} and defining $u^\pm_R \colon R \to \mathbb{R}$ as $u^\pm_R(x):= v^\pm_R(\varphi^{-1}(x))$, we can make use of \eqref{e:relje5}--\eqref{e:relje5.1.1} to infer
\[
\arctan^{-1}(u^\pm_R(x))= \aplim_{s \to t^{' \pm \sigma(x) }}\hat{u}^{\xi}_{y'}(s), \ \ \text{for $\mathcal{H}^{n-1}$-a.e. $y' \in \pi_\xi(C_i)$ and every $t' \in R^\xi_y$},
\]
whenever one between the approximate limits in \eqref{e:corollary-relje} exists at $x= \varphi(y'+t'\xi)$. 

To conclude, we need to prove that, defining $a:=\arctan^{-1}(u^\pm_R(x))$, the first equality in~\eqref{e:relje3} is satisfied for $\mathcal{H}^{n-1}$-a.e.~$y' \in \pi_\xi(C_i)$ and every $t' \in R^\xi_{y'}$ whenever $x \in \varphi(y'+t'\xi)$. Letting $x=\varphi(x')$ and $x'=y'+t'\xi$, we find a constant~$c>0$ depending only on~$x'$ such that
\begin{align}
\label{e:phi-1}
\mint_{{}_{\scriptstyle \varphi^{-1}({\rm B}_r(x) \cap R^\pm(x))}} \!\!\!\!\!\!\!\! |v(\varphi(z))-u^\pm_{R}(x)| \, \di z &=
\mint_{{}_{\scriptstyle \varphi^{-1}({\rm B}_r(x) \cap R^\pm(x))}} \!\!\!\!\!\!\!\! |v(\varphi(z))-v^\pm_{R}(x')| \, \di z \\ &= \mint_{{}_{\scriptstyle \varphi^{-1}({\rm B}_r(x)) \cap \varphi^{-1}(R^\pm(x))}} \!\!\!\!\!\!\!\! |v(\varphi(z))-v^\pm_{R}(x')| \, \di z\nonumber\\
&\leq \frac{c}{r^n} \int_{{}_{\scriptstyle {\rm B}_{cr}(x') \cap C_i^{\pm}}} \!\!\!\!\!\!\!\! |v(\varphi(z))-v^\pm_{R}(x')| \, \di z \nonumber\\
&\qquad + c\, \frac{\mathcal{L}^n[(\varphi^{-1}(R^\pm(x)) \setminus C_i^{\pm})\cap {\rm B}_{cr}(x')]}{r^n}\,.\nonumber
\end{align}
By exploiting that~$R$ is a $C^{1}$-manifold and that in particular $\text{Tan}(R,x)$ exists for every $x \in R$, it is not difficult to verify that 
\begin{equation}
\label{e:Rto0}
\lim_{r \to 0^{+}} \, \frac{\mathcal{L}^n[(\varphi^{-1}(R^\pm(x)) \setminus C_i^{\pm})\cap {\rm B}_{cr}(x')]}{r^n} = 0 \,.
\end{equation}
Using equality \eqref{e:Rto0} together with \eqref{e:relje5.1.1} in \eqref{e:phi-1} yields immediately the desired conclusion. This concludes the proof in the case of a $C^1$-regular manifold.

\underline{\em Step 2}. Let us now consider the of a countably $(n-1)$-rectifiable set $R$. By~\cite[Theorem 3.2.29]{fed} we can choose countably many $C^1$-regular manifolds~$\{R_{i}\}_{i \in \mathbb{N}}$ of dimension $n-1$ and such that 
\begin{equation}
\label{e:relje99999}
\mathcal{H}^{n-1} \Big( R \setminus \bigcup_{i \in \mathbb{N}} R_i \Big) = 0\,.
\end{equation}
Hence, for every $i=1,2,\dotsc$ we deduce from Step~1 the validity of \eqref{e:corollary-relje} with $R$ replaced by $R_i$. Finally, we infer \eqref{e:corollary-relje} since \eqref{e:relje99999} together with the lipschitzianity of the map $P$ imply $\mathcal{H}^{n-1}(P(R \setminus \bigcup_i R_i))=0$.
\end{proof}

\begin{proof}[Proof of Proposition \ref{p:r=rxi}.]
For every $x \in \Om$, let us define $  \psi_{x} \EEE \colon \mathbb{S}^{n-1} \to \mathbb{S}^{n-1}$ as $  \psi_{x} \EEE (\xi) :=  \frac{\xi_{\varphi}(x)}{| \xi_{\varphi} (x) |} $ for every $\xi \in \mathbb{S}^{n-1}$. We claim that for every $x \in \Om$ and every Borel subset~$B$ of~$\mathbb{S}^{n-1}$ it holds
\begin{equation}
\label{e:claimg}
\mathcal{H}^{n-1} (B) =0 \ \Longrightarrow \ \mathcal{H}^{n-1} (  \psi_{x} \EEE^{-1}(B)) = 0\,.
\end{equation}
 Indeed, by the Area formula (see, e.g.,~\cite[Section~2.10]{afp}) we have that for every $B \in   \mathcal{B}( \mathbb{S}^{n-1}) \EEE $,
\[
\int_{g^{-1}(B)} |\text{J}_{\tau}   \psi_{x} \EEE (\xi)|\, \di \mathcal{H}^{n-1}(\xi) = \int_{B} \mathcal{H}^0(\{  \psi_{x} \EEE^{-1}(\eta)\})\, \di \mathcal{H}^{n-1}(\eta)\,,
\]
 where~$\text{J}_{\tau}  \psi_{x} \EEE$ denotes the tangential Jacobian of~$  \psi_{x} \EEE$. Condition   (4) \EEE of Definition~\ref{d:CP}
implies that
\begin{equation}
\label{e:limg}
\mathcal{H}^{n-1}(  \psi_{x} \EEE^{-1}(B)) = \lim_{t \searrow 0} \mathcal{H}^{n-1} \big(   \psi_{x} \EEE^{-1}(B) \cap \{|\text{J}_\tau g|>t\} \big)\,.
\end{equation}
Hence, for every $t>0$ it holds
\begin{align}
\label{e:hg-1}
\mathcal{H}^{n-1} \big(   \psi_{x} \EEE^{-1}(B) \cap \{|\text{J}_\tau   \psi_{x} \EEE | > t \} \big) &\leq \frac{1}{t}\int_{  \psi_{x} \EEE^{-1}(B)} |\text{J}_{\tau}   \psi_{x} \EEE (\xi)|\, \di \mathcal{H}^{n-1}(\xi)\\
&=  \frac{1}{t} \int_{B} \mathcal{H}^0( \{   \psi_{x} \EEE^{-1}(\eta)\})\, \di \mathcal{H}^{n-1}(\eta)\,. \nonumber
\end{align}
From~\eqref{e:limg} and~\eqref{e:hg-1} we deduce~\eqref{e:claimg}. 

We notice that the set
\[
N := \{(x,\xi) \in R \times \mathbb{S}^{n-1} : \, \nu_{R}(x) \cdot \xi_{\varphi}(x)=0  \}
\]
is Borel (whenever $\nu_{R} \colon R \to \mathbb{S}^{n-1}$ is Borel). Thus, by Fubini's theorem and by~\eqref{e:claimg} we get that
\[
\begin{split}
\big( \mathcal{H}^{n-1} \restr R \otimes \mathcal{H}^{n-1} \restr \mathbb{S}^{n-1}\big) (N)&=\int_{R } \mathcal{H}^{n-1} \big(\{\xi \in \mathbb{S}^{n-1} : \, \nu_{R}(x) \cdot \xi_{\varphi}(x)=0\} \big)\, \di \mathcal{H}^{n-1}(x)\\
&=\int_{R } \mathcal{H}^{n-1} \big(   \psi_{x} \EEE^{-1} ( \nu_{R}(x)^\bot \cap \mathbb{S}^{n-1}) \big)\, \di \mathcal{H}^{n-1}(x)=0\,.
\end{split}
\]
Therefore, again by Fubini's theorem we infer that
\[
\begin{split}
\int_{\mathbb{S}^{n-1}} \mathcal{H}^{n-1}(R\setminus R^\xi)\, \di \mathcal{H}^{n-1}(\xi)&=\int_{\mathbb{S}^{n-1}} \mathcal{H}^{n-1}\big( \{x \in R :\, \nu_{R}(x) \cdot \xi_{\varphi}(x)=0  \} \big)\, \di \mathcal{H}^{n-1}(\xi)\\
&= \big( \mathcal{H}^{n-1} \restr R \otimes \mathcal{H}^{n-1} \restr \mathbb{S}^{n-1} \big) (N) = 0 \,.
\end{split}
\]
This proves~\eqref{e:r=rxi}.
\end{proof}

\begin{proof}[Proof of Proposition \ref{p:prodmeas}.]
 Combining Propositions~\ref{c:relje} and~\ref{p:r=rxi}  it is not difficult to verify the validity of the following property
\begin{equation*}
    \aplim_{\substack{z \to x \\ \pm(z-x) \cdot \nu_{R} (x) >0}} u_\xi(z)  = \aplim_{s \to t^{\pm\sigma(x)}_x}\hat{u}^\xi_{P_\xi(x)}(s) \ \ \text{for $\mathcal{H}^{n-1}$-a.e. $\xi \in \mathbb{S}^{n-1}$, $\mathcal{H}^{n-1}$-a.e. $x \in R$},
\end{equation*}
 whenever at least one between the approximate limits exists. Therefore, equality \eqref{e:nrelje100} can be obtained as a straightforward consequence of Fubini's theorem once that we prove the $(\mathcal{H}^{n-1} \restr R \otimes \mathcal{H}^{n-1} \restr \mathbb{S}^{n-1})$-measurability of~$\Delta$. We claim that the set~$\Delta$ is a Borel subset of~$R \times \mathbb{S}^{n-1}$. By eventually considering the composition $\arctan(u_\xi)$ we may suppose that the $u_\xi$ take values in $[-\pi/2,\pi/2]$ for every $\xi \in \mathbb{S}^{n-1}$.


We introduce the functions 
\begin{align*}
c^+(x,\xi)&:=\limsup_{r \searrow 0} \, \sigma(x)\mint_{0}^{r} \hat{u}^\xi_{P_\xi(x)}(t^\xi_x+\sigma(x) t ) \, \di t\,,\\
c^-(x,\xi)&:=\liminf_{r \searrow 0}\, \sigma(x)\mint_{0}^{r} \hat{u}^\xi_{P_\xi(x)}(t^\xi_x + \sigma(x)t ) \, \di t\,,\\
d^+(x,\xi)&:=\limsup_{r \searrow 0}\, -\sigma(x)\mint_{0}^{r} \hat{u}^\xi_{P_\xi(x)}(t^\xi_x-\sigma(x)t ) \, \di t\,,\\
d^-(x,\xi)&:=\liminf_{r \searrow 0}\, -\sigma(x)\mint_{0}^{r} \hat{u}^\xi_{P_\xi(x)}(t^\xi_x-\sigma(x)t ) \, \di t\,.
\end{align*}
We notice that the Borel measurability of~$c^{\pm}$ and~$d^{\pm}$ in~$R \times \mathbb{S}^{n-1}$ follows by the argument used in the proof of Lemma~\ref{l:meas10000}. We infer that the set
\begin{align*}
\bigg\{(x,\xi) \in R  \times \mathbb{S}^{n-1}: &  \, c^+(x,\xi)=c^-(x,\xi)   \in \bigg(-\frac{\pi}{2},\frac{\pi}{2}\bigg) \EEE \quad \text{and}\quad \\
& \limsup_{r \to 0^+} \, \sigma(x)\mint_{0}^{r} |\hat{u}^\xi_{P_\xi(x)}(t^{\xi}_{x} +\sigma(x)t )-c^+(x,\xi)| \, \di t=0 \bigg\}
\end{align*}
is Borel measurable in~$R \times \mathbb{S}^{n-1}$. We notice that this set coincides with the set 
\begin{align*}
\bigg\{(x,\xi) \in R  \times \mathbb{S}^{n-1}: \,
\aplim_{t \to t^{\sigma(x)}_x}\hat{u}^\xi_{P_\xi(x)}(t) = c^{+}(x, \xi)   \in \bigg(-\frac{\pi}{2},\frac{\pi}{2}\bigg) \EEE\bigg\}.
\end{align*}
For the same reason, the set
\begin{align*}
\bigg\{ (x,\xi) \in R  \times \mathbb{S}^{n-1}: \, \aplim_{t \to t^{-\sigma(x)}_x}\hat{u}^\xi_{P_\xi(x)}(t) =   d^{+} (x, \xi)  \in \bigg(-\frac{\pi}{2},\frac{\pi}{2}\bigg) \EEE \bigg\}
\end{align*}
is Borel measurable. In the same way we obtain that the functions (recall that~$u_{\xi}$ is assumed to be bounded)
\begin{align*}
f^+(x,\xi)&:=\limsup_{r \searrow 0}\mint_{{\rm B}_r(x) \cap R^+(x)} u_\xi(z) \, \di z\,,\\
f^-(x,\xi)&:=\liminf_{r \searrow 0}\mint_{{\rm B}_r(x) \cap R^+(x)} u_\xi(z) \, \di z\,,\\
g^+(x,\xi)&:=\limsup_{r \searrow 0}\mint_{{\rm B}_r(x) \cap R^-(x)} u_\xi(z) \, \di z\,,\\
g^-(x,\xi)&:=\liminf_{r \searrow 0}\mint_{{\rm B}_r(x) \cap R^-(x)} u_\xi(z) \, \di z
\end{align*}
are Borel measurable in $R  \times \mathbb{S}^{n-1}$. Hence, arguing as above it turns out that the sets
\begin{align*}
& \bigg\{ (x,\xi) \in R  \times \mathbb{S}^{n-1}: \,
\aplim_{\substack{z \to x \\ (z-x) \cdot \nu_R(x) >0 }} u_\xi(z)  = f^{+}(x, \xi)   \in \bigg(-\frac{\pi}{2},\frac{\pi}{2}\bigg) \EEE \bigg\},\\
& \bigg\{  (x,\xi) \in R  \times \mathbb{S}^{n-1}: \aplim_{\substack{z \to x \\ (z-x) \cdot \nu_R(x) <0 }} u_\xi(z)  = g^{+}(x, \xi)   \in \bigg(-\frac{\pi}{2},\frac{\pi}{2}\bigg) \EEE \bigg\}
\end{align*}
are Borel measurable. From the Borel measurability of~$c^{\pm}$, $d^{\pm}$, $f^{\pm}$ and~$g^{\pm}$ we infer that~$\Delta$ is Borel measurable in~$R \times \mathbb{S}^{n-1}$. This concludes the proof. 
\end{proof}

\section*{Acknowledgments}
The work of the authors was partially funded by the Austrian Science Fund (\textbf{FWF}) through the project P35359-N. S.A. was also supported by the FWF project ESP-61. E.T. further acknowledges the support of the FWF projects Y1292 and F65. 

\bibliographystyle{siam}
\bibliography{biblio}









\end{document}